\ifdefined\synctex
\fi
\documentclass{amsart}

\usepackage{amsmath}
\usepackage{amssymb}
\usepackage{amsthm}
\usepackage{mathtools}
\usepackage{xcolor}
\usepackage{hyperref}
\usepackage{comment}

\renewcommand{\keywordsname}{Keywords}
\makeatletter
\renewcommand{\@setkeywords}{%
  {\itshape\keywordsname:}\enspace\@keywords\@addpunct.}
\makeatother

\newcommand{\R}{\mathbb{R}}
\newcommand{\C}{\mathbb{C}}
\newcommand{\N}{\mathbb{N}}
\newcommand{\calK}{\mathcal{K}}
\newcommand{\calD}{\mathcal{D}}

\newcommand{\la}{\ensuremath{\langle}}
\newcommand{\ra}{\ensuremath{\rangle}}

\newcommand{\Op}{\ensuremath{\mathrm{Op}}}

\newcommand{\TV}{\ensuremath{\mathrm{TV}}}

\newcommand{\supp}{\ensuremath{\mathrm{supp}}}
\newcommand{\dist}{\ensuremath{\mathrm{dist}}}
\newcommand{\Lip}{\ensuremath{\mathrm{Lip}}}

\numberwithin{equation}{section}
\newtheorem{thm}{Theorem}[section]
\newtheorem{lem}[thm]{Lemma}
\newtheorem{prop}[thm]{Proposition}

\newtheorem{rmk}[thm]{Remark}

\usepackage[letterpaper,top=2cm,bottom=2cm,left=2cm,right=2cm,marginparwidth=1.75cm]{geometry}

\title[Focusing NLS large data scattering]{Scattering for the focusing $H^{1/2}$-critical nonlinear Schr\"odinger equation with large data}

\author{Qiuye Jia}
\address{Australian National University}
\email{Qiuye.Jia@anu.edu.au}


\begin{document}

\begin{abstract}
In this article we prove that the solution to the focusing $H^{1/2}$-critical nonlinear Schr\"odinger equation in dimension $d\geq 5$ scatters outside finitely many arbitrarily small spacetime cones.
We will discuss the relationship between Theorem~\ref{thm:finite-bad-directions} and the resolution of solitons. There are two key ingredients in the proof: the localized below-ground-state scattering in Theorem~\ref{thm:cone-scattering},
and the characterization in Theorem~\ref{thm:measure-convergence} of the asymptotic behaviour of the residual part (i.e., after subtracting the scattering part).

As a byproduct, we also establish an upgrading mechanism: if the solution is below the ground state (resp. small) along a time sequence in a spacetime cone,
then it is also below ground state (resp. small) uniformly for large times in a shrunk spacetime cone.
We expect this to be useful in future works using concentration compactness and our spacetime cone approach.
\end{abstract}

\maketitle

\tableofcontents

\section{Introduction}\label{sec:introduction}

\subsection{The setup and main results}

In this article, we study the forward-global solutions with uniformly bounded $H^1$ norm to the focusing nonlinear Schr\"odinger equation
\begin{align}
\label{eq:NLS}
\begin{split}
 & (i\partial_t+\Delta)u=-|u|^{4/(d-1)}u \qquad\text{for }t\geq 0,\ x\in\R^d, \\
 & u_0=u|_{t=0} \in H^1(\R^d), \quad d\geq5.
 \end{split}
\end{align}
The critical scaling level of this equation is $H^{1/2}(\R^d)$ and the actual scaling critical quantity we will use is mass times energy, as defined below. We adopt the convention $\Delta=\sum_{i=1}^d\partial_{x_i}^2$.
Throughout the paper, we assume the solution is uniformly bounded in $H^1$:
\begin{align}
\label{eq:uniform-H1}
 \sup_{t\geq 0}\|u(t)\|_{H^1}<\infty.
\end{align}

For $f\in H^1(\R^d)$, we define its mass and energy (associated with our particular NLS) by
\begin{align}
\label{eq:def-mass-energy}
\begin{split}
 M(f)& := \int_{\R^d}|f|^2\,dx, \quad P(f):=\Im \int_{\R^d}\overline{f}\nabla f \,dx, \\
 E(f)& := \frac{1}{2} \int_{\R^d}|\nabla f|^2\,dx -\frac{d-1}{2(d+1)} \int_{\R^d} |f|^{2(d+1)/(d-1)}\,dx .
\end{split}
\end{align}

The aim of this paper is to show that with arbitrarily large (but fixed) mass or energy, the solution to \eqref{eq:NLS} still `almost scatters'. By `scatter', as usual, we mean that $u(t)$ converges to a solution to the free linear Schr\"odinger equation.
More concretely, it is known that the weak limit $e^{-it\Delta}u$ as $t\to\infty$ exists and we will denote it by $u_+$, which is the `radiation data'. 
See also Theorem~\ref{thm:compact-attractor} below for how it lies in a decomposition of $u$.
Then
\begin{align}
\label{eq:v,r-def}
  v(t)=e^{it\Delta}u_+, \quad r(t)=u(t)-v(t)
\end{align}
are the free part and the soliton part (or residual part, in the context below) of our solution $u(t)$. Then `$u(t)$ scatters' means $r(t) \to 0$.

However, there is a well-known obstruction for this statement to be true: the `soliton' mentioned above in the name of $r(t)$. Let $Q$ be the unique positive radial decaying solution of (see e.g. \cite{Kwong1989Uniqueness,Weinstein1983Sharp})
\begin{align}
\label{eq:ground-state}
-\Delta Q+Q-|Q|^{4/(d-1)}Q =0,
\end{align}
which is called the `ground state soliton'. Then $e^{it}Q(x)$ solves \eqref{eq:NLS} as well but it does not scatter.
There is a long tradition of avoiding this obstruction and obtaining scattering by imposing a threshold condition stating that the solution is `smaller' than $Q$.
The threshold governing scattering versus soliton behaviour of solutions to \eqref{eq:NLS} is:
\begin{align}
\label{eq:Theta-def}
 \Theta & :=M(Q)E(Q).
\end{align}
Indeed, this serves as a threshold to give scattering versus soliton dichotomy. See Theorem~\ref{thm:blow-thre-sc} and more discussion on the more complicated situation when it is on or above the threshold in Section~\ref{subsec:review-future}. 

The aim of this article is to show that even when the solution is above this threshold in the sense that $M(u)E(u) > M(Q)E(Q)$, we can still say that $u$ scatters except on finitely many directions of velocities of soliton components. To this end, it is more convenient to make statements using spacetime cones.
For $\Omega\subset\R^d$, we call a region like
\begin{equation}
\label{eq:C-Omega-def}
C(\Omega) = \{(t,x)\in\R\times\R^d:t\geq T,\ x/t\in\Omega\}, \quad T \gg 1
\end{equation}
a spacetime cone; thus a cutoff $\chi(x/t)$ localizes to a fixed region in the velocity variable $x/t$.

Our first result, Theorem~\ref{thm:cone-scattering}, is the below ground-state scattering localized to a spacetime cone. 
See Section~\ref{subsec:review-future} for a discussion of the global version of this result proved previously.

\begin{thm}
\label{thm:cone-scattering}
Let $d\geq5$, and let $u\in C(\R_{\geq 0};H^1(\R^d))$ be a forward-global solution of \eqref{eq:NLS} satisfying \eqref{eq:uniform-H1}.
Let $u_+\in H^1(\R^d)$ be the radiation state of $u$. Let $\chi_1,\chi_2\in C_c^\infty(\R^d)$ be real-valued and satisfy
\begin{align}
\label{eq:cutoff-nesting}
 0\leq\chi_1,\chi_2\leq1,\qquad \chi_2&=1\quad\text{on an open neighborhood of }\supp\chi_1.
\end{align}
Fix $T_0\geq1$ and $\delta\in(0,1)$, and assume that $\chi_2(x/t)u(t)$ is below threshold in the sense that for every $t\geq T_0$, we have:
\begin{align}
\label{eq:localized-mass-energy}
 M(\chi_2(x/t)u(t))E(\chi_2(x/t)u(t)) &\leq(1-\delta)M(Q)E(Q), \\
\label{eq:localized-mass-gradient}
 \|\chi_2(x/t)u(t)\|_{L^2}\|\nabla(\chi_2(x/t)u(t))\|_{L^2} &<\|Q\|_{L^2}\|\nabla Q\|_{L^2}.
\end{align}
Then it scatters in a slightly shrunk cone:
\begin{align}
\label{eq:cone-scattering}
\lim_{t\to\infty}\|\chi_1(x/t)(u(t)-e^{it\Delta}u_+)\|_{H^1}&=0.
\end{align}
\end{thm}

With $r(t)$ as in \eqref{eq:v,r-def}, we will use its mass density and current as densities to define measures to identify the asymptotic behaviour of the solution.
Concretely, we let $X=\R^d\cup\{\infty\}$ be the one-point compactification \footnote{Introducing $X$ is completely for the convenience of writing. One can use $\R^d$ and consider the case with velocity going to infinity by a separate discussion.} of $\R^d$,
which one should think of as parametrizing the asymptotic directions. Then we define $\nu_t$, $J_t$ by:
\begin{align}
\label{eq:nu_t,J_t-def}
 \int_Xh\,d\nu_t :=\int_{\R^d}h(x/t)|r(t,x)|^2\,dx, \qquad \int_Xh\,dJ_t :=\int_{\R^d}h(x/t)\Im(\overline{r}\nabla r)(t,x)\,dx.
\end{align}
Our Theorem~\ref{thm:measure-convergence} below shows that they indeed converge in the weak sense as $t\to\infty$ and is the technically key step for us to identify those non-scattering directions. 
The intuition that we should expect this can be seen from `assuming' a soliton resolution type expansion. See the discussion at the beginning of Section~\ref{sec:atomic}.

\begin{thm}
\label{thm:measure-convergence}
Let $d\geq5$, let $M$, $E$, and $P$ be as in \eqref{eq:def-mass-energy}, let $Q$ be as in \eqref{eq:ground-state},
and let $\Theta$ be as in \eqref{eq:Theta-def}. Let $u\in C(\R_{\geq 0};H^1(\R^d))$ solve
\begin{align*}
 (i\partial_t+\Delta)u=-|u|^{4/(d-1)}u
\end{align*}
and satisfy \eqref{eq:uniform-H1}.
Let $u_+$ be its radiation state, let $r(t)$ be as in \eqref{eq:v,r-def}, put $m_\sharp:=M(u)-M(u_+)$,
and assume $m_\sharp>0$. Define $\nu_t$ and $J_t$ by \eqref{eq:nu_t,J_t-def}. Then there are $L\geq1$, distinct $y_1,\ldots,y_L\in\R^d$, and numbers $a_\ell>0$ such that
\begin{align}
\label{eq:measure-convergence-limit-form}
 (\nu_t,J_t) \rightharpoonup(\mu_\infty,\frac{y}{2}\mu_\infty),  \quad \mu_\infty = \sum_{\ell=1}^La_\ell\delta_{y_\ell}
 \quad \text{ as }t\to\infty.
\end{align}
\end{thm}

Next we discuss the estimates upgrading a condition along a sequence of times to a uniform estimate. The substantive new step is the threshold-free Theorem~\ref{thm:upgrade-I}.
\begin{thm}
\label{thm:upgrade-I}
Let $d\geq5$, and let $u\in C(\R_{\geq 0};H^1(\R^d))$ solve
\begin{align*}
 (i\partial_t+\Delta)u=-|u|^{4/(d-1)}u
\end{align*}
and satisfy \eqref{eq:uniform-H1}.
Let $u_+$ be its radiation state and let $r(t)$ be as in \eqref{eq:v,r-def}. Suppose that $\chi_0,\chi_1\in C_c^\infty(\R^d)$ satisfy $0\leq\chi_i\leq1$ for $i=0,1$ and
\begin{align*}
 \supp(\chi_0)\Subset\{y:\chi_1(y)=1\}.
\end{align*}
Let $T_n\to\infty$, and assume
\begin{align}
\label{eq:cone-upgrade-sequential}
 \|\chi_1(x/T_n)r(T_n)\|_{H^1(\R^d)}\to 0.
\end{align}
Then
\begin{align}
\label{eq:uniform-chi0-r-tend-0}
 \lim_{n\to\infty}\sup_{t\geq T_n} \|\chi_0(x/t)r(t)\|_{H^1(\R^d)}=0.
\end{align}
\end{thm}

The reason that we want to have such a result upgrading the smallness along a time sequence to the uniform smallness is that the decompositions given by Theorem~\ref{thm:compact-attractor} need not carry persistent profile labels. The profiles $w_j$ and the centers $x_{j,n}$ in \eqref{eq:profile-decomposition} are extracted along a subsequence, and Theorem~\ref{thm:compact-attractor} does not prevent infinitely many collisions with various possible outcomes. 
On the other hand, our Theorem~\ref{thm:upgrade-I} does not trace an individual profile through those interactions.
It instead uses the convergence of the measures \eqref{eq:nu_t,J_t-def} given by Theorem~\ref{thm:measure-convergence}.

We will use Theorems~\ref{thm:measure-convergence} and \ref{thm:upgrade-I} to prove Theorem~\ref{thm:upgrade-II} below. See Section~\ref{sec:upgrade-II} for details.
It upgrades a threshold bound available only along a sequence of times to one valid for all large times.
We will use the upgrading mechanism (more concretely, Theorem~\ref{thm:upgrade-II}) in the proof of our main result, i.e., Theorem~\ref{thm:finite-bad-directions}.

\begin{thm}
\label{thm:upgrade-II}
Let $d\geq5$, and let $u\in C(\R_{\geq 0};H^1(\R^d))$ solve
\begin{align*}
 (i\partial_t+\Delta)u=-|u|^{4/(d-1)}u
\end{align*}
satisfying \eqref{eq:uniform-H1}. Let $u_+\in H^1(\R^d)$ be the radiation state of $u$.
Fix $\chi_0,\chi_1,\chi_2\in C_c^\infty(\R^d)$ satisfying $0\leq\chi_i\leq1$ for $i=0,1,2$ and
\begin{align}
\label{eq:condition-chi-012}
 \supp(\chi_0)\Subset\{y:\chi_1(y)=1\}, \qquad \supp(\chi_1)\Subset\{y:\chi_2(y)=1\}.
\end{align}
Suppose that $\delta\in(0,1)$, $t_n\to\infty$, and
\begin{align}
\label{eq:sequential-threshold}
 M(\chi_2(x/t_n)u(t_n))E(\chi_2(x/t_n)u(t_n)) \leq(1-\delta)M(Q)E(Q)
\end{align}
for every $n$. Then there is $T\geq1$ such that, for every $t\geq T$,
\begin{align}
\label{eq:ME-chi0-below-threshold}
 M(\chi_0(x/t)u(t))E(\chi_0(x/t)u(t)) \leq(1-\delta/2)M(Q)E(Q).
\end{align}
After increasing $T$,
\begin{align}
\label{eq:M-kinetic-below-threshold}
 \|\chi_0(x/t)u(t)\|_{L^2} \|\nabla(\chi_0(x/t)u(t))\|_{L^2} <\|Q\|_{L^2}\|\nabla Q\|_{L^2}.
\end{align}
\end{thm}

We now state our main result, which says that the solution to \eqref{eq:NLS} scatters except near finitely many directions.
\begin{thm}
\label{thm:finite-bad-directions}
Let $d\geq5$, and let $u\in C(\R_{\geq 0};H^1(\R^d))$ solve
\begin{align*}
 (i\partial_t+\Delta)u=-|u|^{4/(d-1)}u
\end{align*}
and satisfy \eqref{eq:uniform-H1}.
Let $u_+\in H^1(\R^d)$ be the radiation state of $u$. Then there is an integer $N\geq 0$ and distinct $v_1,\ldots,v_N\in\R^d$.
Let $V=\{v_1,\ldots,v_N\}$. For every $\rho>0$, there is $\chi_\rho\in C^\infty(\R^d)$, with $0\leq\chi_\rho\leq1$ and $\nabla\chi_\rho\in L^\infty(\R^d)$, satisfying
\begin{align*}
 \chi_\rho(y)&=0 &&\text{if }\dist(y,V)\leq\rho,\\
 \chi_\rho(y)&=1 &&\text{if }\dist(y,V)\geq2\rho,
\end{align*}
and
\begin{align}
\label{eq:scattering-away-finite-directions}
 \lim_{t\to\infty} \|\chi_\rho(x/t)(u(t)-e^{it\Delta}u_+)\|_{H^1(\R^d)}=0.
\end{align}
Thus $\chi_\rho\equiv1$ when $N=0$, and \eqref{eq:scattering-away-finite-directions} is then precisely forward scattering. More concretely, we have
\begin{align}
\label{eq:finite-bad-directions-bound}
 N &\leq \frac{\|u(0)\|_{L^2}\sup_{t\geq 0}\|\nabla u(t)\|_{L^2}} {\sqrt{2M(Q)E(Q)}} =\frac{\sqrt d\,\|u(0)\|_{L^2} \sup_{t\geq 0}\|\nabla u(t)\|_{L^2}} {\|Q\|_{L^2}\|\nabla Q\|_{L^2}}.
\end{align}
Thus the exceptional spacetime cones about the finitely many velocity directions may be chosen with arbitrarily small aperture.
\end{thm}

\begin{rmk}
One can view this as a soliton resolution type result (characterizing all possibilities of non-scattering solitons remains a very difficult step) on the velocity level.
This groups all non-scattering solutions with the same asymptotic velocity into a single cluster, and we will need to characterize the sublinear dynamics of these non-scattering solutions in future works.

Our expectation is that the method can be extended to the intercritical range (i.e., between mass-critical and energy-critical).
On the other hand, the extension to the mass-critical nonlinearity (like $|u|^{4/d}u$) and the energy critical case (like $|u|^{4/(d-2)}u$) would require some essentially new ingredient, in particular the compact attractor result like Theorem~\ref{thm:compact-attractor} in these settings.
Also, there is an extra scaling parameter to deal with in these settings.
\end{rmk}

\begin{proof}
Let $r(t)$ be as in \eqref{eq:v,r-def}, put $m_\sharp=M(u)-M(u_+)$, and let $\Theta$ be as in \eqref{eq:Theta-def}.
By Proposition~\ref{prop:residual-limits}, $M(r(t))\to m_\sharp\geq 0$. If $m_\sharp=0$, then, for every $R>0$,
\begin{align}
\label{eq:r(t)-H1-bound-1}
 \limsup_{t\to\infty}\|r(t)\|_{H^1} \leq\limsup_{t\to\infty}\|P_{>R}r(t)\|_{H^1},
\end{align}
because $\|P_{\leq R}r(t)\|_{H^1}\leq(1+R^2)^{1/2}\|r(t)\|_{L^2}\to0$. 
Using Proposition~\ref{prop:residual-frequency-tightness}, we know $r(t)\to0$ in $H^1$. Thus the solution itself scatters and we may take $N=0$ and $\chi_\rho\equiv1$.

Now we consider the case $m_\sharp>0$. By Theorem~\ref{thm:measure-convergence}, we have
\begin{align}
\label{eq:finite-direction-atomic-limit}
 \nu_t\rightharpoonup\mu_\infty =\sum_{j=1}^Na_j\delta_{v_j}, \qquad a_j>0,\quad v_j\in\R^d,
\end{align}
where the $v_j$ are distinct. Fix $\chi$ as in Proposition~\ref{prop:finite-bad-directions} and define $\mathcal{A}$ by \eqref{eq:directional-mass-energy}.
We first prove \eqref{eq:atomic-directions-above-threshold}:
\begin{align}
\label{eq:atomic-directions-above-threshold}
 \mathcal{A}(v_j)\geq\Theta \qquad(1\leq j\leq N).
\end{align}
If instead $\mathcal{A}(v_j)<\Theta$, the definition \eqref{eq:directional-mass-energy} gives $\rho_0>0$, $\delta\in(0,1)$, and $t_n\to\infty$ such that, with
\begin{align*}
 \psi_3(y)=\chi\Big(\frac{y-v_j}{\rho_0}\Big),
\end{align*}
one has
\begin{align*}
 M(\psi_3(x/t_n)u(t_n))E(\psi_3(x/t_n)u(t_n)) \leq(1-\delta)\Theta.
\end{align*}
Choose $\psi_0,\psi_1,\psi_2\in C_c^\infty(\R^d)$ satisfying $0\leq\psi_i\leq1$ for $i=0,1,2$, $\psi_0(v_j)=1$, and
\begin{align*}
 \supp\psi_0\Subset\{\psi_1=1\},\qquad \supp\psi_1\Subset\{\psi_2=1\},\qquad \supp\psi_2\Subset\{\psi_3=1\}.
\end{align*}
Then Theorem~\ref{thm:upgrade-II}, applied to $\psi_1,\psi_2,\psi_3$ (being $\chi_0,\chi_1,\chi_2$ there respectively), gives the two eventual below-threshold hypotheses of Theorem~\ref{thm:cone-scattering} for $\psi_0,\psi_1$.
Hence
\begin{align*}
 \|\psi_0(x/t)r(t)\|_{H^1}\to 0.
\end{align*}
On the other hand, \eqref{eq:finite-direction-atomic-limit} gives
\begin{align*}
 \|\psi_0(x/t)r(t)\|_{L^2}^2 =\int\psi_0^2\,d\nu_t  \to\int\psi_0^2\,d\mu_\infty\geq a_j>0,
\end{align*}
a contradiction. This proves \eqref{eq:atomic-directions-above-threshold}.

Taking $c_0=\Theta$ in Proposition~\ref{prop:finite-bad-directions}, we obtain the first bound in \eqref{eq:finite-bad-directions-bound}.
Then \eqref{eq:ground-state-identities} gives the last equality.

Set $V=\{v_1,\ldots,v_N\}$. For a prescribed $\rho>0$, choose $\chi_\rho$ as in Theorem~\ref{thm:finite-bad-directions}; then $\dist(\supp\chi_\rho,V)\geq\rho$.
Since the full trajectory in \eqref{eq:finite-direction-atomic-limit} converges to a measure supported on $V$, Proposition~\ref{prop:interval-evacuation}, applied with $A_n=n$, $B_n=n+1$, and $K=V$, yields
\begin{align*}
 \sup_{n\leq t\leq n+1}\|\chi_\rho(x/t)r(t)\|_{H^1}\to 0.
\end{align*}
This is \eqref{eq:scattering-away-finite-directions}.
\end{proof}

\subsection{Links to the existing literature}
\label{subsec:review-future}

The threshold in \eqref{eq:localized-mass-energy}--\eqref{eq:localized-mass-gradient} is the ground-state threshold of the intercritical focusing NLS.
See Weinstein's work \cite{Weinstein1983Sharp} on its relationship with the sharp Gagliardo--Nirenberg inequality and the characterization and uniqueness given by Kwong \cite{Kwong1989Uniqueness}.

Such a threshold for scattering arose in the pioneering work treating the energy-critical NLS by Kenig and Merle \cite{KenigMerle2006}, where the concentration-compactness and rigidity method was introduced to prove below-ground-state scattering in the radial case.
The below ground-state scattering in the non-radial case for the energy-critical NLS in $d\geq5$ is proven by Killip and Visan \cite{Killip-Visan-H1critical}.
In the intercritical range (i.e., mass supercritical and energy subcritical) the sharp form of the threshold is due to Holmer and Roudenko \cite{HolmerRoudenko2008}, for radial data and for the three-dimensional cubic equation, which is \eqref{eq:NLS} with $d=3$. Afterwards, Duyckaerts, Holmer, and Roudenko \cite{DHR-3d-sc} removed the radial assumption in that dimension.

Fang, Xie, and Cazenave \cite{Fang-Xie-Cazenave-scattering-11} prove below ground-state scattering in general dimension and for every intercritical power, and Guevara \cite{Guevara12} also proved below ground-state scattering in that generality. We use this below ground-state scattering theory in Proposition~\ref{prop:profile-threshold}.
See also \cite{AkahoriNawa2013} for more discussions.

A second route on this question uses the Morawetz estimate.
Dodson and Murphy proved below ground-state scattering first for the three-dimensional cubic equation \cite{DodsonMurphy2017} with radial data,
and then, through an interaction Morawetz estimate, for non-radial data in every dimension $d\geq3$ for the $\dot{H}^{1/2}$-critical equation \cite{DodsonMurphy2018}, which has the same type of nonlinearity as \eqref{eq:NLS}; we use this result as Theorem~\ref{thm:blow-thre-sc}.
The two-dimensional radial case is treated in \cite{AroraDodsonMurphy2020}. Any of \cite{Fang-Xie-Cazenave-scattering-11,Guevara12,DodsonMurphy2018} would serve equally well as the input to Section~\ref{sec:below_threshold_scattering}.

Our Theorem~\ref{thm:cone-scattering} is the analogue of the global below ground-state scattering results, localized to a spacetime cone as in \eqref{eq:C-Omega-def}.
On the other hand, in the global results \cite{DHR-3d-sc,Fang-Xie-Cazenave-scattering-11,Guevara12,DodsonMurphy2018}, the threshold hypothesis is imposed on the initial data globally and yields scattering of the solution as a whole.
Our Theorem~\ref{thm:cone-scattering} aims to allow large initial data and include soliton components while deriving scattering in the region that is away from these soliton components.
So Theorem~\ref{thm:cone-scattering} fits in a more general picture described by the soliton resolution conjecture.
In our setting, no global threshold assumption is made: the solution is arbitrary,
subject only to a uniform $H^1$ bound. Finally, the hypothesis is asked for at all large times, whereas concentration compactness produces information only along sequences of times.
Theorem~\ref{thm:upgrade-II} closes precisely that gap, which is why it is isolated as a separate statement.
After obtaining the localized scattering in Theorem~\ref{thm:cone-scattering}, there can only be finitely many directions on which scattering fails, which is our Theorem \ref{thm:finite-bad-directions}.

On the other hand, starting with the work of Duyckaerts and Merle \cite{Duychaets-Merle-thre-H1NLS}, there have been many successes in characterizing the dynamics of NLS solutions that are on or above the threshold determined by the ground-state soliton.
The works that are most closely related to our current setting are the works of Duyckaerts and Roudenko \cite{Duychaetz-Roudenko-threshold-3Dcubic} and the work of Nakanishi and Schlag \cite{Nakanishi-Schalg-above-thre-3Dcubic}.

The limitation $d\geq 5$ that we have in our main results comes from Theorem~\ref{thm:compact-attractor}, which is \cite[Theorem~1.28]{Tao2007Attractor}.
It is used to give a profile decomposition for us. See also \cite{Tao2004Radial}\cite{Tao2008Potential}\cite{TaoSolitonsStable}.

We should also mention that the spacetime cone approach is inspired by the work of Gell-Redman, Gomes and Hassell \cite{GRGH-LS,GRGH-NLS},
and the long tradition of microlocal analysis.
Although these works are not directly used on a technical level, they have provided conceptual guidance to the author.

\subsection{The structure of the paper}
\label{subsec:the_structure_of_the_paper}
The structure of the paper is as follows.
We recall some facts from the existing literature in Sections~\ref{sec:preliminaries}--\ref{sec:thresholds} on linear and nonlinear solutions (mostly the compact attractor and the ground state soliton).
Then we prove Theorem~\ref{thm:cone-scattering} in Section~\ref{sec:below_threshold_scattering}.
In Sections~\ref{sec:atomic}--\ref{sec:commutators} we study the asymptotic behaviour of the measure using some parts of our solution as the density.
This will be the key step to identify those directions on which we could potentially have solitons.
Afterwards, Section~\ref{sec:limiting-measure} proves Theorem~\ref{thm:measure-convergence}.
Finally, we will prove our byproduct: the estimates upgrading the smallness or the below threshold property along a time sequence to a uniform property in Sections~\ref{sec:upgrade-I} and~\ref{sec:upgrade-II}.

\subsection*{Acknowledgements}
The project grew out of an ongoing project with Andrew Hassell.
The author is very grateful to Andrew Hassell for many helpful conversations and sharing the idea of using spacetime cones.
The author is supported by the Australian Research Council through grant FL220100072.

\section{Preliminaries: Radiation and compact attractor}
\label{sec:preliminaries}

\subsection{Free radiation}
\label{subsec:free_radiation}
We begin by recording the standard estimates for the free Schr\"odinger equation that are used throughout the paper.
The estimates in Proposition~\ref{prop:strichartz} are standard; see \cite[Theorem~2.3]{Tao-06-book} or \cite[Chapter~2]{Cazenave-03-book}.
\begin{prop}[Dispersive and Strichartz estimates]
\label{prop:strichartz}
Let $d\geq2$. For every $t\neq0$, we have
\begin{align}
\label{eq:dispersive}
 \|e^{it\Delta}\phi\|_{L^\infty(\R^d)} \leq(4\pi|t|)^{-d/2}\|\phi\|_{L^1(\R^d)} \qquad(\phi\in L^1(\R^d)).
\end{align}
Then interpolation and the unitarity of $e^{it\Delta}$ on $L^2(\R^d)$ give:
\begin{align}
\label{eq:dispersive-interpolated}
 \|e^{it\Delta}\phi\|_{L^p(\R^d)} \leq C_{d,p}|t|^{-d(\frac{1}{2}-\frac{1}{p})}\|\phi\|_{L^{p'}(\R^d)} \qquad(\phi\in L^{p'}(\R^d)),
\end{align}
where $2\leq p\leq \infty$ and $\frac{1}{p'}+\frac{1}{p}=1$.
We say that a pair $(q,r)$ is admissible if
\begin{align}
\label{eq:admissible}
 \frac{2}{q}+\frac{d}{r}=\frac{d}{2}, \quad 2\leq q,r\leq\infty,
\end{align}
excluding $(q,r)=(2,\infty)$ when $d=2$.
Then, for all admissible $(q,r)$ and $(\tilde{q},\tilde{r})$, every interval $I\subset\R$, and every $t_0\in I$,
\begin{align}
\label{eq:strichartz-homogeneous}
 \|e^{it\Delta}f\|_{L_t^qL_x^r(I\times\R^d)} &\leq C_{d,q}\|f\|_{L^2(\R^d)}, \\
\label{eq:strichartz-inhomogeneous}
 \Big\|\int_{t_0}^te^{i(t-s)\Delta}F(s)\,ds\Big\| _{L_t^qL_x^r(I\times\R^d)} &\leq C_{d,q,\tilde{q}} \|F\|_{L_t^{\tilde{q}'}L_x^{\tilde{r}'}(I\times\R^d)},
\end{align}
with constants independent of $I$ and $t_0$.
\end{prop}

We next record some other properties of the linear solutions.
\begin{lem}\label{lem:diagonal-radiation}
Let $d\geq5$. There is a constant $C_d$ such that
\begin{align}
\label{eq:diagonal-radiation-estimate}
 \|e^{it\Delta}f\|_{L^{2(d+1)/(d-1)}(\R_t\times\R_x^d)} \leq C_d\|f\|_{H^1(\R^d)}
\end{align}
for every $f\in H^1(\R^d)$. In particular,
\begin{align}
\label{eq:diagonal-radiation-integrability}
 \int_\R\|e^{it\Delta}f\|_{L^{2(d+1)/(d-1)}}^{2(d+1)/(d-1)}\,dt<\infty
\end{align}
and
\begin{align}
\label{eq:weighted-free-integrability}
 \lim_{T\to\infty}\int_T^\infty \|e^{it\Delta}f\|_{L^{2(d+1)/(d-1)}}^{2(d+1)/(d-1)}\frac{dt}{t}=0.
\end{align}
\end{lem}

\begin{proof}
Set $ s_*=\frac{1}{d+1}$, $r_*=\frac{2d(d+1)}{d^2-d+2}$. Then $(2(d+1)/(d-1),r_*)$ is admissible in the sense of \eqref{eq:admissible}. Apply \eqref{eq:strichartz-homogeneous} to $|D|^{s_*}f$ to obtain
\begin{align*}
 \||D|^{s_*}e^{it\Delta}f\|_{L_t^{2(d+1)/(d-1)}L_x^{r_*}} \leq C_d\|f\|_{\dot{H}^{s_*}}.
\end{align*}
Since
\begin{align*}
 \frac{d-1}{2(d+1)}=\frac{1}{r_*}-\frac{s_*}{d},
\end{align*}
Sobolev embedding gives
\begin{align*}
 \|e^{it\Delta}f\|_{L_x^{2(d+1)/(d-1)}} \leq C_d\||D|^{s_*}e^{it\Delta}f\|_{L_x^{r_*}}.
\end{align*}
Combining inequalities above, we have
\begin{align*}
 \|e^{it\Delta}f\|_{L_{t,x}^{2(d+1)/(d-1)}} \leq C_d\|f\|_{\dot{H}^{s_*}} \leq \|f\|_{L^2}^{1-s_*} \|\nabla f\|_{L^2}^{s_*} \leq\|f\|_{H^1},
\end{align*}
which is \eqref{eq:diagonal-radiation-estimate}. Raising \eqref{eq:diagonal-radiation-estimate} to the power $2(d+1)/(d-1)$ proves \eqref{eq:diagonal-radiation-integrability}, and \eqref{eq:weighted-free-integrability} follows from this directly.
\end{proof}

We record Lemma~\ref{lem:free-lq-decay}, a modified version of the dispersive estimate. It follows directly from \eqref{eq:dispersive-interpolated} if the initial data is in $L^{\frac{2(d+1)}{d+3}}$,
but needs a bit more justification if we only require the $L^2$-level decay of the initial data.

\begin{lem}
\label{lem:free-lq-decay}
Let $d\geq5$. For every $f\in H^1(\R^d)$,
\begin{align}
\label{eq:modified-dispersive}
 \lim_{|t|\to\infty}\|e^{it\Delta}f\|_{L^{2(d+1)/(d-1)}(\R^d)}=0.
\end{align}
\end{lem}

\begin{proof}
Since $2<2(d+1)/(d-1)<2d/(d-2)$, Sobolev embedding gives
\begin{align}
\label{eq:subcritical-sobolev-embedding}
 \|h\|_{L^{2(d+1)/(d-1)}}\leq C_d\|h\|_{H^1}.
\end{align}
For a Schwartz function $\phi$, the case $p=2(d+1)/(d-1)$ of \eqref{eq:dispersive-interpolated} gives
\begin{align}
\label{eq:free-lq-dispersive}
 \|e^{it\Delta}\phi\|_{L^{2(d+1)/(d-1)}} \leq C_d|t|^{-d/(d+1)}\|\phi\|_{L^{2(d+1)/(d+3)}}.
\end{align}

For $f\in H^1$, choose Schwartz $\phi$ with $\|f-\phi\|_{H^1}<\epsilon$. The $H^1$-unitarity of $e^{it\Delta}$ and \eqref{eq:subcritical-sobolev-embedding} give
\begin{align*}
 \sup_t\|e^{it\Delta}(f-\phi)\|_{L^{2(d+1)/(d-1)}}\leq C_d\epsilon.
\end{align*}
Using \eqref{eq:free-lq-dispersive}, $\|e^{it\Delta}\phi\|_{L^{2(d+1)/(d-1)}} \leq \epsilon$ for sufficiently large $|t|$.
Therefore
\begin{align*}
 \limsup_{|t|\to\infty}\|e^{it\Delta}f\|_{L^{2(d+1)/(d-1)}}\leq(C_d+1)\epsilon.
\end{align*}
Sending $\epsilon \to 0$ proves \eqref{eq:modified-dispersive}.
\end{proof}

Next we discuss the behavior of the measure given by the mass and current density.
The intuition behind the limits below is that if we expand a scattering solution to the Schr\"odinger equation, then the leading order term (with $x/2t$ as the variable) is the Fourier transform of the initial data.
The convention for the Fourier transform we use in this paper is
\begin{align*}
 \widehat{f}(\xi)=(2\pi)^{-d}\int_{\R^d}e^{-ix\cdot\xi}f(x)\,dx.
\end{align*}

\begin{prop} 
\label{prop:free-mass-current}
Let $d\geq1$, $u_0\in H^1(\R^d)$, $v_{u_0}(t)=e^{it\Delta}u_0$ for $t>0$, and let $j$ be as in \eqref{eq:j-def}.
Define the measure $\rho_{u_0}$ by
\begin{align*}
 \int_{\R^d}\Psi(y)\,d\rho_{u_0}(y) =(2\pi)^d\int_{\R^d}\Psi(2\xi)|\widehat{u_0}(\xi)|^2\,d\xi
\end{align*}
for bounded continuous $\Psi$. Define $\mu_t^{u_0}$ and $J_t^{u_0}$ by
\begin{align}
\label{eq:muf_t-def}
 \int\Psi\,d\mu_t^{u_0} =\int_{\R^d}\Psi(x/t)|v_{u_0}(t,x)|^2\,dx,
\end{align}
\begin{align}
\label{eq:Jf_t-def}
 \int\Psi\,dJ_t^{u_0} =\int_{\R^d}\Psi(x/t)j(v_{u_0}(t))(x)\,dx.
\end{align}
Then, as $t\to\infty$,
\begin{align}
\label{eq:free-mass-current-limits}
 \mu_t^{u_0}& \to\rho_{u_0}, \; J_t^{u_0} \to \frac{y}{2}\rho_{u_0} \; \text{in total variation}.
\end{align}
Moreover, $\mu_t^{u_0}$ and $|J_t^{u_0}|$ are uniformly tight for all sufficiently large $t$: for every $\epsilon>0$,
there are $R,T>0$ such that, for all $t\geq T$,
\begin{align}
\label{eq:tight-def}
 \mu_t^{u_0}(\{y\in\R^d:|y|>R\}) +|J_t^{u_0}|(\{y\in\R^d:|y|>R\})<\epsilon.
\end{align}
\end{prop}

\begin{proof}
Set
\begin{align*}
 g_t(x)=e^{i|x|^2/(4t)}u_0(x).
\end{align*}
Since $v_{u_0}=e^{it\Delta}u_0$, we have
\begin{align*}
 v_{u_0}(t,2t\xi)=\Big(\frac{\pi}{t}\Big)^{d/2}\gamma_d e^{it|\xi|^2}\widehat{g}_t(\xi), \qquad |\gamma_d|=1.
\end{align*}
After the change of variables $x=2t\xi$,
\begin{align}
\label{eq:free-mass-fourier-pushforward}
 \int\Psi(x/t)|v_{u_0}(t,x)|^2\,dx =(2\pi)^d\int\Psi(2\xi)|\widehat{g}_t(\xi)|^2\,d\xi.
\end{align}
Multiplication by $e^{i|x|^2/(4t)}$ converges strongly to the identity on $L^2$. Thus $g_t\to u_0$ in $L^2$, $\widehat{g}_t\to\widehat{u_0}$ in $L^2$, and
\begin{align}
\label{eq:free-fourier-density-l1-convergence}
 \big\||\widehat{g}_t|^2-|\widehat{u_0}|^2\big\|_{L^1} \leq(2\pi)^{-d}(\|g_t\|_{L^2}+\|u_0\|_{L^2})\|g_t-u_0\|_{L^2}\to 0.
\end{align}
Equations \eqref{eq:free-mass-fourier-pushforward}--\eqref{eq:free-fourier-density-l1-convergence} prove total-variation convergence of the mass measures.

Suppose first that $u_0$ is Schwartz. A direct computation gives (with $\xi=x/(2t)$)
\begin{align*}
 j(v_{u_0}(t))(x)=\Big(\frac{\pi}{t}\Big)^{d}\Big( \xi|\widehat{g}_t(\xi)|^2 +\frac{1}{2t}\Im(\overline{\widehat{g}_t(\xi)} \nabla_\xi\widehat{g}_t(\xi))\Big).
\end{align*}
Consequently, we have
\begin{align}
\label{eq:free-current-fourier-decomposition}
 \int\Psi(x/t)\,dJ_t^{u_0} =(2\pi)^d\int\Psi(2\xi)\xi|\widehat{g}_t(\xi)|^2\,d\xi+R_t,
\end{align}
where $R_t$ satisfies $|R_t|\leq\frac{\|\Psi\|_{L^\infty}}{2t}\|u_0\|_{L^2}\|xu_0\|_{L^2}$.
For Schwartz $u_0$, we have $g_t\to u_0$ in $H^1$. This is because
\begin{align*}
 \nabla g_t=e^{i|x|^2/(4t)} \Big(\nabla u_0+\frac{i x}{2t}u_0\Big).
\end{align*}
This gives $g_t\to u_0$ in $H^1$ for Schwartz $u_0$ as $t\to\infty$ by dominated convergence.
Taking Fourier transform, we have $\xi\widehat{g}_t\to\xi\widehat{u_0}$ in $L^2$, and consequently
\begin{align*}
 \xi|\widehat{g}_t|^2\to\xi|\widehat{u_0}|^2 \quad\text{in }L^1.
\end{align*}
Therefore the right-hand side of \eqref{eq:free-current-fourier-decomposition} (hence the left-hand side as well) converges to
\begin{align*}
 (2\pi)^d\int\Psi(2\xi)\xi|\widehat{u_0}(\xi)|^2\,d\xi =\int\Psi(y)\frac{y}{2}\,d\rho_{u_0}(y).
\end{align*}
For general $u_0\in H^1$, choose Schwartz $u_{0,n}\to u_0$ in $H^1$. Uniformly in $t$,
\begin{align}
\label{eq:free-mass-stability}
 |\int\Psi\,d(\mu_t^{u_0}-\mu_t^{u_{0,n}})| \leq\|\Psi\|_{L^\infty}(\|u_0\|_{L^2}+\|u_{0,n}\|_{L^2})\|u_0-u_{0,n}\|_{L^2},
\end{align}
and
\begin{align}
\label{eq:free-current-stability}
 |\int\Psi\,d(J_t^{u_0}-J_t^{u_{0,n}})| \leq\|\Psi\|_{L^\infty}( \|u_0-u_{0,n}\|_{L^2}\|\nabla u_0\|_{L^2} +\|u_{0,n}\|_{L^2}\|\nabla(u_0-u_{0,n})\|_{L^2}).
\end{align}
The estimates \eqref{eq:free-mass-stability} and \eqref{eq:free-current-stability} also compare the two limiting measures on the right-hand side of \eqref{eq:free-mass-current-limits}. First let $t\to\infty$ and then $n\to\infty$.

Now we prove the tightness in \eqref{eq:tight-def}. For every measurable set $E$, by Cauchy--Schwarz we have 
\begin{align}
\label{eq:free-current-mass-control}
 |J_t^{u_0}|(E)\leq\mu_t^{u_0}(E)^{1/2}\|\nabla u_0\|_{L^2}.
\end{align}
Then \eqref{eq:free-mass-current-limits} controls the mass term in \eqref{eq:tight-def}, while \eqref{eq:free-current-mass-control} controls the current term.
\end{proof}

The following lemma shows that the majority of the mass of the solution can't move arbitrarily fast.
\begin{lem}\label{lem:exterior-cone}
Let $d\geq1$, $p>0$, and $u$ be a solution to
\begin{align*}
 (i\partial_t+\Delta)u=-|u|^pu.
\end{align*}  
Suppose $u\in C(\R_{\geq 0};H^1(\R^d))$ and $B=\sup_{t\geq 0}\|u(t)\|_{H^1}<\infty$. Let $u_+\in H^1$, and let $v(t)$ and $r(t)$ be as in \eqref{eq:v,r-def}.
If $\eta\in C^\infty(\R^d)$, $0\leq\eta\leq1$, $\eta=0$ on $\{|y|\leq1\}$, and $\eta=1$ on $\{|y|\geq2\}$, then
\begin{align}
\label{eq:uniform-exterior-cone-tightness}
 \lim_{R\to\infty}\sup_{t\geq1} ( \|\eta(x/(Rt))u(t)\|_{L^2}+ \|\eta(x/(Rt))v(t)\|_{L^2}+ \|\eta(x/(Rt))r(t)\|_{L^2} )=0.
\end{align}
\end{lem}

\begin{proof}
Fix $R\geq1$ and $t\geq1$. For $s \in [0,t]$, define
\begin{align*}
 F_{R,t}(s)=\int\eta(x/(Rt))^2|u(s,x)|^2\,dx.
\end{align*}
Differentiating in $s$ and using \eqref{eq:NLS}, we know that $\frac{d}{ds}F_{R,t}(s)$ is an integral involving $u$ and $\nabla u$, and we have \footnote{Rigorously speaking, we need to approximate by nice solutions first.}
\begin{align*}
 \big|\frac{d}{ds}F_{R,t}(s)\big| \leq\frac{C_\eta}{Rt}\|u(s)\|_{L^2}\|\nabla u(s)\|_{L^2} \leq\frac{C_\eta B^2}{Rt}.
\end{align*}
Integrating, we obtain
\begin{align}
\label{eq:exterior-cone-u-bound}
 F_{R,t}(t) \leq\int_{|x|\geq Rt}|u(0,x)|^2\,dx+\frac{C_\eta B^2}{R} \leq\int_{|x|\geq R}|u(0,x)|^2\,dx+\frac{C_\eta B^2}{R}.
\end{align}
The right-hand side of \eqref{eq:exterior-cone-u-bound} is independent of $t$ and tends to zero as $R \to \infty$.

Applying the same argument in the proof of \eqref{eq:exterior-cone-u-bound} to $v$ gives
\begin{align*}
 \sup_{t\geq1}\|\eta(x/(Rt))v(t)\|_{L^2}^2 \leq\int_{|x|\geq R}|u_+(x)|^2\,dx +\frac{C_\eta\|u_+\|_{H^1}^2}{R}.
\end{align*}
Finally, $\eta(x/(Rt))r=\eta(x/(Rt))u-\eta(x/(Rt))v$, so applying the triangle inequality proves \eqref{eq:uniform-exterior-cone-tightness}.
\end{proof}

\subsection{Compact attractor}

In this subsection, we recall some results from \cite{Tao2007Attractor}. 
For every $f$ for which the NLS flow is defined, denote the solution at time $t$ with initial data $f$ by $S(t)f$.
A set $K\subset H^1(\R^d)$ is precompact modulo at most $J$ translations if there is a compact set $K_0\subset H^1(\R^d)$ such that every element of $K$ is a sum of at most $J$ translates of elements of $K_0$.
Then we have the following decay estimate that is uniform over $K$.
\begin{lem}{\cite[Corollary~B.6]{Tao2007Attractor}}
\label{lem:tao-decay}
Let $K\subset H^1(\R^d)$ be precompact modulo at most $J$ translations. For every $2<q\leq2d/(d-2)$,
\begin{align}
\label{eq:tao-global-decay}
\lim_{t\to\pm\infty}\sup_{f\in K}\|e^{it\Delta}f\|_{L^q}&=0.
\end{align}
For every $R>0$,
\begin{align}
\label{eq:tao-local-decay}
\lim_{t\to\pm\infty}\sup_{f\in K}\sup_{x_0\in\R^d} \int_{|x-x_0|\leq R}(|e^{it\Delta}f(x)|^2+|\nabla e^{it\Delta}f(x)|^2)\,dx&=0.
\end{align}
\end{lem}

Then we come to an important black box for us, the profile decomposition of the solutions to \eqref{eq:NLS}.
\begin{thm}
\label{thm:compact-attractor}
Let $d\geq5$, and let $u\in C(\R_{\geq 0};H^1(\R^d))$ solve
\begin{align*}
 (i\partial_t+\Delta)u=-|u|^{4/(d-1)}u
\end{align*}
and satisfy \eqref{eq:uniform-H1}.
There exist an integer $J\geq1$, a unique radiation state $u_+\in H^1(\R^d)$, and a closed translation-invariant set $\calK$,
precompact modulo at most $J$ translations. The forward NLS solution operator $S(t)$ is defined on $\calK$ and leaves it forward invariant.
Writing $J\calK$ for the set of sums of $J$ elements of $\calK$, one has
\begin{align*}
 \operatorname{dist}_{H^1} \big(u(t)-e^{it\Delta}u_+,J\calK\big)\to 0 \qquad(t\to\infty).
\end{align*}
Moreover, for every sequence $t_n\to\infty$ there are a subsequence, profiles $w_1,\ldots,w_J\in\calK$,
and centers $x_{1,n},\ldots,x_{J,n}\in\R^d$ satisfying
\begin{align}
\label{eq:profile-decomposition}
 u(t_n)=e^{it_n\Delta}u_+ +\sum_{j=1}^Jw_j(\,\cdot-x_{j,n})+o_{H^1}(1),
\end{align}
and
\begin{align}
\label{eq:profile-separation}
 |x_{j,n}-x_{k,n}| \to\infty \qquad (j\neq k).
\end{align}
\end{thm}
Verifying that our nonlinearity satisfies the conditions in \cite{Tao2007Attractor} is standard, and in fact this is already listed as an example in \cite{Tao2007Attractor} immediately after the conditions are stated.

Next we show that the soliton part can't have arbitrarily large frequency.
For $y\in\R^d$, we denote the translation action by $\tau_yf(x)=f(x-y)$.
\begin{prop}
\label{prop:residual-frequency-tightness}
Let $r(t)$ be as in \eqref{eq:v,r-def}.
If $P_{>R}$ is the Fourier projection onto $\{|\xi|>R\}$, then
\begin{align}
\label{eq:residual-frequency-tightness}
 \lim_{R\to\infty}\limsup_{t\to\infty} \|P_{>R}r(t)\|_{H^1}=0.
\end{align}
\end{prop}

\begin{proof}
By Theorem~\ref{thm:compact-attractor}, the residual approaches a set contained in a finite sum of translates of one compact subset $K\subset H^1$.
Thus there are $M\in\N$ and $\eta(t)\to0$ such that, for all sufficiently large $t$,
\begin{align*}
 \big\|r(t)-\sum_{m=1}^M\tau_{x_m(t)}k_m(t)\big\|_{H^1} \leq\eta(t), \qquad k_m(t)\in K.
\end{align*}
Compactness of $K$ implies
\begin{align*}
 \sup_{k\in K}\|P_{>R}k\|_{H^1}\to 0.
\end{align*}
Indeed, take a finite $\epsilon$-net in $K$, use dominated convergence in Fourier space for the net points,
and use that $P_{>R}$ has norm at most one on $H^1$. Since Fourier projection commutes with translations,
\begin{align*}
 \|P_{>R}r(t)\|_{H^1} \leq\eta(t)+M\sup_{k\in K}\|P_{>R}k\|_{H^1}.
\end{align*}
Taking the late-time limsup and then $R\to\infty$ proves \eqref{eq:residual-frequency-tightness}.
\end{proof}

Next we estimate the current of the solution. Define:
\begin{align}
\label{eq:j-def}
 j(f):=\Im(\overline{f}\nabla f).
\end{align}
Asymptotically, the soliton part and the radiation part tend to decouple, and we have the following asymptotic orthogonality.
\begin{prop}
\label{prop:strong-orthogonality}
Let $v(t)$ and $r(t)$ be as in \eqref{eq:v,r-def}. Then
\begin{align}
\label{eq:residual-radiation-mixed-l1-decay}
 \int_{\R^d}\big( |r||v|+|r||\nabla v|+|v||\nabla r| \big)\,dx\to 0.
\end{align}
Consequently,
\begin{align}
\label{eq:mass-density-l1-decoupling}
 \big\||u|^2-|v|^2-|r|^2\big\|_{L^1}\to 0,
\end{align}
\begin{align}
\label{eq:current-density-l1-decoupling}
 \|j(u)-j(v)-j(r)\|_{L^1}\to 0.
\end{align}
\end{prop}

\begin{proof}
For $f\in L^2$ and fixed $R>0$,
\begin{align}
\label{eq:free-uniform-local-l2-decay}
 \sup_{y\in\R^d}\|e^{it\Delta}f\|_{L^2(B(y,R))}\to 0.
\end{align}
Approximate $f$ by a compactly supported smooth function, use \eqref{eq:dispersive} on the approximation,
and use $L^2$ unitarity for the error. Apply \eqref{eq:free-uniform-local-l2-decay} to $u_+$ and each component of $\nabla u_+$.

For fixed $w\in H^1$, arbitrary centers $x_n$, and $t_n\to\infty$, split $w(\cdot-x_n)$ into a fixed-radius ball and an $H^1$-small tail.
Cauchy--Schwarz and \eqref{eq:free-uniform-local-l2-decay} show
\begin{align*}
 \int|w(x-x_n)||v(t_n,x)|\,dx\to0,
\end{align*}
\begin{align}
\label{eq:translated-profile-free-orthogonality}
 \int|w(x-x_n)||\nabla v(t_n,x)|\,dx\to0, \qquad \int|\nabla w(x-x_n)||v(t_n,x)|\,dx\to0.
\end{align}
For an arbitrary sequence $t_n\to\infty$, apply Theorem~\ref{thm:compact-attractor}:
\begin{align}
\label{eq:r-tn-decomposition}
 r(t_n)=\sum_{j=1}^Jw_j(\cdot-x_{j,n})+e_n, \qquad e_n\to0\text{ in }H^1.
\end{align}
Thus \eqref{eq:residual-radiation-mixed-l1-decay} holds on every extracted subsequence and therefore along the full time variable. Also, we have
\begin{align*}
 \big(|u|^2-|v|^2\big)-|r|^2=2\Re(\overline{v}r), \quad  j(u)-j(v)-j(r) =\Im(\overline{v}\nabla r+\overline{r}\nabla v).
\end{align*}
Taking $L^1$ norms gives \eqref{eq:mass-density-l1-decoupling}--\eqref{eq:current-density-l1-decoupling}.
Pushforward does not increase total variation.
\end{proof}

We consider the mass, energy, momentum of the soliton part below.
By the decoupling property of the soliton part and the radiation part mentioned above, these quantities are obtained by subtracting the mass, energy, or momentum of the radiation part from that of the entire solution.
This decoupling property will be used many times throughout the paper.

\begin{prop}\label{prop:residual-limits}
Let $M$, $E$, and $P$ be as in \eqref{eq:def-mass-energy}. Define
\begin{align} \label{eq:m,e,p,R,residual}
\begin{split}
 m_\sharp&:=M(u)-M(u_+), \qquad e_\sharp:=E(u)-\frac{1}{2}\|\nabla u_+\|_{L^2}^2,\\
 p_\sharp&:=P(u)-P(u_+), \qquad R_\sharp:=m_\sharp e_\sharp-\frac{|p_\sharp|^2}{2}.
\end{split}
\end{align}
For $v(t)$ and $r(t)$ as in \eqref{eq:v,r-def},
\begin{align} \label{eq:residual-limits}
 M(r(t))\to m_\sharp,\qquad E(r(t))\to e_\sharp,\qquad P(r(t))\to p_\sharp.
\end{align}
\end{prop}

\begin{proof}
The limits of $M(r(t))$ and $P(r(t))$ follow directly from Proposition~\ref{prop:strong-orthogonality}.
It remains to prove the energy limit.
The argument used in the proof of Proposition~\ref{prop:strong-orthogonality} also gives
\begin{align*}
 \la\nabla v(t),\nabla r(t)\ra\to0.
\end{align*}
More precisely, along an arbitrary sequence $t_n\to\infty$, take the profile decomposition \eqref{eq:r-tn-decomposition}.
The argument proving \eqref{eq:translated-profile-free-orthogonality}, now applied to the components of $\nabla u_+$ and $\nabla w_j$, makes the pairing of $\nabla v(t_n)$ with every translated profile tend to zero, while Cauchy--Schwarz handles the $H^1$-small remainder.
Since the original sequence was arbitrary, the displayed limit holds for $t\to\infty$.

We have $\|v(t)\|_{L^{2(d+1)/(d-1)}} \to 0$ by Lemma~\ref{lem:free-lq-decay}.
The pointwise power inequality and the uniform $H^1$ bounds give
\begin{align*}
 \big|\|u(t)\|_{L^{2(d+1)/(d-1)}}^{2(d+1)/(d-1)}-\|r(t)\|_{L^{2(d+1)/(d-1)}}^{2(d+1)/(d-1)}\big|\lesssim \|v(t)\|_{L^{2(d+1)/(d-1)}}\big(\|u(t)\|_{L^{2(d+1)/(d-1)}}^{(d+3)/(d-1)}+\|r(t)\|_{L^{2(d+1)/(d-1)}}^{(d+3)/(d-1)}\big)\to0.
\end{align*}
Using $u=v+r$ and the definition of the energy, we therefore have
\begin{align*}
 E(u)-\frac{1}{2}\|\nabla v(t)\|_{L^2}^2-E(r(t))= &\Re\la\nabla v(t),\nabla r(t)\ra\\
 &-\frac{d-1}{2(d+1)}\big(\|u(t)\|_{L^{2(d+1)/(d-1)}}^{2(d+1)/(d-1)}-\|r(t)\|_{L^{2(d+1)/(d-1)}}^{2(d+1)/(d-1)}\big)\to0.
\end{align*}
Finally, $\|\nabla v(t)\|_{L^2}=\|\nabla u_+\|_{L^2}$ by unitarity of the free flow.
This proves the energy limit and hence \eqref{eq:residual-limits}.
\end{proof}

\subsection{Finiteness of non-scattering directions}
\label{subsec:finiteness_of_bad_directions}

Now we prove that there can only be finitely many directions on which the localized product of mass and energy does not tend to zero in a narrow cone near these directions. To quantify this, for $v\in\R^d$, we define
\begin{align}
\label{eq:directional-mass-energy}
\mathcal{A}(v) :=\liminf_{\rho\downarrow0}\liminf_{t\to\infty} M\Big(\chi\Big(\frac{x/t-v}{\rho}\Big)u(t)\Big) E\Big(\chi\Big(\frac{x/t-v}{\rho}\Big)u(t)\Big).
\end{align}
Then our result is the following, which is a crucial step in proving Theorem~\ref{thm:finite-bad-directions}.
\begin{prop} \label{prop:finite-bad-directions}
Let $u$ be a global $H^1$ solution to \eqref{eq:NLS} satisfying \eqref{eq:uniform-H1}.
Fix a bump function $\chi\in C_c^\infty(\R^d)$, with $0\leq\chi\leq1$, supported in $B(0,1)$, equal to one near the origin.
For every $c_0>0$, there are only finitely many $v\in\R^d$ such that $\mathcal{A}(v)\geq c_0$. More precisely,
\begin{align}
\label{eq:N-est}
\#\{v\in\R^d:\mathcal{A}(v)\geq c_0\} \leq\frac{\|u(0)\|_{L^2}\sup_{t\geq 0}\|\nabla u(t)\|_{L^2}} {\sqrt{2c_0}}.
\end{align}
\end{prop}

\begin{proof}
Take any $N$ distinct velocities $v_1,\ldots,v_N$ satisfying
\begin{align}
\label{eq:positive-directions}
\mathcal{A}(v_j)\geq c_0,\qquad 1\leq j\leq N.
\end{align}
We prove a bound for $N$ independent of the chosen velocities.
Fix $0<\epsilon<c_0/2$. By \eqref{eq:directional-mass-energy} and \eqref{eq:positive-directions},
for each $j$, whenever $\rho>0$ is sufficiently small,
\begin{align}
\label{eq:ME-lowerbound-1}
\liminf_{t\to\infty}M\Big(\chi\Big(\frac{x/t-v_j}{\rho}\Big)u(t)\Big)E\Big(\chi\Big(\frac{x/t-v_j}{\rho}\Big)u(t)\Big)\geq c_0-\epsilon.
\end{align}
Since $v_1,\ldots,v_N$ are distinct, we can choose sufficiently small $\rho_1,\ldots,\rho_N>0$ so that the balls $B(v_j,\rho_j)$,
$1\leq j\leq N$, are pairwise disjoint and
\begin{align}
\label{eq:chosen-cone-lower-bound}
\liminf_{t\to\infty}M\Big(\chi\Big(\frac{x/t-v_j}{\rho_j}\Big)u(t)\Big)E\Big(\chi\Big(\frac{x/t-v_j}{\rho_j}\Big)u(t)\Big)\geq c_0-\epsilon
\end{align}
for every $1\leq j\leq N$.

Since there are only finitely many $j$, there exists $T>0$ such that, for every $t\geq T$ and every $1\leq j\leq N$,
\begin{align}
\label{eq:eventual-cone-lower-bound}
M\Big(\chi\Big(\frac{x/t-v_j}{\rho_j}\Big)u(t)\Big)E\Big(\chi\Big(\frac{x/t-v_j}{\rho_j}\Big)u(t)\Big)\geq c_0-2\epsilon.
\end{align}
For the energy defined in \eqref{eq:def-mass-energy}, we have $E(\frac{x/t-v_j}{\rho_j}\Big)u(t))\leq\frac{1}{2}\|\nabla \|\Big(\frac{x/t-v_j}{\rho_j}\Big)u(t) \|_{L^2}^2$.
So 
\begin{align}
\label{eq:localized-mass-gradient-lower-bound}
\big\|\chi\Big(\frac{x/t-v_j}{\rho_j}\Big)u(t)\big\|_{L^2}\big\|\nabla\Big(\chi\Big(\frac{x/t-v_j}{\rho_j}\Big)u(t)\Big)\big\|_{L^2}\geq\sqrt{2(c_0-2\epsilon)}.
\end{align}
Summing \eqref{eq:localized-mass-gradient-lower-bound} over $j$ and applying the Cauchy--Schwarz inequality gives
\begin{align}
\label{eq:CS-chi-v}
\begin{split}
N\sqrt{2(c_0-2\epsilon)} &\leq\sum_{j=1}^N\big\|\chi\Big(\frac{x/t-v_j}{\rho_j}\Big)u(t)\big\|_{L^2}\big\|\nabla\Big(\chi\Big(\frac{x/t-v_j}{\rho_j}\Big)u(t)\Big)\big\|_{L^2}
\\ & \leq \Big(\sum_{j=1}^N\big\|\chi\Big(\frac{x/t-v_j}{\rho_j}\Big)u(t)
\big \|_{L^2}^2\Big)^{1/2}\Big(\sum_{j=1}^N\big\|\nabla\Big(\chi\Big(\frac{x/t-v_j}{\rho_j}\Big)u(t)\Big)\big\|_{L^2}^2\Big)^{1/2}.
\end{split}
\end{align}
By construction, the supports of the cutoffs are pairwise disjoint for every $t>0$, and $0\leq\chi\leq1$. Therefore, by conservation of mass,
\begin{align}
\label{eq:localized-mass-packing}
\sum_{j=1}^N\big\|\chi\Big(\frac{x/t-v_j}{\rho_j}\Big)u(t)\big\|_{L^2}^2 &=\int_{\R^d}\sum_{j=1}^N\chi\Big(\frac{x/t-v_j}{\rho_j}\Big)^2|u(t,x)|^2\,dx \leq\|u(t)\|_{L^2}^2=\|u(0)\|_{L^2}^2.
\end{align}
Moreover,
\begin{align}
\label{eq:localized-gradient-expansion}
\nabla\Big(\chi\Big(\frac{x/t-v_j}{\rho_j}\Big)u(t)\Big) =\chi\Big(\frac{x/t-v_j}{\rho_j}\Big)\nabla u(t)+\frac{1}{t\rho_j}(\nabla\chi)\Big(\frac{x/t-v_j}{\rho_j}\Big)u(t).
\end{align}
Using again the pairwise disjointness of the supports and the triangle inequality in the direct sum over $j$,
we obtain
\begin{align}
\label{eq:est-chi-u-1}
\Big(\sum_{j=1}^N\big\|\nabla\Big(\chi\Big(\frac{x/t-v_j}{\rho_j}\Big)u(t)\Big)\big\|_{L^2}^2\Big)^{1/2} \leq\|\nabla u(t)\|_{L^2}+\frac{\|\nabla\chi\|_{L^\infty}}{t}\max_{1\leq j\leq N}\frac{1}{\rho_j}\|u(0)\|_{L^2}.
\end{align}
Combining \eqref{eq:CS-chi-v}, \eqref{eq:localized-mass-packing}, and \eqref{eq:est-chi-u-1},
for every $t\geq T$ we have
\begin{align}
\label{eq:cardinality-bound-with-error}
N\sqrt{2(c_0-2\epsilon)} \leq\|u(0)\|_{L^2} \Big(\sup_{s\geq 0}\|\nabla u(s)\|_{L^2} +
\frac{\| \nabla\chi\|_{L^\infty}}{t}\max_{1\leq j\leq N}\frac{1}{\rho_j}\|u(0)\|_{L^2}\Big).
\end{align}
Letting $t\to\infty$ in \eqref{eq:cardinality-bound-with-error} gives
\begin{align}
\label{eq:N-est-1}
N\sqrt{2(c_0-2\epsilon)}\leq\|u(0)\|_{L^2}\sup_{t\geq 0}\|\nabla u(t)\|_{L^2}.
\end{align}
Finally, letting $\epsilon \to 0$ in \eqref{eq:N-est-1}, we obtain \eqref{eq:N-est}.
\end{proof}

\section{Preliminaries about the soliton and the upgrading mechanism}
\label{sec:thresholds}

We recall some basic facts about the soliton in this section and prove some preliminary results about the upgrading mechanism that improves an estimate along a time sequence to an estimate that is uniform for all large times.

\begin{prop}[Ground-state identities and threshold trapping]
\label{prop:ground-state-property}
Let $d\geq5$, let $M$ and $E$ be defined by \eqref{eq:def-mass-energy}, and let $Q$ be the ground state in \eqref{eq:ground-state}. Then
\begin{align}
\label{eq:ground-state-identities}
\begin{aligned}
\|\nabla Q\|_{L^2}^2+\|Q\|_{L^2}^2 &=\|Q\|_{L^{2(d+1)/(d-1)}}^{2(d+1)/(d-1)} =(d+1)M(Q),\\
\|\nabla Q\|_{L^2}^2&=dM(Q),\qquad E(Q)=\frac{1}{2}M(Q)=\frac{1}{2d}\|\nabla Q\|_{L^2}^2,\\
M(Q)E(Q) &=\frac{(\|Q\|_{L^2}\|\nabla Q\|_{L^2})^2}{2d}>0.
\end{aligned}
\end{align}
For every $f\in H^1(\R^d)$, we have the sharp Gagliardo--Nirenberg inequality
\begin{align}
\label{eq:sharp-gn}
\|f\|_{L^{2(d+1)/(d-1)}}^{2(d+1)/(d-1)} &\leq C_{\mathrm{GN}}\|f\|_{L^2}^{2/(d-1)}\|\nabla f\|_{L^2}^{2d/(d-1)},
\end{align}
where
\begin{align}
\label{eq:sharp-gn-constant}
C_{\mathrm{GN}} &=\frac{d+1}{d}(\|Q\|_{L^2}\|\nabla Q\|_{L^2})^{-2/(d-1)}.
\end{align}
Set
\begin{align}
\label{eq:threshold-ratio}
s(f)&=\frac{\|f\|_{L^2}\|\nabla f\|_{L^2}}{\|Q\|_{L^2}\|\nabla Q\|_{L^2}}.
\end{align}
Then, without assuming a mass-energy bound,
\begin{align}
\label{eq:threshold-energy}
E(f)& \geq \Big(\frac{1}{2}-\frac{d-1}{2d}s(f)^{2/(d-1)}\Big) \|\nabla f\|_{L^2}^2.
\end{align}
In particular, $E(f)\geq 0$ whenever $s(f)\leq1$. If, in addition, for some $\delta\in(0,1)$,
\begin{align}
\label{eq:threshold-assumptions}
M(f)E(f)&\leq(1-\delta)M(Q)E(Q),\qquad s(f)<1,
\end{align}
then
\begin{align}
\label{eq:threshold-gap}
\|f\|_{L^2}\|\nabla f\|_{L^2} &\leq\sqrt{1-\delta}\|Q\|_{L^2}\|\nabla Q\|_{L^2}.
\end{align}
\end{prop}

\begin{proof}
Multiplying \eqref{eq:ground-state} by $Q$ and integrating by parts gives
\begin{align}
\label{eq:ground-state-testing}
\|\nabla Q\|_{L^2}^2+\|Q\|_{L^2}^2 &=\|Q\|_{L^{2(d+1)/(d-1)}}^{2(d+1)/(d-1)}.
\end{align}
The Pohozaev identity (in the form given in \cite{Berestycki-Lions-83}) gives
\begin{align}
\label{eq:ground-state-pohozaev}
\frac{d-2}{2}\|\nabla Q\|_{L^2}^2+\frac{d}{2}\|Q\|_{L^2}^2 &=\frac{d(d-1)}{2(d+1)}\|Q\|_{L^{2(d+1)/(d-1)}}^{2(d+1)/(d-1)}.
\end{align}
Combining \eqref{eq:ground-state-testing} and \eqref{eq:ground-state-pohozaev} gives
\begin{align*}
\|\nabla Q\|_{L^2}^2=d\|Q\|_{L^2}^2, \qquad \|Q\|_{L^{2(d+1)/(d-1)}}^{2(d+1)/(d-1)}=(d+1)\|Q\|_{L^2}^2,
\end{align*}
hence $E(Q)=\frac{1}{2}\|Q\|_{L^2}^2$, and all the formulas in \eqref{eq:ground-state-identities} follow.

Now we turn to the remaining claims.
Weinstein's result \cite{Weinstein1983Sharp} shows that the sharp constant in \eqref{eq:sharp-gn} is attained by a positive ground state, which solves \eqref{eq:ground-state}; Kwong's uniqueness theorem then identifies it with $Q$ \cite{Kwong1989Uniqueness}.
Hence equality holds in \eqref{eq:sharp-gn} for $Q$. Combining this with \eqref{eq:ground-state-identities} yields \eqref{eq:sharp-gn-constant}.
Substituting this constant into \eqref{eq:sharp-gn} gives \eqref{eq:threshold-energy}.
The coefficient on the right-hand side of \eqref{eq:threshold-energy} is nonnegative when $s(f)\leq1$, proving the $E(f)\geq 0$ part.

Assume now \eqref{eq:threshold-assumptions}. Multiplying \eqref{eq:threshold-energy} by $M(f)$ and combining with \eqref{eq:ground-state-identities} gives
\begin{align}
\label{eq:threshold-function}
\frac{M(f)E(f)}{M(Q)E(Q)} &\geq ds(f)^2-(d-1)s(f)^{2d/(d-1)} =s(f)^2(d-(d-1)s(f)^{2/(d-1)}) \geq s(f)^2,
\end{align}
where the last inequality used $s(f)<1$. Combining this with \eqref{eq:threshold-assumptions} gives $s(f)\leq\sqrt{1-\delta}$,
which is \eqref{eq:threshold-gap}.
\end{proof}

As one would expect, if a nonlinear solution has a compact trajectory (say in $H^1$), then it should not scatter and it should be `above threshold'. This is indeed the case as we shall show now.
On the other hand, the bound $M(w)E(w)\geq\Theta$ is not Galilean invariant. 
For fixed $M(w)$, the quantity $E(w)$ is minimized in the `rest frame', that is the one in which the momentum is zero.
Then being above threshold in this frame gives a stronger inequality in terms of our original $w$.
It is this form below that enters our proof in Section~\ref{sec:atomic}.
\begin{prop}
\label{prop:profile-threshold}
Let $M$, $E$, $P$, and $\Theta$ be as in \eqref{eq:def-mass-energy} and \eqref{eq:Theta-def}. Then every $w\in\calK\setminus\{0\}$ satisfies
\begin{align}
\label{eq:rest-frame-profile-threshold}
 E(w)-\frac{|P(w)|^2}{2M(w)}>0, \qquad M(w)E(w)-\frac{|P(w)|^2}{2}\geq\Theta.
\end{align}
\end{prop}

\begin{proof}
Fix $w\in\calK\setminus\{0\}$, and let $W(t)=S(t)w\in \calK$ be the solution of the NLS in Theorem~\ref{thm:compact-attractor} with $W(0)=w$.
Since $\calK$ is precompact modulo at most $J$ translations, $W$ is uniformly bounded in $H^1$ for all $t \geq 0$.
Let $P$ be as in \eqref{eq:def-mass-energy} and $w \neq 0$. Set $\xi=-P(w)/M(w)$ and define
\begin{align}
\label{eq:profile-galilean-transform}
 \widetilde{W}(t,x) =e^{i(x\cdot\xi-t|\xi|^2)}W(t,x-2\xi t),
\end{align}
which solves \eqref{eq:NLS}, and is still uniformly bounded in $H^1$.
In particular, we have
\begin{align*}
 M(\widetilde{W}(0))&=M(w),\qquad P(\widetilde{W}(0))=0,\\
 E(\widetilde{W}(0)) &=E(w)+\xi\cdot P(w)+\frac{1}{2}|\xi|^2M(w) =E(w)-\frac{|P(w)|^2}{2M(w)}.
\end{align*}
When $E(\widetilde{W}(0))<0$, the finite time blow-up is well-known (see \cite{Glassey-NLS-blowup}).

If $E(\widetilde{W}(0))=0$, then \cite[Lemma~2.16]{Guevara12} rules out Part~I of \cite[Theorem~A*]{Guevara12} for nonzero data,
while Part~II contradicts the same global boundedness. Hence $E(\widetilde{W}(0))>0$, which is the first inequality in \eqref{eq:rest-frame-profile-threshold}.
Multiplying the displayed identity for $E(\widetilde{W}(0))$ by $M(\widetilde{W}(0))=M(w)$ gives
\begin{align*}
 M(\widetilde{W}(0)) E(\widetilde{W}(0)) =M(w)E(w)- \frac{|P(w)|^2}{2}.
\end{align*}
If $M(w)E(w)-\frac{|P(w)|^2}{2}<\Theta$, then $M(\widetilde{W}(0))E(\widetilde{W}(0))<\Theta$. 
Combining with the uniform boundedness in $H^1$, \cite[Theorems~A and A*]{Guevara12} show that $\widetilde{W}$ scatters in the forward direction, and so does $W$.

If $W$ scattered forward, then for some $\phi_+\in H^1$, $e^{-it\Delta}W(t)\to\phi_+$ in $H^1$, and hence in $L^r$ for $2<r<2d/(d-2)$.
However, since $W(t) \in \calK$, Lemma~\ref{lem:tao-decay} gives
\begin{align}
\label{eq:Wt-Lr-decay}
 \|e^{-it\Delta}W(t)\|_{L^r} \leq\sup_{f\in\calK}\|e^{-it\Delta}f\|_{L^r} \to 0 \qquad(t\to\infty).
\end{align}
So we have $\phi_+=0$ and the conservation of mass shows
\begin{align*}
 M(w)=\|e^{-it\Delta}W(t)\|_{L^2}^2  \to\|\phi_+\|_{L^2}^2=0,
\end{align*}
whcih is a contradiction and completes the proof.
\end{proof}

Two results below are the starting points of our upgrading mechanism.
We will prove our upgrading estimates by contradiction.
The goal of this part is to show that if the uniform estimate for all large times fails, then the mass $m_\sharp=M(u)-M(u_+)$ of the soliton part is strictly positive.
For why this is useful, see Proposition~\ref{prop:mass-limit-form} below and the proof of Theorem~\ref{thm:upgrade-I} in Section~\ref{sec:upgrade-I}.

\begin{lem}\label{lem:first-entry}
Let $d\geq1$, $u\in C(\R_{\geq1};H^1)$, $u_+\in H^1$, and $v$ and $r$ as in \eqref{eq:v,r-def}. Let $\chi_0,\chi_1\in C_c^\infty(\R^d)$ satisfy $\chi_1=1$ on $\supp\chi_0$.
Suppose $T_n\to\infty$ and
\begin{align*}
 \|\chi_1(x/T_n)r(T_n)\|_{H^1}\to0.
\end{align*}
If
\begin{align}
\label{eq:uniform-est-chi0-r-Lemma-version}
 \limsup_{n\to\infty}\sup_{t\geq T_n} \|\chi_0(x/t)r(t)\|_{H^1}>0,
\end{align}
where the limsup may be infinite, then after passage to a subsequence there are $\epsilon>0$ and finite $s_n>T_n$,
with $s_n\to\infty$, such that
\begin{align*}
 \|\chi_0(x/s_n)r(s_n)\|_{H^1}=\epsilon,
\end{align*}
and we have $\|\chi_0(x/t)r(t)\|_{H^1}<\epsilon$ for $t \in (T_n,s_n)$. 
\end{lem}

\begin{proof}
Since $\chi_0=\chi_0\chi_1$, we have
\begin{align*}
 \|\chi_0(x/T_n)r(T_n)\|_{H^1} \leq C_{\chi_0}\|\chi_1(x/T_n)r(T_n)\|_{H^1} \to 0.
\end{align*}
Since $r(t)$ is $H^1$-continuous in $t$, we know $\|\chi_0(x/t)r(t)\|_{H^1}$ is continuous in $t$.

If \eqref{eq:uniform-est-chi0-r-Lemma-version} is true, then the quantity will fluctuate between 0 and $\limsup_{n\to\infty}\sup_{t\geq T_n} \|\chi_0(x/t)r(t)\|_{H^1}$ infinitely many times.
Then, fixing a sufficiently small $\epsilon$, we have $\|\chi_0(x/T_n)r(T_n)\|_{H^1}<\epsilon$ for large $n$.
Then we define
\begin{align*}
 s_n=\inf\{t\geq T_n:\|\chi_0(x/t)r(t)\|_{H^1}\geq\epsilon\},
\end{align*}
and $s_n$ will have the claimed property.
\end{proof}
The lemma above gives:
\begin{prop}
\label{prop:residual-mass-positive}
Let the setting be as in Lemma~\ref{lem:first-entry}.
Suppose \eqref{eq:uniform-chi0-r-tend-0} fails (or equivalently \eqref{eq:uniform-est-chi0-r-Lemma-version} is true). Then $m_\sharp=M(u)-M(u_+)>0$.
\end{prop}

\begin{proof}
We have $M(r(t))\to m_\sharp$ by Proposition~\ref{prop:residual-limits}.
If $m_\sharp=0$, arguing as in the discussion after \eqref{eq:r(t)-H1-bound-1}, we know $\|r(t)\|_{H^1} \to 0$, contradicting the failure of \eqref{eq:uniform-chi0-r-tend-0}.
Thus $m_\sharp>0$.
\end{proof}

\section{Below threshold scattering in a spacetime cone}
\label{sec:below_threshold_scattering}

We prove Theorem~\ref{thm:cone-scattering} in this section. We first recall the scattering result to which we resort.

\begin{thm}{\cite[Theorem~1.1]{DodsonMurphy2018}, \cite{Guevara12,Fang-Xie-Cazenave-scattering-11}}
\label{thm:blow-thre-sc}
Let $d\geq3$. For $u_0\in H^1(\R^d)$ such that
\begin{align}
\label{eq:dm-mass-energy}
M(u_0)E(u_0)&<M(Q)E(Q), \\
\label{eq:dm-mass-gradient}
\|u_0\|_{L^2}\|\nabla u_0\|_{L^2} & < \|Q\|_{L^2}\|\nabla Q\|_{L^2},
\end{align}
the solution $u$ of \eqref{eq:NLS} with $u(0)=u_0$ is global and scatters in both time directions.
We consider only the forward direction: there exists $u_+ \in H^1(\R^d)$ such that
\begin{align}
\label{eq:dm-forward}
\|u(t)-e^{it\Delta}u_+\|_{H^1}&\to 0\quad\text{as }t\to\infty.
\end{align}
\end{thm}

Using Theorem~\ref{thm:blow-thre-sc}, we now prove Theorem~\ref{thm:cone-scattering}.

\begin{proof}[Proof of Theorem~\ref{thm:cone-scattering}]
We first apply Theorem~\ref{thm:compact-attractor} to $u$ and use notation there. 
Take a sequence $t_n\to\infty$. 
After passing to a subsequence, Theorem~\ref{thm:compact-attractor} gives the decomposition in \eqref{eq:profile-decomposition} with centers satisfying \eqref{eq:profile-separation}. 
Passing to a subsequence again, each of the following three bounded sequences converges to a finite number:
\begin{align}
\label{eq:localized-scalar-limits}
&M(\chi_2(x/t_n)u(t_n)),\qquad \|\nabla(\chi_2(x/t_n)u(t_n))\|_{L^2}^2,\qquad E(\chi_2(x/t_n)u(t_n)),
\end{align}
and, for every $j$, the sequence $x_{j,n}/t_n$ either converges in $\R^d$ or tends to infinity:
\begin{align}
\label{eq:profile-velocity}
\frac{x_{j,n}}{t_n}&  \to v_j\in X=\mathbb{R}^d\cup\{\infty\}.
\end{align}

We will need the mass and energy decouplings:
\begin{align}
\label{eq:mass-decoupling}
M(\chi_2(x/t_n)u(t_n)) &=M(\chi_2(x/t_n)e^{it_n\Delta}u_+) +\sum_{j=1}^JM(c_jw_j)+o(1), \\
\label{eq:kinetic-decoupling}
\|\nabla(\chi_2(x/t_n)u(t_n))\|_{L^2}^2 &=\|\nabla(\chi_2(x/t_n)e^{it_n\Delta}u_+)\|_{L^2}^2 +\sum_{j=1}^J\|\nabla(c_jw_j)\|_{L^2}^2+o(1),
\end{align}
\begin{align}
\label{eq:full-nonlinear-decoupling}
\|\chi_2(x/t_n)u(t_n)\|_{L^{2(d+1)/(d-1)}}^{2(d+1)/(d-1)} &=\sum_{j=1}^J \|c_jw_j\|_{L^{2(d+1)/(d-1)}}^{2(d+1)/(d-1)}+o(1).
\end{align}
\begin{align}
\label{eq:energy-decoupling}
E(\chi_2(x/t_n)u(t_n)) &=\frac{1}{2}\|\nabla(\chi_2(x/t_n)e^{it_n\Delta}u_+)\|_{L^2}^2 +\sum_{j=1}^JE(c_jw_j)+o(1).
\end{align}
See e.g. \cite[Lemma~3.3]{Killip-Visan-H1critical}\cite[Lemma~2.2, Lemma~2.3]{DHR-3d-sc} for the proof.
The nonlinear contribution of the radiation is absent from \eqref{eq:full-nonlinear-decoupling} and \eqref{eq:energy-decoupling} because of the `dispersive estimate' in Lemma~\ref{lem:free-lq-decay}.

We now use the threshold estimates from Proposition~\ref{prop:ground-state-property}. For all sufficiently large $n$, one has $t_n\geq T_0$,
so \eqref{eq:localized-mass-energy}--\eqref{eq:localized-mass-gradient} apply to $f=\chi_2(x/t_n)u(t_n)$.
Applying \eqref{eq:threshold-gap} to this $f$, and using \eqref{eq:mass-decoupling} and \eqref{eq:kinetic-decoupling},
every fixed profile satisfies
\begin{align}
\label{eq:profile-threshold-gap}
\|c_jw_j\|_{L^2}\|\nabla(c_jw_j)\|_{L^2} &\leq\sqrt{1-\delta}\|Q\|_{L^2}\|\nabla Q\|_{L^2}<\|Q\|_{L^2}\|\nabla Q\|_{L^2}.
\end{align}
Applying \eqref{eq:threshold-energy} with $f=c_jw_j$, and using \eqref{eq:profile-threshold-gap}, gives
\begin{align}
\label{eq:profile-energy-positive}
E(c_jw_j)&\geq 0.
\end{align}
Thus the quadratic radiation term and every profile-energy term displayed explicitly on the right side of \eqref{eq:energy-decoupling} are nonnegative.

Suppose $v_j\in\supp\chi_1$.
By \eqref{eq:cutoff-nesting}, $c_j=\chi_2(v_j)=1$.
\begin{align}
\label{eq:profile-mass-bound}
M(w_j)&\leq\lim_{n\to\infty}M(\chi_2(x/t_n)u(t_n)), \\
\label{eq:profile-energy-bound}
E(w_j)&\leq\lim_{n\to\infty}E(\chi_2(x/t_n)u(t_n)).
\end{align}
The limits on the right-hand sides of \eqref{eq:profile-mass-bound}--\eqref{eq:profile-energy-bound} exist as explained in \eqref{eq:localized-scalar-limits}.
The inequalities \eqref{eq:profile-mass-bound}--\eqref{eq:profile-energy-bound} follow from \eqref{eq:mass-decoupling}, \eqref{eq:energy-decoupling}, and \eqref{eq:profile-energy-positive}.
All quantities in \eqref{eq:profile-mass-bound}--\eqref{eq:profile-energy-bound} are nonnegative by \eqref{eq:mass-decoupling}, \eqref{eq:energy-decoupling}, and \eqref{eq:profile-energy-positive}.
Using \eqref{eq:localized-mass-energy}, this gives
\begin{align}
\label{eq:profile-mass-energy-bound}
M(w_j)E(w_j) &\leq(1-\delta)M(Q)E(Q)<M(Q)E(Q).
\end{align}
If $w_j\neq0$, then \eqref{eq:rest-frame-profile-threshold} gives
\begin{align*}
 M(w_j)E(w_j)\geq M(w_j)E(w_j)-\frac{|P(w_j)|^2}{2}\geq\Theta=M(Q)E(Q),
\end{align*}
contradicting \eqref{eq:profile-mass-energy-bound}. Therefore
\begin{align}
\label{eq:inner-profile-zero}
v_j\in\supp\chi_1&\quad\Longrightarrow\quad w_j=0.
\end{align}

For every remaining profile, we know either $v_j=\infty$ or $v_j\notin\supp\chi_1$, then we have
\begin{align}
\label{eq:outer-profile-zero}
\| \chi_1(x/t_n) \tau_{x_{j,n}}w_j\|_{H^1}&\to 0.
\end{align}
The $o_{H^1}(1)$ term in \eqref{eq:profile-decomposition} remains $o_{H^1}(1)$ after multiplication by $\chi_1(x/t_n)$.

Subtracting the radiation term from \eqref{eq:profile-decomposition}, multiplying by $\chi_1(x/t_n)$, and using \eqref{eq:inner-profile-zero} and \eqref{eq:outer-profile-zero}, we obtain
\begin{align}
\label{eq:subsequence-conclusion}
\|\chi_1(x/t_n)(u(t_n)-e^{it_n\Delta}u_+)\|_{H^1}&\to 0.
\end{align}
Since the sequence $t_n\to\infty$ is arbitrary, this proves
\eqref{eq:cone-scattering}.
\end{proof}

\section{The atomic measures associated with the non-scattering part}

\label{sec:atomic}

In this section, we discuss the property of the measures \eqref{eq:nu_t,J_t-def} as $t\to\infty$.
In particular, we show that they can only concentrate at finitely many asymptotic directions.
This gives a characterization of the non-scattering part of our solution $u(t)$.
Equivalently, this says that asymptotically the non-scattering part $r(t)$ can only concentrate on finitely many directions.

To explain the intuition, suppose one indeed has a resolution of solitons.
Then one should expect an asymptotic expansion of $u(t)$ like
\begin{align}
u(t) \sim  Ct^{-d/2}\widehat{u}_+(x/2t)+\sum_{i=1}^J Q_i,
\end{align}
where $Q_i$ is the soliton component with velocity $y_i$ and $\widehat{u}_+$ is the Fourier transform of the `free part' $u_+$. See \cite{Hassell-Jia-NLS} for the discussion about the first term in the defocusing setting.
If we consider the space of velocities parametrized by $x/t$, then as $t\to\infty$, the radiation part (i.e., the first term above) tends to a continuous measure with density $|\widehat{u}_+|^2$, while the remaining soliton components tend to a sum of Dirac measures.
In addition, since, as we have seen, the soliton components and the radiation component decouple in the long-time asymptotics, if we consider the measure with density $|r(t)|^2=|u(t)-e^{it\Delta}u_+|^2$, then we remove the continuous part directly (that is, the cross terms tend to zero) and are left with a sum of atomic (i.e., Dirac) measures.
We will justify this intuition on the weak limit level.

\begin{prop}
\label{prop:subseq-atomicity}
Let $J$ and $u_+$ be as in Theorem~\ref{thm:compact-attractor},
and define $r(t)$ by \eqref{eq:v,r-def}. Let $X=\R^d\cup\{\infty\}$ be the one-point compactification of $\R^d$,
and let $\nu_t$ be the residual mass measure defined in \eqref{eq:nu_t,J_t-def}.

For every sequence $t_n\to\infty$, after passing to a subsequence, not relabeled, there are functions $w_1,\ldots,w_J\in H^1$,
centers $x_{j,n}\in\R^d$, points $y_j\in X$, and errors $e_n\in H^1$ such that
\begin{align}
\label{eq:residual-profile-decomposition}
 r(t_n)&=\sum_{j=1}^Jw_j(\cdot-x_{j,n})+e_n, \qquad \|e_n\|_{H^1}\to 0, \\
\label{eq:normalized-profile-centers}
 \frac{x_{j,n}}{t_n}& \to y_j \qquad(1\leq j\leq J),
\end{align}
and the centers satisfy \eqref{eq:profile-separation}.
Along this subsequence,
\begin{align}
\label{eq:nu-seq-limit}
 \nu_{t_n} \rightharpoonup \sum_{j=1}^JM(w_j)\delta_{y_j}.
\end{align}
Here the functions $w_j$ and the points $y_j$ may depend on the original sequence.
\end{prop}

\begin{proof}
The expression of $r(t)$ in \eqref{eq:residual-profile-decomposition} follows from \eqref{eq:profile-decomposition}. 
Since $X$ is compact\footnote{Again, this is just a convenience of writing. One can separately consider sequences for which $x_{j,n}/t_n$ remains bounded and sequences for which $x_{j,n}/t_n \to \infty$.} and there are only finitely many centers, a further subsequence satisfies $x_{j,n}/t_n\to y_j$ in $X$ for every $j$.

Let $\Phi\in C(X)$. The error $e_n$ contributes $o(1)$ because
\begin{align}
\label{eq:residual-profile-error-removal}
 \big|\int_{\R^d}\Phi(x/t_n) \Big(|r(t_n,x)|^2- \big|\sum_{j=1}^Jw_j(x-x_{j,n})\big|^2\Big)\,dx\big| &\leq \|\Phi\|_{L^\infty(X)}\|e_n\|_{L^2} \Big(2\sum_{j=1}^J\|w_j\|_{L^2}+\|e_n\|_{L^2}\Big) \to 0.
\end{align}
For each $j$, the change of variables $x=x_{j,n}+z$ gives
\begin{align}
\label{eq:mass-limit-3}
 \int_{\R^d}\Phi(x/t_n)|w_j(x-x_{j,n})|^2\,dx &=\int_{\R^d}\Phi\Big(\frac{x_{j,n}}{t_n}+\frac{z}{t_n}\Big) |w_j(z)|^2\,dz  \to \Phi(y_j)M(w_j).
\end{align}
Indeed, the argument of $\Phi$ converges to $y_j$ in $X$ for every fixed $z$, and dominated convergence applies.

For $j\neq k$, \eqref{eq:profile-separation} gives
\begin{align}
\label{eq:residual-profile-cross-limit}
 \int_{\R^d}\Phi(x/t_n)w_j(x-x_{j,n}) \overline{w_k(x-x_{k,n})}\,dx\to 0.
\end{align}
To see this, approximate $w_j$ and $w_k$ in $L^2$ by compactly supported functions. Their translates have disjoint supports for all sufficiently large $n$ by \eqref{eq:profile-separation},
and Cauchy--Schwarz controls the approximation errors uniformly in $n$.

Expanding $|\sum_{j=1}^Jw_j(x-x_{j,n})|^2$ and using
\eqref{eq:residual-profile-error-removal},
\eqref{eq:mass-limit-3}, and
\eqref{eq:residual-profile-cross-limit}, together with the definition
of $\nu_t$ in \eqref{eq:nu_t,J_t-def}, yields
\begin{align}
\label{eq:Phi-nu-limit}
 \int_X\Phi\,d\nu_{t_n} & \to\sum_{j=1}^J\Phi(y_j)M(w_j) =\int_X\Phi\,d\Big(\sum_{j=1}^JM(w_j) \delta_{y_j}\Big),
\end{align}
which proves \eqref{eq:nu-seq-limit}.
\end{proof}

Now we introduce the distance and the class of measures describing the limit we will consider. 
Fix a metric $\rho_X$ on $X$ such that the induced topology is the topology of the one-point compactification of $\R^d$.
For example, we can identify $X$ with $\mathbb{S}^{d}$ using the stereographic projection and use the round sphere metric.
Then we define
\begin{align*}
 d_{\mathrm{BL}}(\mu,\sigma) =\sup\big\{ |\int_X\Phi\,d(\mu-\sigma)|: \|\Phi\|_{L^\infty}\leq1,\ \Lip_{\rho_X}(\Phi)\leq1 \big\},
\end{align*}
where BL stands for `bounded Lipschitz'. Let
\begin{align*}
 d_J((\mu,J),(\sigma,H)) =d_{\mathrm{BL}}(\mu,\sigma) +\sum_{k=1}^dd_{\mathrm{BL}}(J_k,H_k).
\end{align*}

We use $\calD_\sharp$ to denote the class of pairs $(\mu,J)$ such that
\begin{align}
\label{eq:atomic-pair-def}
 \mu=\sum_{\ell=1}^La_\ell\delta_{y_\ell}, \qquad J=\sum_{\ell=1}^Lb_\ell\delta_{y_\ell},
\end{align}
where the $y_\ell\in X$ are distinct, $a_\ell>0$, $b_\ell\in\R^d$, and $n_\ell\in\N$, and where the following conditions hold.
Let $m_\sharp,e_\sharp,p_\sharp,R_\sharp$ be defined in \eqref{eq:m,e,p,R,residual}, and define
\begin{align}
\label{eq:eta-def}
 \eta=\frac{p_\sharp}{m_\sharp}.
\end{align}
We require:
\begin{align*}
 \sum_\ell a_\ell=m_\sharp,\qquad \sum_\ell b_\ell=p_\sharp,
\end{align*}
\begin{align}
\label{eq:mres-bound}
 m_\sharp \sum_{\ell=1}^L \frac{n_\ell^2\Theta}{a_\ell} +\frac{m_\sharp}{2}\sum_{\ell=1}^La_\ell \big |\frac{b_\ell}{a_\ell}-\eta\big|^2 \leq R_\sharp.
\end{align}

\begin{rmk} 
\label{rmk:n-ell}
In the proofs below, the reader can simply take $n_\ell=1$ and collect all terms with the same $y_\ell$ together.
We still include this parameter here because we believe that this is the conceptually correct formulation for future work treating the dynamics of soliton components at scales sublinear in $t$, and we want to prove these technical lemmas at that level of generality.
See \eqref{eq:e-ell-ex} below for an example of the natural choice of $n_\ell$.
\end{rmk}

Then we have the following result.
\begin{prop}
\label{prop:atomic-and-bound}
Let $\calD_\sharp$ be the class defined by \eqref{eq:atomic-pair-def}--\eqref{eq:mres-bound}.
If $m_\sharp>0$, then $R_\sharp>0$. With $J_t$ defined in \eqref{eq:nu_t,J_t-def}, we have
\begin{align}
\label{eq:dist-to-class}
 \dist_{d_J}((\nu_t,J_t),\calD_\sharp)\to0.
\end{align}
For every profile decomposition \eqref{eq:residual-profile-decomposition}--\eqref{eq:normalized-profile-centers} whose associated measures \eqref{eq:nu_t,J_t-def} converge to an atomic pair $(\mu,J)$ of the form \eqref{eq:atomic-pair-def},
$n_\ell$ may be chosen as
\begin{align}
\label{eq:e-ell-ex}
 n_\ell =\#\big\{j:w_j\neq0,\ \frac{x_{j,n}}{t_n}\to y_\ell\big\}.
\end{align}
If $N=\sum_\ell n_\ell$, then
\begin{align}
\label{eq:atom-mass-and-count}
 a_\ell\geq\frac{m_\sharp n_\ell^2\Theta}{R_\sharp},\qquad N^2\Theta\leq R_\sharp.
\end{align}
\end{prop}

\begin{proof}
Take an arbitrary $t_n \to \infty$ and apply Proposition~\ref{prop:subseq-atomicity} with profiles $w_j$. 
  
By the same arguement referred for obtaining \eqref{eq:mass-decoupling}-\eqref{eq:energy-decoupling}, 
applied to \eqref{eq:residual-profile-decomposition}, without the cut-off, and the same argument applied to $j(t)$ in place of $|u|^2$, together with Proposition~\ref{prop:residual-limits}, we have
\begin{align}
\label{eq:sum-M,E,P}
 \sum_j M(w_j)=m_\sharp,\qquad \sum_jE(w_j)=e_\sharp,\qquad \sum_jP(w_j)=p_\sharp.
\end{align}
We already have the convergence of $\nu_{t_n}$ in \eqref{eq:nu-seq-limit}.
Now we prove
\begin{align}
\label{eq:subsequential-momentum-measure-limit}
J_{t_n}  \rightharpoonup \sum_jP(w_j)\delta_{y_j}.
\end{align}
Indeed, the $H^1$-small remainder changes the current in $L^1$ by $o(1)$.
The cross terms also have no contribution as $n\to\infty$ for the same reason as in the preceding decoupling argument.
Finally, each diagonal term converges to give terms on the right-hand side of \eqref{eq:subsequential-momentum-measure-limit}.

Define
\begin{align*}
 \nu_j=\frac{P(w_j)}{M(w_j)},\qquad R_j=M(w_j)E(w_j)-\frac{|P(w_j)|^2}{2} \geq \Theta>0,
\end{align*}
where the last inequality used Proposition~\ref{prop:profile-threshold}. Using \eqref{eq:sum-M,E,P}, a direct computation gives
\begin{align}
\label{eq:Rres/mres}
 \frac{R_\sharp}{m_\sharp} =\sum_j\Big( \frac{R_j}{M(w_j)} +\frac{M(w_j)}{2}|\nu_j-\eta|^2\Big).
\end{align}
This proves $R_\sharp>0$.

Partition the indices $j$ with $w_j\neq0$ according to the limit of their normalized centers in \eqref{eq:normalized-profile-centers}, and write $G_\ell$ for the set of indices whose limit is $y_\ell$. For each $G_\ell$, set
\begin{align*}
 n_\ell=|G_\ell|,\quad a_\ell=\sum_{j\in G_\ell}M(w_j),\quad b_\ell=\sum_{j\in G_\ell}P(w_j).
\end{align*}
Using $R_j\geq\Theta$ recalled above and Cauchy-Schwarz, we have
\begin{align}
\label{eq:ineq-1}
 \sum_{j\in G_\ell}\frac{R_j}{M(w_j)} \geq\Theta\sum_{j\in G_\ell}\frac{1}{M(w_j)} \geq\frac{n_\ell^2\Theta}{a_\ell}.
\end{align}
Noticing that $(\sqrt{M(w_j)}(\nu_j-\frac{b_\ell}{a_\ell}),...)$ and $(..., \sqrt{M(w_j)}(\frac{b_\ell}{a_\ell}-\eta),...)$, where components run through all $j \in G_\ell$, are orthogonal, we have
\begin{align}
\label{eq:ineq-2}
 \sum_{j\in G_\ell}M(w_j)|\nu_j-\eta|^2 = \sum_{j \in G_\ell} M(w_j)|\nu_j-\frac{b_\ell}{a_\ell}|^2+ a_\ell\big|\frac{b_\ell}{a_\ell}-\eta\big|^2 \geq a_\ell\big|\frac{b_\ell}{a_\ell}-\eta\big|^2.
\end{align}
Inserting \eqref{eq:ineq-1} and \eqref{eq:ineq-2} into \eqref{eq:Rres/mres} gives \eqref{eq:mres-bound}.

The mass and current measures have uniformly bounded total variation, so weak convergence on compact $X$ is convergence in $d_J$.
Since our sequence $t_n$ is arbitrary, we have proved \eqref{eq:dist-to-class}.

Finally, each term in \eqref{eq:mres-bound} is nonnegative, giving the bound in \eqref{eq:atom-mass-and-count}. Also
\begin{align*}
 \sum_\ell\frac{n_\ell^2}{a_\ell} \geq\frac{(\sum_\ell n_\ell)^2}{\sum_\ell a_\ell} =\frac{N^2}{m_\sharp}.
\end{align*}
This yields $N^2\Theta\leq R_\sharp$.
\end{proof}

We next record the compactness and connectedness of the set of subsequential limits.
\begin{prop}
\label{prop:connected-omega}
Let $\Omega_\sharp$ be the set of all $d_J$-limits of $(\nu_{t_n},J_{t_n})$ along $t_n\to\infty$. Then $\calD_\sharp$ is compact and $\Omega_\sharp$ is a nonempty compact connected subset of $\calD_\sharp$.
\end{prop}

\begin{rmk}
\label{rmk:no-hopping}
In particular, this shows that if $(\nu_t,J_t)$ does not converge (i.e., if the sequential limits are not all the same), the failure cannot consist of shifts within a discrete set of configurations of the asymptotic atomic measure in \eqref{eq:atomic-pair-def}.
\end{rmk}

\begin{proof}
We prove $\calD_\sharp$ is compact first.
Given a sequence $(\mu_k,J_k)$ in $\calD_\sharp$, write
\begin{align}
\mu_k=\sum_{\ell=1}^{L_k}a_{\ell,k}\delta_{y_{\ell,k}},\qquad J_k=\sum_{\ell=1}^{L_k}b_{\ell,k}\delta_{y_{\ell,k}},
\end{align}
where the $y_{\ell,k}$ are distinct and the corresponding multiplicities are $n_{\ell,k}$. By \eqref{eq:atom-mass-and-count}, we have
\begin{align}
L_k\leq\sum_{\ell=1}^{L_k}n_{\ell,k}\leq\sqrt{\frac{R_\sharp}{\Theta}},\qquad a_{\ell,k}\geq\frac{m_\sharp n_{\ell,k}^2\Theta}{R_\sharp}\geq\frac{m_\sharp\Theta}{R_\sharp}.
\end{align}
Moreover, using \eqref{eq:mres-bound}, we have
\begin{align}
a_{\ell,k}\big|\frac{b_{\ell,k}}{a_{\ell,k}}-\eta\big|^2\leq\frac{2R_\sharp}{m_\sharp},
\end{align}
and hence
\begin{align}
|b_{\ell,k}|\leq a_{\ell,k}|\eta|+\sqrt{\frac{2R_\sharp a_{\ell,k}}{m_\sharp}}\leq m_\sharp|\eta|+\sqrt{2R_\sharp}.
\end{align}
So $L_k$ and $n_{\ell,k}$ are uniformly bounded. After passing to a subsequence and relabeling, we may therefore assume
\begin{align}
L_k=L,\qquad n_{\ell,k}=n_\ell.
\end{align}
Similarly, since $X$ is compact, passing to a subsequence again, we may assume
\begin{align}
y_{\ell,k} \to y_\ell,\qquad a_{\ell,k} \to a_\ell,\qquad b_{\ell,k} \to b_\ell
\end{align}
for every $1\leq\ell\leq L$. The lower bound for $a_{\ell,k}$ gives $a_\ell>0$.

The limiting points $y_\ell$ need not be distinct. Suppose, for example,
\begin{align}
y_1=\cdots=y_s.
\end{align}
Replace these terms by one term at $y_1$ with coefficient $\sum_{\ell=1}^s a_\ell$ and $\sum_{\ell=1}^s b_\ell$, and multiplicity $\sum_{\ell=1}^s n_\ell$.
Cauchy--Schwarz gives
\begin{align} \label{eq:CS-al-nl-bl}
\frac{(\sum_{\ell=1}^s n_\ell)^2}{\sum_{\ell=1}^s a_\ell}\leq\sum_{\ell=1}^s\frac{n_\ell^2}{a_\ell}, 
\quad \Big(\sum_{\ell=1}^s a_\ell\Big)\big|\frac{\sum_{\ell=1}^s b_\ell}{\sum_{\ell=1}^s a_\ell}-\eta\big|^2 
= \frac{\Big|\sum_{\ell=1}^s a_\ell \Big(\frac{b_\ell}{a_\ell}-\eta\Big)\Big|^2}{\sum_{\ell=1}^s a_\ell} 
\leq \sum_{\ell=1}^s a_\ell\big|\frac{b_\ell}{a_\ell}-\eta\big|^2.
\end{align}
Thus collecting the same limiting positions together does not increase either term in \eqref{eq:mres-bound}.
Repeating this produces an expression that is a sum of delta functions at distinct limiting positions and still belongs to $\calD_\sharp$.
The convergence of the positions, masses and momenta (i.e., those coefficients) gives convergence in $d_J$. Hence every sequence in $\calD_\sharp$ has a subsequence converging in $d_J$ to an element of $\calD_\sharp$,
and we know that $\calD_\sharp$ is compact.

The pair $(\nu_t,J_t)$ is continuous in $t$ with respect to $d_J$. This follows from $H^1$-continuity of $r(t)$, dominated convergence for the self-similar multiplier, and
\begin{align*}
 \|j(r(t_n))-j(r(t))\|_{L^1}\to0 \quad\text{when }r(t_n)\to r(t)\text{ in }H^1.
\end{align*}
Its range is relatively compact because the mass and current total variations are uniformly bounded.
For $T\geq1$, consider
\begin{align*}
 K_T=\overline{\{(\nu_t,J_t):t\geq T\}}.
\end{align*}
Each $K_T$ is nonempty, compact, and connected, and the family is nested (decreasing in $T$). Its intersection is $\Omega_\sharp$, so the latter is nonempty, compact,
and connected. Since $\calD_\sharp$ is compact, \eqref{eq:dist-to-class} shows $\Omega_\sharp\subset \calD_\sharp$.
\end{proof}

Next we consider the family of subsequential limits of our measures $(\nu_t,J_t)$ parametrized by a time translation parameter. Since the proof is quite involved, we briefly sketch the idea of the proof first.
The goal is to derive the limiting behaviour of $(\nu_t,J_t)$.
As Proposition~\ref{prop:strong-orthogonality} shows, roughly speaking, one can approximate $|r|^2$ by $|u|^2-|v|^2$, and we will reduce our problem to investigating the measure with density $|u|^2-|v|^2$.
Similarly, for the measure that originally has density $j(r)$, we will approximate it by the measure with density $j(u)-j(v)$. After integrating the time derivative of such a measure (see \eqref{eq:A-phi-def} below), we obtain an identity that is analogous to \eqref{eq:mass-difference} at finite time.
Then, roughly speaking, the goal is to `take the limit' in this identity.
See also $D_A$ in \eqref{eq:DA-def} below for the motivation to consider such measures.

Another ingredient is that, in order to obtain a nice family of limits parametrized by the delayed time $\tau$, we prove the equicontinuity in $\tau$ of the auxiliary sequence of measures above and then apply the Arzel\`a--Ascoli theorem.
It is when we compute the time derivative of these measures that our PDEs enter, through the local conservation law in \eqref{eq:current}.
Now we state our result.

\begin{thm}
\label{thm:subsequential-family}
Let $r(t)$ be defined by \eqref{eq:v,r-def},
and suppose $m_\sharp:=M(u)-M(u_+)>0$. Define $\nu_t$ and $J_t$ by \eqref{eq:nu_t,J_t-def}, and let $d_J$ and $\calD_\sharp$ be as in Proposition~\ref{prop:atomic-and-bound}.
For every sequence $s_n\to\infty$, after potentially passing to a subsequence, there exists a family $(\mu(\tau),J(\tau))$ that is continuous in $\tau$ with respect to $d_J$,
such that, for every $S>0$,
\begin{align} \label{eq:family-limit}
 \sup_{|\tau|\leq S} d_J\big((\nu_{e^{s_n+\tau}},J_{e^{s_n+\tau}}), (\mu(\tau),J(\tau))\big)\to0.
\end{align}
We call such a map a subsequential limit family. For every $\tau$, the pair belongs to $\calD_\sharp$ and $\mu(\tau)(\{\infty\})=0$.
Writing $J_k(\tau)$ for the $k$-th scalar component of $J(\tau)$, we have
\begin{align*}
 \supp J_k(\tau)\subseteq\supp\mu(\tau) \qquad(1\leq k\leq d),
\end{align*}
and $\supp\mu(\tau)$ is a finite subset of $\R^d$. For every real-valued $\phi\in C_c^\infty(\R^d)$ and $\alpha<\beta$, we have
\begin{align}
\label{eq:mass-difference}
 \int\phi\,d\mu(\beta)-\int\phi\,d\mu(\alpha) = \int_\alpha^\beta\Big( 2\int_{\R^d}\nabla\phi(y)\cdot dJ(\tau)(y) -\int_{\R^d}y\cdot\nabla\phi(y)\,d\mu(\tau)(y) \Big)d\tau.
\end{align}
\end{thm}

\begin{proof}
Let $v(t)=e^{it\Delta}u_+$ be as in \eqref{eq:v,r-def}, and let $j(f)=\Im(\overline{f}\nabla f)$ as in \eqref{eq:j-def}.  
Since we have uniform $H^1$ bounds for $u$, $v$, and $r$, Cauchy--Schwarz gives
\begin{align}
\label{eq:auxiliary-L1-bounds}
 \sup_{t}\Big(  \||u(t)|^2-|v(t)|^2\|_{L^1}
 + \|j(u(t))-j(v(t))\|_{L^1}\Big)<\infty.
\end{align}
Set $t_{n,\tau}=e^{s_n+\tau}$.
For a real-valued $\phi\in C_c^\infty(\R^d)$, define
\begin{align} 
\label{eq:A-phi-def}
A_\phi(s)=\int_{\R^d}\phi(x/e^s)
 (|u|^2-|v|^2)(e^s,x)\,dx.
\end{align}
To compute the time derivative of $A_\phi$, we use the mass-current conservation laws for $u$ and $v$ (due to \eqref{eq:NLS} and the fact that $v$ satisfies the linear Schr\"odinger equation):
\begin{align*}
 \partial_t|u|^2 + 2\nabla\cdot j(u) = 0, \quad
 \partial_t|v|^2 + 2\nabla\cdot j(v) = 0.
\end{align*}
This gives $\partial_t(|u|^2-|v|^2)=-2\nabla\cdot(j(u)-j(v))$, and we have
\begin{align}
\label{eq:log-mass-derivative}
 A_\phi'(s)=\int_{\R^d}\nabla\phi(x/e^s)\cdot
 \Big(2(j(u)-j(v))(e^s,x)
 -(x/e^s)(|u|^2-|v|^2)(e^s,x)\Big)\,dx.
\end{align}
In particular, for $\alpha<\beta$, integrating \eqref{eq:log-mass-derivative} from $s_n+\alpha$ to $s_n+\beta$ gives
\begin{align}
\label{eq:prelimit-mass}
\begin{split}
 A_\phi(s_n+\beta)-A_\phi(s_n+\alpha)  &= 2\int_\alpha^\beta\int_{\R^d}
 \nabla\phi(x/t_{n,\tau})\cdot
 (j(u)-j(v))(t_{n,\tau},x)\,dx\,d\tau \\
 &\quad-\int_\alpha^\beta\int_{\R^d}
 (x/t_{n,\tau})\cdot\nabla\phi(x/t_{n,\tau})
 (|u|^2-|v|^2)(t_{n,\tau},x)\, dx\, d\tau.
 \end{split}
\end{align}
This is the finite-time equality that we will pass to the limit. The right-hand side of \eqref{eq:log-mass-derivative} is uniformly bounded by \eqref{eq:auxiliary-L1-bounds}.

We also need to consider the measure with $j(u)-j(v)$ being its density. 
This is for both the joint convergence in \eqref{eq:family-limit} and for the current term in \eqref{eq:prelimit-mass}. For $1\leq k\leq d$, set
\begin{align}
\label{eq:def-B-phik}
 B_{\phi,k}(s)=\int_{\R^d}\phi(x/e^s) \big(j_k(u)-j_k(v) \big)(e^s,x)\,dx.
\end{align}
The local momentum laws for $u$ and $v$ reads
\begin{align}
\label{eq:current}
\begin{split}
 \partial_t j_k(u)
 &+\sum_{m=1}^d\partial_m\Big(
 2\Re(\partial_k\overline{u}\,\partial_m u)
 -\frac{2}{d+1} \delta_{km}|u|^{2(d+1)/(d-1)}\Big)
 -\frac{1}{2} \partial_k\Delta|u|^2=0,\\
 \partial_t j_k(v)
 &+\sum_{m=1}^d\partial_m\Big(
 2\Re(\partial_k\overline{v}\,\partial_m v)\Big) -\frac{1}{2} \partial_k\Delta|v|^2=0.
\end{split}
\end{align}
Using \eqref{eq:current}, we have
\begin{align}
\label{eq:log-current-derivative}
\begin{split}
 B_{\phi,k}'(s) = &-\int_{\R^d} (x/e^s)\cdot\nabla\phi(x/e^s)
 \big(j_k(u)-j_k(v)\big)(e^s,x)\,dx\\
 &+\sum_{m=1}^d\int_{\R^d}\partial_m\phi(x/e^s)
 \Big(2\Re(\partial_k\overline{u}\,\partial_m u
 -\partial_k\overline{v}\,\partial_m v)
 -\frac{2}{d+1}\delta_{km}|u|^{2(d+1)/(d-1)}\Big)(e^s,x)\,dx\\
 &-\frac{1}{2e^{2s}}\int_{\R^d}(\partial_k\Delta\phi)(x/e^s)
 (|u|^2-|v|^2)(e^s,x)\,dx.
\end{split}
\end{align}
The first line of \eqref{eq:log-current-derivative} is uniformly bounded by \eqref{eq:auxiliary-L1-bounds}. The quadratic terms in the second line are controlled by the uniform $H^1$ bounds, and for
the nonlinear term we use the embedding $H^1\hookrightarrow L^{2(d+1)/(d-1)}$ from
\eqref{eq:subcritical-sobolev-embedding}. 
The last line is $O(e^{-2s})$ by \eqref{eq:auxiliary-L1-bounds}. 

Thus, for every fixed $S>0$, the translated functions $A_\phi(s_n+\tau)$ and $B_{\phi,k}(s_n+\tau)$ are Lipschitz with a uniform Lipschitz constant for $|\tau|\leq S$ and all sufficiently large $n$.

Now we explain our approximation step. If $|\tau|\leq S$, then $t_{n,\tau}\geq e^{s_n-S}\to\infty$. 
Consequently, \eqref{eq:mass-density-l1-decoupling} and \eqref{eq:current-density-l1-decoupling} show that, for every $\Phi\in C(X;\R)$, we have
\begin{align}
\label{eq:mass-uvr-decouple}
 \sup_{|\tau|\leq S}\Big|\int_{\R^d}\Phi(x/t_{n,\tau})
 \big(|u|^2-|v|^2-|r|^2\big)(t_{n,\tau},x)\,dx\Big|\to0,
\end{align}
\begin{align}
\label{eq:current-decouple}
 \sup_{|\tau|\leq S}\Big|\int_{\R^d}\Phi(x/t_{n,\tau})
 \big(j(u)-j(v)-j(r)\big)(t_{n,\tau},x)\,dx\Big|\to 0.
\end{align}

We now prove \eqref{eq:family-limit}. 
Choose a countable dense family $\{\Phi_m\}_{m\geq1}\subset C(X;\R)$ containing the constant function $1$, such that every $\Phi_m$ restricts to a constant plus a function in $C_c^\infty(\R^d)$, and set
\begin{align*}
 \mathcal{V}=\operatorname{span}_{\R}\{\Phi_m:m\geq1\}.
\end{align*}
Write $\Phi_m=c_m+\phi_m$, where $c_m\in\R$ and
$\phi_m\in C_c^\infty(\R^d)$. Conservation of mass and momentum gives
\begin{align*}
 \int_{\R^d}(|u(t)|^2-|v(t)|^2)\,dx=m_\sharp,\qquad
 \int_{\R^d}\big(j(u(t))-j(v(t))\big)\,dx=p_\sharp,
\end{align*}
and hence
\begin{align*}
 &\int_{\R^d}\Phi_m(x/t_{n,\tau})(|u|^2-|v|^2)(t_{n,\tau},x)\,dx
 =c_m m_\sharp+A_{\phi_m}(s_n+\tau),\\
 &\int_{\R^d}\Phi_m(x/t_{n,\tau})\big(j_k(u)-j_k(v)\big)(t_{n,\tau},x)\,dx
 =c_m(p_\sharp)_k+B_{\phi_m,k}(s_n+\tau).
\end{align*}
These are uniformly bounded by our uniform $H^1$ property of $u,v$, and they are Lipschitz in $\tau$ over any fixed compact interval.
The Arzel\`a--Ascoli theorem and a diagonal extraction over $m$, the finitely many
current components $k$, and the bounds $|\tau|\leq S$ with $S \in \N$ therefore give one subsequence on which all these pairings converge locally uniformly in $\tau$.
By linearity, the same holds for every $\Phi\in\mathcal{V}$. The decoupling estimates \eqref{eq:mass-uvr-decouple} and \eqref{eq:current-decouple} show that the corresponding pairings of $(\nu_{t_{n,\tau}},J_{t_{n,\tau}})$ have the same limits.

Let $\|\cdot\|_{\TV}$ denote total variation. By
\eqref{eq:nu_t,J_t-def},
\begin{align}
\label{eq:tv-bounds}
 \|\nu_t\|_{\TV}=\|r(t)\|_{L^2}^2,\qquad
 \|J_t\|_{\TV}\leq\|r(t)\|_{L^2}\|\nabla r(t)\|_{L^2}.
\end{align}
Thus each scalar limit at each $\tau$ defines a linear functional on $\mathcal{V}$ bounded by $C\|\Phi\|_{L^\infty}$, with $C$ independent of $\tau$.
Since $\mathcal{V}$ is dense in $C(X;\R)$, each such functional has a unique extension to $C(X;\R)$ and is represented by one of the measures $\mu(\tau),J_1(\tau),\ldots,J_d(\tau)$, whose total variation is at most $C$.

Now we upgrade our results from $\mathcal{V}$ to general $\Phi \in C(X;\R)$. 
Indeed, for $\Psi\in\mathcal{V}$,
\begin{align}
\label{eq:dense-test-approximation}
 \big|\int_X\Phi\,d(\nu_{t_{n,\tau}}-\mu(\tau))\big|
 \leq &\big|\int_X\Psi\,d(\nu_{t_{n,\tau}}-\mu(\tau))\big|
 +2C\|\Phi-\Psi\|_{L^\infty}.
\end{align}
The same estimate holds for each current component. Approximating $\Phi$ uniformly by an element of $\mathcal{V}$ proves locally uniform convergence for every fixed continuous test function.

Fix $S>0$ and let
\begin{align*}
 \mathcal{B} = \{\Phi\in C(X):\|\Phi\|_{L^\infty}\leq1,  \ \Lip_{\rho_X}(\Phi) \leq 1 \}.
\end{align*}
Since $X$ is compact, $\mathcal{B}$ is compact in $C(X)$ by Arzel\`a--Ascoli theorem. Given $\epsilon>0$, choose a finite $\epsilon$-net $\Psi_1,\ldots,\Psi_N$ for $\mathcal{B}$. 
For $|\tau|\leq S$, \eqref{eq:dense-test-approximation} gives
\begin{equation}
\label{eq:bounded-lipschitz-net-estimate}
\begin{aligned}
d_J \big((\nu_{t_{n,\tau}},J_{t_{n,\tau}}),(\mu(\tau),J(\tau))\big) \leq & \max_{1\leq i\leq N}
 \big|\int_X\Psi_i\,d(\nu_{t_{n,\tau}}-\mu(\tau))\big|
 \\ + & \sum_{k=1}^d\max_{1\leq i\leq N} \big|\int_X\Psi_i\,d(J_{t_{n,\tau},k}-J_k(\tau))\big|
 +2(d+1) C \epsilon.
\end{aligned}
\end{equation}
The first two terms on the right tend to zero uniformly for $|\tau|\leq S$, because the net is finite. Hence
\begin{align*}
 \limsup_{n\to\infty}\sup_{|\tau|\leq S}
 d_J\big((\nu_{t_{n,\tau}},J_{t_{n,\tau}}),(\mu(\tau),J(\tau))\big)
 \leq2(d+1)C\epsilon.
\end{align*}
Letting $\epsilon \to 0$ proves \eqref{eq:family-limit}.

For each $n$, $(\nu_{t_{n,\tau}},J_{t_{n,\tau}})$ is continuous in $\tau$ with respect to $d_J$. Its locally uniform limit $(\mu(\tau),J(\tau))$ is therefore also continuous in $\tau$ with respect to $d_J$. 
Now for each fixed $\tau$, $t_{n,\tau}\to\infty$. 
So we have
\begin{align*}
(\mu(\tau),J(\tau)) \in \calD_\sharp
\end{align*}
by Proposition~\ref{prop:atomic-and-bound} and the compactness of $\calD_\sharp$ from Proposition~\ref{prop:connected-omega}. By the definition of
$\calD_\sharp$, the measure $\mu(\tau)$ has finite support in $X$ and
\begin{align*}
 \supp J_k(\tau)\subseteq\supp\mu(\tau)\qquad(1\leq k\leq d).
\end{align*}
By Lemma~\ref{lem:exterior-cone}, we have $\mu(\tau)(\{\infty\})=0$, or equivalently
$\supp\mu(\tau) \subset \R^d$ is compact.

We finally return to the exact finite-time identity \eqref{eq:prelimit-mass}. Fix $\alpha<\beta$. For
$\gamma\in\{\alpha,\beta\}$, \eqref{eq:family-limit} and \eqref{eq:mass-uvr-decouple} give
\begin{align*}
 A_\phi(s_n+\gamma)\to\int_X\phi\,d\mu(\gamma).
\end{align*}
Moreover, uniformly for $\alpha\leq\tau\leq\beta$,
\eqref{eq:family-limit},
\eqref{eq:mass-uvr-decouple}, and
\eqref{eq:current-decouple} give
\begin{align*}
 &\int_{\R^d}\nabla\phi(x/t_{n,\tau})\cdot
 (j(u)-j(v))(t_{n,\tau},x)\,dx
 \to\int_X\nabla\phi(y)\cdot dJ(\tau)(y),\\
 &\int_{\R^d}(x/t_{n,\tau})\cdot\nabla\phi(x/t_{n,\tau})
 (|u|^2-|v|^2)(t_{n,\tau},x)\,dx
 \to\int_X y\cdot\nabla\phi(y)\,d\mu(\tau)(y).
\end{align*}
The functions $\nabla\phi(y)$ and $y\cdot\nabla\phi(y)$ extend continuously to $X$, with value zero at $\infty$. We may therefore pass to the limit in \eqref{eq:prelimit-mass}. 
Since $\mu(\tau)$ and $J(\tau)$ have no support at $\infty$, the resulting identity is exactly \eqref{eq:mass-difference}.
\end{proof}

The identity \eqref{eq:mass-difference} in Theorem~\ref{thm:subsequential-family} does not control the momentum and this is the issue that we will handle next using a similar argument.
Throughout the rest of this section we fix a subsequential limit family $(\mu(\tau),J(\tau))$ obtained along $s_n\to\infty$
in Theorem~\ref{thm:subsequential-family}, write $t_{n,\tau}=e^{s_n+\tau}$, and let $v(t)$ be as in \eqref{eq:v,r-def}.
Finally, for $1\leq k,m\leq d$ we define a (signed) measure $\Sigma_{n,km}(\tau)$ on $X$ by
\begin{align*}
 \int_X\phi(y)\,d\Sigma_{n,km}(\tau)(y) = &\int_{\R^d}\phi(x/t_{n,\tau}) \Big( 2\Re(\partial_k\overline{u}\,\partial_m u -\partial_k\overline{v}\,\partial_m v) -\frac{2}{d+1}\delta_{km}|u|^{2(d+1)/(d-1)} \Big)(t_{n,\tau},x)\,dx
\end{align*}
for $\phi\in C(X)$.

\begin{prop}
\label{prop:stress-1}
After passing to a further subsequence, there are signed measures $\Sigma_{km}(\tau)$ on $X$,
uniformly bounded in total variation, such that
\begin{equation}
\label{eq:stress-measure-weakstar-convergence}
\begin{aligned}
 \int_{\R}\int_X\Phi(\tau,y)\,d\Sigma_{n,km}(\tau)(y)\,d\tau &\to \int_{\R}\int_X\Phi(\tau,y)\,d\Sigma_{km}(\tau)(y)\,d\tau
\end{aligned}
\end{equation}
whenever $\Phi(\tau,\cdot)$ is integrable in $\tau$ with values in $C(X)$ and vanishes outside a bounded interval.
For every $\chi\in C_c^\infty(\R^d)$, $\alpha<\beta$, and $1\leq k\leq d$,
\begin{align}
\label{eq:momentum-difference}
 \int_{\R^d}\chi\,dJ_k(\beta) -\int_{\R^d}\chi\,dJ_k(\alpha) &= \int_\alpha^\beta\Big( -\int_{\R^d} y\cdot\nabla\chi(y)\,dJ_k(\tau)(y) +\sum_{m=1}^d\int_{\R^d}\partial_m\chi(y)\, d\Sigma_{km}(\tau)(y) \Big)d\tau.
\end{align}
Let $I\subset\R$ be a compact interval, and let $K\subset\R^d$ be compact.
Suppose that
\begin{align}
\supp \mu(\tau) \subset K, \quad \forall \tau \in I,
\end{align}
Then \footnote{Rigorously speaking, $\Sigma_{km}$ is only defined a.e. for $\tau$. But it will only be involved in integrals and we ignore this issue below.}
\begin{align}
\label{eq:Sigma-km-supp}
\supp \, \Sigma_{km}(\tau) \subset K
\qquad\text{for almost every }\tau\in I.
\end{align}
\end{prop}

\begin{proof}
The uniform $H^1$ bound in Theorem~\ref{thm:compact-attractor}, the $H^1$-unitarity of $e^{it\Delta}$,
and \eqref{eq:subcritical-sobolev-embedding} give
\begin{align*}
 \|\Sigma_{n,km}(\tau)\|_{\mathrm{TV}} \lesssim \|\nabla u(e^{s_n+\tau})\|_{L^2}^2 +\|\nabla v(e^{s_n+\tau})\|_{L^2}^2 +\|u(e^{s_n+\tau})\|_{L^{2(d+1)/(d-1)}}^{2(d+1)/(d-1)} \lesssim1.
\end{align*}
On each bounded interval $I_0$, integration against $\Sigma_{n,km}(\tau)$ defines a uniformly bounded functional on $L^1(I_0;C(X))$.
So we have a subsequential weak-star limit, and the representation of the dual gives signed measures $\Sigma_{km}(\tau)$, defined for almost every $\tau$,
whose pairings with functions in $C(X)$ are measurable and whose total variations have the same uniform bound.
A diagonal extraction over the components and over an increasing sequence of intervals gives \eqref{eq:stress-measure-weakstar-convergence}.

For $\chi\in C_c^\infty(\R^d)$, use the current pairing $B_{\chi,k}$ defined in \eqref{eq:def-B-phik}. By the chain rule and \eqref{eq:log-current-derivative}, the function $B_{\chi,k}(s_n+\tau)$ is locally absolutely continuous in $\tau$ and has, for almost every $\tau$, the derivative given there with $\phi=\chi$ and $s=s_n+\tau$. 
By the definition of $\Sigma_{n,km}$, the stress term in that identity is
\begin{align*}
 \sum_{m=1}^d\int_X\partial_m\chi(y)\,d\Sigma_{n,km}(\tau)(y).
\end{align*}
By \eqref{eq:family-limit} and \eqref{eq:current-decouple}, the endpoint values $B_{\chi,k}(s_n+\gamma)$, $\gamma\in\{\alpha,\beta\}$, converge to the corresponding pairings with $J_k$. 
The first term on the right-hand side of \eqref{eq:log-current-derivative}, evaluated with $\phi=\chi$ and $s=s_n+\tau$, converges uniformly on bounded $\tau$-intervals to
\begin{align*}
 -\int_{\R^d} y\cdot\nabla\chi(y)\,dJ_k(\tau)(y).
\end{align*}
The last term in \eqref{eq:log-current-derivative}, with the same substitution, tends uniformly to zero by
\eqref{eq:auxiliary-L1-bounds}, since $e^{s_n+\tau}\to\infty$ uniformly on bounded $\tau$-intervals.
Integrating \eqref{eq:log-current-derivative} from $s_n+\alpha$ to $s_n+\beta$, changing variables $s=s_n+\tau$, 
and using \eqref{eq:stress-measure-weakstar-convergence} with $\Phi(\tau,y)=\mathbf1_{\{\alpha\leq\tau\leq\beta\}}(\tau)\partial_m\chi(y)$, we obtain \eqref{eq:momentum-difference}.

It remains to prove \eqref{eq:Sigma-km-supp}. 
Take $\tilde{\chi}\in C^\infty(I\times\R^d)$ that has bounded derivatives, and is constant in $y$ outside a fixed compact set, and vanishes on a neighborhood of $K$.
Then we will show
\begin{align}
\label{eq:off-atom-stress-localization}
 \sup_{\tau\in I} \|\tilde{\chi}(\tau,x/e^{s_n+\tau})r(e^{s_n+\tau})\|_{H^1} \to0.
\end{align}
If \eqref{eq:off-atom-stress-localization} is not true, we can choose $\tau_n\to\tau_*\in I$ along which the norm stays bounded away from zero.
Apply Theorem~\ref{thm:compact-attractor} at the times $e^{s_n+\tau_n}$ and pass to a subsequence on which all normalized profile centers converge in $X$.
The convergence \eqref{eq:family-limit}, the $d_J$-continuity in Theorem~\ref{thm:subsequential-family},
and Proposition~\ref{prop:subseq-atomicity} show that the limiting center of every nonzero profile belongs to $\supp\mu(\tau_*)$.
Multiplication by $\tilde{\chi}(\tau_n,x/e^{s_n+\tau_n})$ therefore sends each translated profile to zero in $H^1$ by dominated convergence.
The same is true of the $o_{H^1}(1)$ remainder in Theorem~\ref{thm:compact-attractor}; the term created by differentiating the cutoff has the additional factor $e^{-(s_n+\tau_n)}$.
This contradiction proves \eqref{eq:off-atom-stress-localization}.

Let $\Phi\in C(I\times X)$ be supported away from $K$. Choose $\tilde{\chi}$ satisfying the hypotheses of \eqref{eq:off-atom-stress-localization} and equal to one on a neighborhood of $\supp\Phi$.
The product rule gives, uniformly on $I$,
\begin{equation}
\label{eq:off-atom-gradient-localization}
\begin{aligned}
 \|\tilde{\chi}(\tau,x/e^{s_n+\tau})\nabla r(e^{s_n+\tau})\|_{L^2} &\leq \|\nabla\big(\tilde{\chi}(\tau,x/e^{s_n+\tau})r(e^{s_n+\tau})\big)\|_{L^2} +e^{-(s_n+\tau)}\|\nabla_y\tilde{\chi}\|_{L^\infty} \|r(e^{s_n+\tau})\|_{L^2} \to0.
\end{aligned}
\end{equation}
Since $u=v+r$, \eqref{eq:off-atom-gradient-localization} and the uniform $H^1$ bounds make the quadratic part of
\begin{align*}
 \int_X\Phi(\tau,y)\,d\Sigma_{n,km}(\tau)(y)
\end{align*}
tend to zero uniformly for $\tau\in I$. Its nonlinear part also tends to zero uniformly, because \eqref{eq:subcritical-sobolev-embedding} controls the localized $r$ term and Lemma~\ref{lem:free-lq-decay} gives
\begin{align*}
 \sup_{\tau\in I} \|v(e^{s_n+\tau})\|_{L^{2(d+1)/(d-1)}}\to0.
\end{align*}
After integration in $\tau$, \eqref{eq:stress-measure-weakstar-convergence} shows that the signed measure $\Sigma_{km}(\tau)\,d\tau$ on $I\times X$ annihilates every continuous test supported in $(I\times X)\setminus(I\times K)$, which proves \eqref{eq:Sigma-km-supp}.
\end{proof}

Now we show that the asymptotic velocities at which $(\nu_t,J_t)$ are concentrated cannot be arbitrarily large.
Conceptually this is similar to Lemma~\ref{lem:exterior-cone}.

\begin{prop}
\label{prop:supp-mu-bound}
There is a constant $R_u<\infty$, depending only on $u$, such that every subsequential limit family satisfies
\begin{align}
\label{eq:uniform-support-bound}
 \supp\mu(\tau)\subseteq\{y\in\R^d:|y|\leq R_u\} \qquad(\tau\in\R).
\end{align}
In addition, \eqref{eq:uniform-support-bound} also holds with $\supp\mu(\tau)$ replaced by $\supp J_k(\tau)$, $1\leq k \leq d$. 
\end{prop}

\begin{proof}

Fix $\tau$ and write the atomic masses of $\mu(\tau)$ as $a_\ell$. We prove \eqref{eq:uniform-support-bound}. Since every $(\mu(\tau),J(\tau))$ belongs to $\calD_\sharp$ by Theorem~\ref{thm:subsequential-family},
\eqref{eq:atom-mass-and-count} gives
\begin{align*}
 a_\ell\geq a_*:=\frac{m_\sharp\Theta}{R_\sharp}>0
\end{align*}
for every atom of every subsequential limit family. Let $\vartheta$ be the cutoff in Lemma~\ref{lem:exterior-cone}.
Choose $R<\infty$ so large that
\begin{align*}
 \sup_{t\geq1}\|\vartheta(x/(Rt))r(t)\|_{L^2}^2<\frac{a_*}{2}.
\end{align*}
Regard $\vartheta(y/R)^2$ as a function on $X$ that is continuous in $y$, with value one at $\infty$. For each fixed $\tau$,
\eqref{eq:family-limit} gives
\begin{align*}
 \int_X\vartheta(y/R)^2\,d\mu(\tau)(y) =\lim_{n\to\infty}\|\vartheta(x/(Rt_{n,\tau}))r(t_{n,\tau})\|_{L^2}^2 \leq\frac{a_*}{2}.
\end{align*}
Because $\vartheta(y/R)=1$ when $|y|\geq2R$, no atom of mass at least $a_*$ can lie there. Thus \eqref{eq:uniform-support-bound} holds with $R_u=2R$.
The assertion for $J_k(\tau)$ follows from $\supp J_k(\tau)\subseteq\supp\mu(\tau)$ in Theorem~\ref{thm:subsequential-family}.

\end{proof}

We record the atomic form of the limiting mass measure.

\begin{prop}
\label{prop:mass-limit-form}
Let the setting and notation be as in Theorem~\ref{thm:subsequential-family}. Assume that $\nu_t\rightharpoonup\mu_\infty$ weakly on $X$ as $t\to\infty$.
Then there are $L\geq1$, distinct $y_1,\ldots,y_L\in\R^d$, and numbers $a_\ell>0$ such that
\begin{align*}
 \mu_\infty=\sum_{\ell=1}^La_\ell\delta_{y_\ell}.
\end{align*}
\end{prop}

\begin{proof}
Let $t_n\to\infty$ be arbitrary and set $s_n=\log t_n$. After passing to a subsequence, Theorem~\ref{thm:subsequential-family} gives a subsequential limit family $(\mu(\tau),J(\tau))$.
At $\tau=0$, the assumed convergence of $\nu_t$ and \eqref{eq:family-limit} give $\mu(0)=\mu_\infty$.
In addition, Theorem~\ref{thm:subsequential-family} also gives that $\supp \mu(0)$ is in a finite subset of $\R^d$, which gives the stated representation.
\end{proof}

\section{Uniform in time estimates}
\label{sec:commutators}

We prove some upgrading estimates in this section.
The upgrading here is from estimates at a sequence of times to estimates on a sequence of time intervals.
We begin with some computations preparing for our `propagation estimates'.
Here we take $u\in C_c^\infty(\R^d;\C)$; it will eventually be a truncated version of our function.
\begin{lem}
\label{lem:At-commutator}
Let $b\in C_c^\infty(\R^d;\R^d)$, $c\in C_c^\infty(\R^d;\R)$, and define $a(y,\xi)=b(y)\cdot\xi+c(y)$.
For $t>0$, set $A_t=\Op^w(a(x/t,\xi))$. For $u\in C_c^\infty(\R^d;\C)$, define $M_uf=|u|^{4/(d-1)}f$. Then
\begin{align}
\label{eq:AtMu-commutator}
 i\la (A_tM_u-M_uA_t)u,u\ra =-\frac{2}{d+1}t^{-1} \int_{\R^d}(\nabla\cdot b)(x/t)|u|^{2(d+1)/(d-1)}\,dx,
\end{align}
and this functional has a unique continuous extension to $H^1$.
\end{lem}

\begin{proof}
For $f\in H^1(\R^d)$, the definition of Weyl quantization gives
\begin{align}
\label{eq:At-computation}
 A_tf=-ib(x/t)\cdot\nabla f -\frac{i}{2t}(\nabla\cdot b)(x/t)f+c(x/t)f.
\end{align}
The function $M_uu=|u|^{4/(d-1)}u$ belongs to $H^1(\R^d)$. Because $b$, $c$, and $|u|^{4/(d-1)}$ are real-valued,
integration by parts shows that $A_t$ is symmetric as a first-order form on $H^1$ and that $M_u$ is symmetric.
Consequently,
\begin{align*}
 \la (A_tM_u-M_uA_t)u,u\ra =\la M_uu,A_tu\ra -\overline{\la M_uu,A_tu\ra}.
\end{align*}
Using \eqref{eq:At-computation} with $f=u$, we obtain
\begin{align*}
 i\la (A_tM_u-M_uA_t)u,u\ra
 &=2\Re\la iM_uu,A_tu\ra\\
 &=-\int_{\R^d}b(x/t)\cdot |u|^{4/(d-1)}\nabla(|u|^2)\,dx\\
 &\quad-t^{-1}\int_{\R^d}(\nabla\cdot b)(x/t)|u|^{2(d+1)/(d-1)}\,dx.
\end{align*}
Since $|u|^{4/(d-1)}\nabla(|u|^2)=\frac{d-1}{d+1}\nabla(|u|^{2(d+1)/(d-1)})$, integration by parts proves \eqref{eq:AtMu-commutator}.

Using $\big||z|^{2(d+1)/(d-1)}-|w|^{2(d+1)/(d-1)}\big| \leq C_d\big(|z|^{(d+3)/(d-1)}+|w|^{(d+3)/(d-1)}\big)|z-w|$ and H\"older's inequality, we have
\begin{align*}
 \Big|\int_{\R^d}(\nabla\cdot b)(x/t) \big(|u|^{2(d+1)/(d-1)}-|\widetilde{u}|^{2(d+1)/(d-1)}\big)\,dx\Big| &\leq C_d\|\nabla\cdot b\|_{L^\infty} \big(\|u\|_{L^{2(d+1)/(d-1)}}^{(d+3)/(d-1)}\\
 &\qquad +\|\widetilde{u}\|_{L^{2(d+1)/(d-1)}}^{(d+3)/(d-1)}\big) \|u-\widetilde{u}\|_{L^{2(d+1)/(d-1)}}.
\end{align*}
Combining this estimate with the Sobolev embedding in \eqref{eq:subcritical-sobolev-embedding} shows that the right-hand side of \eqref{eq:AtMu-commutator} is continuous in the $H^1$ norm, and the conclusion follows.
\end{proof}

Now we compute the time derivative of the action $D_A$ below.
This will be used in Proposition~\ref{prop:Aq-no-drop} and eventually this will be used to approximate $L_q$ in \eqref{eq:def-Lq}.

\begin{prop}
\label{prop:DA,DA'}
Let $u_+\in H^1(\R^d)$, and let $v(t)$ and $r(t)$ be as in \eqref{eq:v,r-def}.
Let
\begin{align*}
 b\in C_c^\infty(\R^d;\R^d),\qquad c\in C_c^\infty(\R^d;\R),\qquad a(y,\xi)=b(y)\cdot\xi+c(y),
\end{align*}
and define $V=(2\xi-y)\cdot\nabla_y$. Let $A_t=\Op^w(a(x/t,\xi))$, which is the same as in \eqref{eq:At-computation}. 
Then
\begin{align}
\label{eq:DA-def}
 D_A(t)=\la A_tu(t),u(t)\ra -\la A_tv(t),v(t)\ra
\end{align}
is locally absolutely continuous for $t\geq1$ and satisfies
\begin{align}
\label{eq:DA-derivative}
 \begin{aligned}
 D_A'(t) & = t^{-1}\la\Op^w((Va)(x/t,\xi))r(t),r(t)\ra + 2t^{-1}\Re\la \Op^w((Va)(x/t,\xi))r(t),v(t)\ra \\
 &\quad-\frac{2}{d+1}t^{-1} \int_{\R^d}(\nabla\cdot b)(x/t)|u(t,x)|^{2(d+1)/(d-1)}\,dx
 \end{aligned}
\end{align}
for almost every $t\geq1$. The first two pairings on the right-hand side of \eqref{eq:DA-derivative} are $H^{-1}$--$H^1$ pairings.
\end{prop}

\begin{proof}
Integration by parts yields
\begin{align*}
 \la A_tf,f\ra =\int_{\R^d}b(x/t)\cdot\Im(\overline{f}\nabla f)\,dx +\int_{\R^d}c(x/t)|f|^2\,dx.
\end{align*}
Consequently,
\begin{align}
\label{eq:At-first-order-bound}
|\la A_tf,f\ra| &\leq\|b\|_{L^\infty}\|f\|_{L^2}\|\nabla f\|_{L^2} +\|c\|_{L^\infty}\|f\|_{L^2}^2  
 \lesssim \|f\|_{H^1}^2.
\end{align}

We next prove \eqref{eq:DA-derivative}. Using
\begin{align*}
 \partial_tA_t =-t^{-1}\Op^w((y\cdot\nabla_ya)(x/t,\xi)), \quad i\big((-\Delta)A_t-A_t(-\Delta)\big) =2t^{-1}\Op^w((\xi\cdot\nabla_ya)(x/t,\xi)),
\end{align*}
we have
\begin{align}
\label{eq:heisenberg-identity}
 \partial_tA_t+i\big((-\Delta)A_t-A_t(-\Delta)\big)=t^{-1}B_t,
\end{align}
where $B_t=\Op^w((Va)(x/t,\xi))$. Since
\begin{align}
\label{eq:Va-computation}
 Va= & \xi^T(Db+Db^T)\xi +(2\nabla c-Db\,y)\cdot\xi-y\cdot\nabla c,
\end{align}
$B_t$ is a second-order pseudodifferential operator with symbol bounds uniform in $t\geq 1$.
Thus $B_t:H^1\to H^{-1}$ is bounded uniformly in $t$ and $\la B_t \cdot, \cdot \ra$ is symmetric. In particular, 
\begin{align}
\label{eq:Bt-form-bound}
 |\la B_tf,g\ra| \lesssim \|f\|_{H^1}\|g\|_{H^1}, \qquad f,g\in H^1(\R^d),\quad t \geq 1.
\end{align} 
So we can group and ungroup pairings below. 

A direct computation shows
\begin{align} 
\label{eq:DA'}
 D_A'(t) &=t^{-1}\big(\la B_tu(t),u(t)\ra -\la B_tv(t),v(t)\ra\big) -\frac{2}{d+1}t^{-1} \int_{\R^d}(\nabla\cdot b)(x/t)|u(t,x)|^{2(d+1)/(d-1)}\,dx .
\end{align}

Finally, since $u=r+v$, the symmetric property of $B_t$ gives
\begin{align*}
 \la B_tu,u\ra-\la B_tv,v\ra =\la B_tr,r\ra+2\Re\la B_tr,v\ra.
\end{align*}
Substituting this into \eqref{eq:DA'} proves \eqref{eq:DA-derivative}.
\end{proof}

Now we justify the following intuition.
Suppose that $\nu_t$, which has density $|r|^2$, converges to a measure $\mu$ with $\supp\mu\subset K$ and that $\supp\chi$ is separated from $K$.
Then $\chi(x/t)r$ should converge to $0$.

\begin{prop}
\label{prop:interval-evacuation}
Let $r$ be defined by \eqref{eq:v,r-def}, and let $\nu_t$ be the residual mass measure defined in \eqref{eq:nu_t,J_t-def}. Let $1\leq A_n\leq B_n$ and $A_n \to \infty$. Suppose there is a fixed closed set $K\subset\R^d$ such that whenever $n_k\to\infty$, $A_{n_k}\leq t_k\leq B_{n_k}$,
and $\nu_{t_k}\rightharpoonup\mu$ weakly on $X=\R^d\cup\{\infty\}$, one has $\supp\mu\subset K$. If $\chi\in C^1(\R^d)$,
$\chi,\nabla\chi\in L^\infty$, and $\dist(\supp\chi,K)>0$, then
\begin{align}
\label{eq:interval-evacuation-uniform}
 \sup_{A_n\leq t\leq B_n}\|\chi(x/t)r(t)\|_{H^1}\to0.
\end{align}
\end{prop}

\begin{proof}
If \eqref{eq:interval-evacuation-uniform} fails, then, after passing to a subsequence, there are $\epsilon>0$, indices $n_k\to\infty$,
and $A_{n_k}\leq t_k\leq B_{n_k}$ such that
\begin{align*}
 \|\chi(x/t_k)r(t_k)\|_{H^1}\geq\epsilon.
\end{align*}
Since $A_{n_k}\to\infty$, we have $t_k\to\infty$. By Proposition~\ref{prop:subseq-atomicity}, after passing to a further subsequence, there are $w_1,\ldots,w_J\in H^1$,
centers $x_{j,k}$, points $y_j\in X$, and $e_k\to0$ in $H^1$ such that
\begin{align}
\label{eq:interval-evacuation-profile-decomposition}
 r(t_k) =\sum_{j=1}^Jw_j(\cdot-x_{j,k})+e_k, \quad  \frac{x_{j,k}}{t_k}&\to y_j, \quad \nu_{t_k} \rightharpoonup \sum_{j=1}^JM(w_j)\delta_{y_j}.
\end{align}
The hypothesis on weak limits therefore implies $w_j\ne0\Longrightarrow y_j\in K$. For every such $j$ and fixed $z\in\R^d$,
\begin{align*}
 \frac{x_{j,k}}{t_k}+\frac{z}{t_k}\to y_j, \qquad \chi\Big(\frac{x_{j,k}}{t_k}+\frac{z}{t_k}\Big)\to0,
\end{align*}
where the second limit follows from $\dist(\supp\chi,K)>0$.
The product rule gives
\begin{align}
\label{eq:gradient-chi-wj}
 \nabla(\chi(x/t)w_j) =\chi(x/t)\nabla w_j+t^{-1}(\nabla\chi)(x/t)w_j,
\end{align}
The term containing $\nabla\chi$ is bounded by $t_k^{-1}\|\nabla\chi\|_{L^\infty}\|w_j\|_{L^2}$.
Hence we have
\begin{align}
\label{eq:interval-evacuation-profile-vanishing}
 \|\chi(x/t_k)w_j(\cdot-x_{j,k})\|_{H^1}\to0.
\end{align}
Similarly, \eqref{eq:gradient-chi-wj} with $w_j$ replaced by $e_k$ gives
\begin{align}
\label{eq:interval-evacuation-error-vanishing}
 \|\chi(x/t_k)e_k\|_{H^1} \leq C(\|\chi\|_{L^\infty}+\|\nabla\chi\|_{L^\infty}) \|e_k\|_{H^1}\to0.
\end{align}
Applying \eqref{eq:interval-evacuation-profile-vanishing} and
\eqref{eq:interval-evacuation-error-vanishing} to
\eqref{eq:interval-evacuation-profile-decomposition} contradicts the lower bound $\epsilon$.
\end{proof}

Conversely, if \eqref{eq:uniform-chi0-r-tend-0} fails (or equivalently \eqref{eq:uniform-est-chi0-r-Lemma-version} holds), then Lemma~\ref{lem:first-entry} gives a time sequence along which the $H^1$ norm of localized $r$ is bounded below.
Now we upgrade this to an $L^2$ lower bound.

\begin{prop}
\label{prop:active-cells}
Let $d\geq5$, and let $u\in C(\R_{\geq 0};H^1(\R^d))$ solve
\begin{align*}
 (i\partial_t+\Delta)u=-|u|^{4/(d-1)}u
\end{align*}
and satisfy \eqref{eq:uniform-H1}.
Let $u_+\in H^1(\R^d)$ be the radiation state as in Theorem~\ref{thm:compact-attractor}, and let $r$ be defined by \eqref{eq:v,r-def}.
Let $\chi_0,\chi_1\in C_c^\infty(\R^d)$ satisfy
\begin{align*}
 \chi_1=1\qquad\text{on }\supp\chi_0.
\end{align*}
Suppose that $T_n\to\infty$,
\begin{align}
\label{eq:active-cells-sequential-smallness}
 \|\chi_1(x/T_n)r(T_n)\|_{H^1}\to0,
\end{align}
and that \eqref{eq:uniform-est-chi0-r-Lemma-version} holds. After passing to a subsequence and relabeling,
let $\epsilon>0$ and $s_n>T_n$ be the times given by Lemma~\ref{lem:first-entry}, so that
\begin{align}
\label{eq:active-cells-first-entry-level}
 \|\chi_0(x/s_n)r(s_n)\|_{H^1} &=\epsilon.
\end{align}
Then there are $m_*>0$ and $n_0$ such that
\begin{align}
\label{eq:active-cells-L2-lower}
 \|\chi_0(x/s_n)r(s_n)\|_{L^2} &\geq m_*, \qquad n\geq n_0.
\end{align}
\end{prop}

\begin{proof}
We first prove
\begin{align}
\label{eq:active-cells-frequency-tightness}
 \lim_{L\to\infty} \limsup_{n \to \infty} \|P_{>2L}(\chi_0(x/s_n)r(s_n))\|_{H^1}=0.
\end{align}
Here $P_{>L}$ and $P_{\leq L}$ denote the Fourier projections to $\{|\xi|>L\}$ and $\{|\xi|\leq L\}$,
respectively.

For $L\geq 1$, decompose
\begin{align}
\label{eq:P>2Lr-decomposition}
 P_{>2L}(\chi_0(x/s_n)r(s_n)) &=P_{>2L}(\chi_0(x/s_n)P_{>L}r(s_n)) +P_{>2L}(\chi_0(x/s_n)P_{\leq L}r(s_n)).
\end{align}
Multiplication by $\chi_0(x/s_n)$ is bounded on $H^1$ (for $s_n \geq 1$). So we have
\begin{align}
\label{eq:active-cells-high-input}
 \|P_{>2L}(\chi_0(x/s_n)P_{>L}r(s_n))\|_{H^1} \leq C_{\chi_0}\|P_{>L}r(s_n)\|_{H^1}.
\end{align}

For the second term on the right-hand side of \eqref{eq:P>2Lr-decomposition}, 
we take Fourier transform.
Then we write the Fourier transform of $P_{>2L}(\chi_0(x/s_n)P_{\leq L}r(s_n))$ as a convolution.
If the total frequency satisfies $|\xi|>2L$ and the frequency in $\widehat{P_{\leq L}r(s_n)}$ satisfies $|\eta|\leq L$, then $|\xi-\eta|>L$.
So we have
\begin{align} 
\label{eq:active-cells-low-input}
 \|P_{>2L}(\chi_0(x/s_n)P_{\leq L}r(s_n))\|_{H^1} &\leq C  \|P_{\leq L}r(s_n)\|_{H^1} \int_{|\zeta|>s_nL}(1+|\zeta|)|\widehat{\chi_0}(\zeta)|\,d\zeta .
\end{align}

Since $\chi_0\in C_c^\infty(\R^d)$, its Fourier transform is rapidly decreasing. Hence, for every fixed $L$,
\begin{align*}
 \int_{|\zeta|>s_nL} (1+|\zeta|)|\widehat{\chi_0}(\zeta)|\,d\zeta \to0.
\end{align*}
Since we have $\sup_{t\geq 0}\|r(t)\|_{H^1}<\infty$, combining \eqref{eq:active-cells-high-input} and \eqref{eq:active-cells-low-input} gives
\begin{align*}
 \limsup_{n\to\infty}\|P_{>2L}(\chi_0(x/s_n)r(s_n))\|_{H^1} \leq C_{\chi_0} \limsup_{n\to\infty}\|P_{>L}r(s_n)\|_{H^1},
\end{align*}
which gives \eqref{eq:active-cells-frequency-tightness} by Proposition~\ref{prop:residual-frequency-tightness}.

Choose $L$ sufficiently large and then $n_0$ sufficiently large so that
\begin{align*}
 \|P_{>2L}(\chi_0(x/s_n)r(s_n))\|_{H^1} \leq\frac{\epsilon}{2}, \qquad n\geq n_0.
\end{align*}
Since
\begin{align*}
 \chi_0(x/s_n)r(s_n)=P_{\leq2L}(\chi_0(x/s_n)r(s_n))+P_{>2L}(\chi_0(x/s_n)r(s_n))
\end{align*}
and $\|\chi_0(x/s_n)r(s_n)\|_{H^1}=\epsilon$, the triangle inequality gives
\begin{align*}
 \|P_{\leq2L}(\chi_0(x/s_n)r(s_n))\|_{H^1}\geq\frac{\epsilon}{2}.
\end{align*}
Since $1+|\xi|^2\leq1+4L^2$ on the Fourier support of $P_{\leq2L}$, we have
\begin{align*}
 \|P_{\leq2L}(\chi_0(x/s_n)r(s_n))\|_{H^1} &\leq\sqrt{1+4L^2}\|P_{\leq2L}(\chi_0(x/s_n)r(s_n))\|_{L^2} \leq\sqrt{1+4L^2}\|\chi_0(x/s_n)r(s_n)\|_{L^2}.
\end{align*}
It follows that
\begin{align*}
 \|\chi_0(x/s_n)r(s_n)\|_{L^2} \geq\frac{\epsilon}{2\sqrt{1+4L^2}} =:m_*>0
\end{align*}
for every $n\geq n_0$, which proves \eqref{eq:active-cells-L2-lower}.
\end{proof}

\section{The limiting measure and the proof of Theorem~\ref{thm:measure-convergence}}
\label{sec:limiting-measure}

We prove Theorem~\ref{thm:measure-convergence} in this section. We sketch the overall idea first.
Let $\Omega_\sharp$ be the set of subsequential limits of $(\nu_t,J_t)$ as $t\to\infty$ in Proposition~\ref{prop:connected-omega}, which is compact and connected.
Now the convergence of $(\nu_t,J_t)$ is equivalent to this set being a single point.
We first find a convex function $\Psi$ such that $\int \Psi d\mu$ (as a function of $\mu$) has a unique maximizer $\mu_*$.
Then we prove that elements of $\Omega_\sharp$ have the following rigidity: if the first component is $\mu_*$, then the momentum component has to be $J_*=\frac{y}{2}\mu_*$.
If the convergence fails, then there will be two time sequences such that $(\nu_t,J_t)$ is very close to $(\mu_*,J_*)$ along one sequence but stays a fixed positive distance from $(\mu_*,J_*)$ along the other.

Then we will show that the elements of $\Omega_\sharp$ have the following coercive property: if $(\mu,J)$ stays a fixed positive distance from $(\mu_*,J_*)$, then $\int\Psi\,d\mu$ is below $\int\Psi\,d\mu_*$ by a fixed positive amount.
Then we show that this is `repulsive' enough to make the deficit from $\int\Psi\,d\mu_*$ grow exponentially in the delayed time parameter $|\tau|$ from Theorem~\ref{thm:subsequential-family}.
This will lead to a contradiction.
The last step is difficult, so we will spend most of this section computing and estimating the currents involved in this step.

We will consider $(\nu_t,J_t)$ reparametrized on an exponential scale:
\begin{align}
\label{eq:Fs-def}
 F(s)=(\nu_{e^s},J_{e^s}).
\end{align}

\begin{prop}
\label{prop:convex-exposure}
Assume $F$ in \eqref{eq:Fs-def} does not converge as $s\to\infty$. There is a real polynomial $\Psi$,
a neighborhood $U$ of a common ball containing the supports of all subsequential limits, and $\lambda_\Psi>0$ such that
\begin{align*}
 z^TD^2\Psi(y)z\geq\lambda_\Psi|z|^2
\end{align*}
for $y\in U$, and
\begin{align*}
 \int\Psi\,d\mu
\end{align*}
has a unique maximizer $\mu_*$ over the projection of the set $\Omega_\sharp$ in Proposition~\ref{prop:connected-omega} to the mass component.
\end{prop}

\begin{proof}
By Proposition~\ref{prop:supp-mu-bound} and \eqref{eq:atom-mass-and-count}, for an element of $\Omega_\sharp$, the mass component (i.e., the $\mu$-component) is supported in one fixed compact ball and has at most
\begin{align*}
 N_{\max}=\big\lfloor\sqrt{R_\sharp/\Theta}\big\rfloor
\end{align*}
atoms. The coordinate monomials of total degree at most $2N_{\max}-1$ separate all such measures. Indeed,
the difference of two such measures has at most $2N_{\max}$ support points, and Lagrange-type polynomials of degree at most $2N_{\max}-1$ recover every signed weight.

Let $g_1,\ldots,g_A$ be those monomials and $\phi_0(y)=|y|^2$. Choose $C_j$ large enough that
\begin{align*}
 \phi_j=C_j\phi_0+g_j
\end{align*}
has positive-definite Hessian on a neighborhood of the support ball. The moment map
\begin{align*}
 L(\mu)=\Big(\int\phi_0\,d\mu,\ldots,\int\phi_A\,d\mu\Big)
\end{align*}
is continuous and injective on the compact set $\Omega_{\mathrm{mass}}$ formed by the projection of $\Omega_\sharp$ to the mass component.

For a compact set $C\subset\R^{A+1}$, define
\begin{align*}
 h_C(w)=\max_{x\in C}w\cdot x.
\end{align*}
This function is Lipschitz in $w$, hence differentiable almost everywhere.
If $x\in C$ maximizes $w\cdot x$, then
\begin{align*}
 h_C(w+z)\geq h_C(w)+z\cdot x \; \forall z\in\R^{A+1}.
\end{align*}
At a point where $h_C$ is differentiable, this inequality, applied to both $z$ and $-z$, forces $x=\nabla h_C(w)$. Thus the maximizer $x$ is unique.
Apply this with $C=L(\Omega_{\mathrm{mass}})$, choose such a vector $w=(w_0,\ldots,w_A)$ with every $w_j>0$, and set
\begin{align*}
 \Psi=\sum_{j=0}^Aw_j\phi_j.
\end{align*}
Because every $w_j$ is positive, the Hessian of $\Psi$ has a uniform positive lower bound on a neighborhood of the support ball.
The injectivity of $L$ then gives a unique measure maximizing $\int\Psi\,d\mu$ over $\Omega_{\mathrm{mass}}$.
\end{proof}

Now we prove the rigidity of the limiting measures in $\Omega_\sharp$: if the first component is $\mu_*$, then the momentum component is determined by it as well.
That is, $y$ is indeed the velocity, and the momentum equals one-half of the velocity times the mass.

\begin{prop}[Uniqueness at the maximizing measure]
\label{prop:exposed-rigidity}
Let $\Psi$ and
\begin{align*}
 \mu_*=\sum_{i=1}^La_i^*\delta_{x_i^*}, \qquad a_i^*>0,
\end{align*}
be as in Proposition~\ref{prop:convex-exposure}, with the $x_i^*$ distinct, and set
\begin{align}
\label{eq:Z*-def}
 Z_*=\Big(\mu_*,\frac{y}{2}\mu_*\Big).
\end{align}
If $(\mu(\tau),J(\tau))$ is a subsequential limit family from Theorem~\ref{thm:subsequential-family} and $\mu(0)=\mu_*$,
then
\begin{align}
\label{eq:exposed-rigidity-constant}
 (\mu(\tau),J(\tau))=Z_* \qquad(\tau\in\R).
\end{align}
Consequently, $Z_*$ is the only subsequential limit of $F$ in \eqref{eq:Fs-def} whose first component is $\mu_*$.
\end{prop}

\begin{proof}
Let $(\mu(\tau),J(\tau))$ be such a subsequential limit family, obtained along $s_n\to\infty$, and put
\begin{align*}
 H(\tau)=\int\Psi\,d\mu(\tau).
\end{align*}
For each fixed $\tau$, \eqref{eq:family-limit} shows that $(\mu(\tau),J(\tau))$ is a subsequential limit of $F$.
The maximizing property in Proposition~\ref{prop:convex-exposure} therefore gives
\begin{align*}
 H(\tau)\leq H(0) \qquad(\tau\in\R).
\end{align*}

Let $U$ and $\lambda_\Psi$ be as in Proposition~\ref{prop:convex-exposure}. Choose pairwise disjoint balls $U_i$ containing $x_i^*$,
with $\overline{U_i}\subset U$. The $d_J$-continuity in Theorem~\ref{thm:subsequential-family} and the atom-mass lower bound in Proposition~\ref{prop:atomic-and-bound} imply that,
for some $\epsilon>0$,
\begin{align}
\label{eq:exposed-rigidity-local-support}
 \supp\mu(\tau)\subset\bigcup_{i=1}^LU_i \qquad(|\tau|<\epsilon).
\end{align}
Indeed, otherwise an atom of uniformly positive mass would remain outside these balls along a sequence $\tau\to0$,
contradicting $\mu(\tau)\rightharpoonup\mu_*$.

Choose $\chi_i\in C_c^\infty(\R^d)$ equal to one on a neighborhood of $\overline{U_i}$ and zero near every $\overline{U_j}$ with $j\neq i$,
and define
\begin{align*}
 P_i(\tau)=\int\chi_i\,dJ(\tau), \qquad B_i(\tau)=\int y\chi_i(y)\,d\mu(\tau)(y).
\end{align*}
The support containment \eqref{eq:exposed-rigidity-local-support} and the support relation in Theorem~\ref{thm:subsequential-family} show that $\nabla\chi_i$ vanishes on the supports of both $\mu(\tau)$ and $J(\tau)$.
Thus \eqref{eq:mass-difference}, first with $\phi=\chi_i$, gives
\begin{align*}
 \int\chi_i\,d\mu(\tau)=a_i^* \qquad(|\tau|<\epsilon).
\end{align*}
Now we pass to the further subsequence given by Proposition~\ref{prop:stress-1}.
This does not change the subsequential limit family $(\mu(\tau),J(\tau))$ because of \eqref{eq:family-limit}.
By \eqref{eq:Sigma-km-supp}, we have
\begin{align}
\label{eq:exposed-rigidity-stress-support}
 \supp\Sigma_{km}(\tau)\subset\bigcup_{i=1}^L\overline{U_i}
\end{align}
for almost every $|\tau|<\epsilon$. Since $\nabla\chi_i$ vanishes on the union in \eqref{eq:exposed-rigidity-stress-support}, \eqref{eq:momentum-difference} shows that $P_i(\tau)$ is constant and we will denote it by $P_i$ below.
Applying \eqref{eq:mass-difference} with $\phi(y)=\chi_i(y)y_k$ for each coordinate $k$ gives
\begin{align*}
 B_i'=2P_i-B_i.
\end{align*}
Since $B_i(0)=a_i^*x_i^*$, it follows that
\begin{align}
\label{eq:cluster-barycenter-drift}
 B_i(\tau)-a_i^*x_i^* =(1-e^{-\tau})(2P_i-a_i^*x_i^*).
\end{align}

The Hessian bound in Proposition~\ref{prop:convex-exposure} gives, for $y\in U_i$,
\begin{align*}
 \Psi(y)\geq\Psi(x_i^*)+\nabla\Psi(x_i^*)\cdot(y-x_i^*) +\frac{\lambda_\Psi}{2}|y-x_i^*|^2.
\end{align*}
Integrating over the $U_i$, summing, using \eqref{eq:cluster-barycenter-drift}, and applying Cauchy--Schwarz yields
\begin{align}
\label{eq:exposed-rigidity-convexity-first-bound}
 H(\tau)-H(0) &\geq(1-e^{-\tau}) \sum_i\nabla\Psi(x_i^*)\cdot(2P_i-a_i^*x_i^*) +\frac{\lambda_\Psi}{2} \sum_i\int_{U_i}|y-x_i^*|^2\,d\mu(\tau)(y)\\
 &\geq(1-e^{-\tau}) \sum_i\nabla\Psi(x_i^*)\cdot(2P_i-a_i^*x_i^*) +\frac{\lambda_\Psi}{2}(1-e^{-\tau})^2 \sum_i\frac{|2P_i-a_i^*x_i^*|^2}{a_i^*} \qquad(|\tau|<\epsilon).\notag
\end{align}
Together with $H(\tau)\leq H(0)$ for both signs of $\tau$, this first makes the coefficient of $1-e^{-\tau}$ vanish and then gives
\begin{align*}
 2P_i=a_i^*x_i^* \qquad(1\leq i\leq L).
\end{align*}
Substituting $2P_i=a_i^*x_i^*$ into \eqref{eq:exposed-rigidity-convexity-first-bound} and using $H(\tau)\leq H(0)$ give
\begin{align*}
 0\geq H(\tau)-H(0) \geq\frac{\lambda_\Psi}{2} \sum_i\int_{U_i}|y-x_i^*|^2\,d\mu(\tau)(y).
\end{align*}
Hence $\mu(\tau)=\mu_*$ for $|\tau|<\epsilon$. The support relation in Theorem~\ref{thm:subsequential-family},
together with $P_i=a_i^*x_i^*/2$, shows that $P_i$ is the vector weight of $J(\tau)$ at $x_i^*$ and gives
\begin{align*}
 J(\tau)=\frac{y}{2}\mu_* \qquad(|\tau|<\epsilon).
\end{align*}

The set of $\tau$ for which $(\mu(\tau),J(\tau))=Z_*$ is closed by the $d_J$-continuity in Theorem~\ref{thm:subsequential-family}.
On the other hand, if $(\mu(\tau_0),J(\tau_0))=Z_*$, then \eqref{eq:family-limit} shows that the family $(\mu(\tau_0+\sigma),J(\tau_0+\sigma))$ is the subsequential limit family obtained along $s_n+\tau_0$,
so the preceding local argument applies at $\tau_0$. 
So this set of $\tau$ is also open. Hence $(\mu(\tau),J(\tau)) \equiv Z_*$.

Finally, let $(\mu_*,J)$ be any subsequential limit of $F$. By the definition in Proposition~\ref{prop:connected-omega},
there is a sequence $s_n\to\infty$ such that $F(s_n)\to(\mu_*,J)$.
After passing to a subsequence, Theorem~\ref{thm:subsequential-family} and \eqref{eq:family-limit} give a subsequential limit family $(\mu(\tau),J(\tau))$ with $(\mu(0),J(0))=(\mu_*,J)$.
Equation~\eqref{eq:exposed-rigidity-constant}, now proved, gives $(\mu_*,J)=Z_*$. At least one such point exists by the definition of $\mu_*$ in Proposition~\ref{prop:convex-exposure}.
\end{proof}

Now we approximate $\Psi$ by a simpler function that is almost piecewise linear.
The advantage of using this $q$ rather than using $\Psi$ above directly is that its second- and higher-order derivatives are easier to control.

\begin{prop}
\label{prop:q-construction}
Let $\Psi$ and $\mu_*$ be as in Proposition~\ref{prop:convex-exposure}, with $\mu_*=\sum_j a_j^*\delta_{x_j^*}$. There are disjoint balls $U_j$ about $x_j^*$,
a smooth convex function $q:\R^d\to\R$, and nested smooth cutoffs $\chi_0,\tilde{\theta},\chi_1$ such that:
\begin{enumerate}
\item For every $j$ and every $y\in U_j$,
\begin{align}
\label{eq:q(y)-local-expression}
 q(y)=\Psi(x_j^*)+\nabla\Psi(x_j^*)\cdot(y-x_j^*).
\end{align}
\item $D^2q\geq 0$, and $\nabla q$ and
$y\cdot\nabla q(y)-q(y)$, together with all their derivatives, are bounded. Moreover, $D^2q=0$ outside $\{\chi_0=1\}$,
$\tilde{\theta}=1$ near the supports of $\chi_0,\nabla\chi_0,\Delta\chi_0$, $\chi_1=1$ near $\supp\tilde{\theta}$, and $y\cdot\nabla\chi_0(y)$ is bounded.
\item $\dist(\supp\chi_1,\bigcup_j\overline{U_j})>0$.
\end{enumerate}
\end{prop}

\begin{proof}
Let $U, \lambda_\Psi>0$ be as in Proposition~\ref{prop:convex-exposure}. In addition, we can choose it so that the line segment joining any two atoms $x_j^*,x_k^*$ lies in $U$.
For each atom $x_j^*$, define the affine function
\begin{align*}
\ell_j(y)&=\Psi(x_j^*)+\nabla\Psi(x_j^*)\cdot(y-x_j^*).
\end{align*}
If $j\neq k$, as $\Psi$ is strictly convex (see Proposition~\ref{prop:convex-exposure}), we have
\begin{align*}
\ell_j(x_j^*)-\ell_k(x_j^*) &=\Psi(x_j^*)-\Psi(x_k^*) -\nabla\Psi(x_k^*)\cdot(x_j^*-x_k^*)\\
&=\int_0^1(1-s)(x_j^*-x_k^*)^T D^2\Psi(x_k^*+s(x_j^*-x_k^*)) (x_j^*-x_k^*)\,ds\\
&\geq\frac{\lambda_\Psi}{2}|x_j^*-x_k^*|^2>0.
\end{align*}
Since there are only finitely many atoms, we may choose $r_j>0$ such that the balls $B(x_j^*,2r_j)$ are pairwise disjoint,
are contained in $U$, and satisfy
\begin{align*}
\ell_j(y)&>\ell_k(y) \qquad\text{for }y\in B(x_j^*,2r_j),\ k\neq j.
\end{align*}

Define
\begin{align*}
m(y)&=\max_j\ell_j(y).
\end{align*}
Then $m$ is convex and globally Lipschitz, and $m =\ell_j$ on $B(x_j^*,2r_j)$.
Choose a nonnegative radially symmetric function $\rho_\epsilon \in C_c^\infty(\R^d)$ such that
\begin{align*}
\supp\rho_\epsilon&\subset B(0,\epsilon),\qquad \int_{\R^d}\rho_\epsilon(z)\,dz=1,\qquad 0<\epsilon<\min_jr_j,
\end{align*}
and define
\begin{align}
\label{eq:q-def}
q = m*\rho_\epsilon.
\end{align}
Because $m$ is convex and $\rho_\epsilon\geq 0$, $q$ is smooth and $D^2q \geq 0$.

Suppose that $y\in B(x_j^*,r_j)$ and $z\in\supp\rho_\epsilon$. Then $y-z\in B(x_j^*,2r_j)$, so $m(y-z)=\ell_j(y-z)$.
Since $\rho_\epsilon$ is even,
\begin{align*}
q(y) &=\int_{\R^d}\ell_j(y-z)\rho_\epsilon(z)\,dz =\ell_j(y)-\nabla\Psi(x_j^*)\cdot \int_{\R^d}z\rho_\epsilon(z)\,dz =\ell_j(y).
\end{align*}
Thus
\begin{align*}
q(y)&=\Psi(x_j^*)+\nabla\Psi(x_j^*)\cdot(y-x_j^*) \qquad\text{for }y\in B(x_j^*,r_j).
\end{align*}

Since $m$ is Lipschitz, its weak gradient is bounded. Therefore
\begin{align*}
\nabla q =(\nabla m)*\rho_\epsilon, \quad \partial^\alpha\nabla q =(\nabla m)*\partial^\alpha\rho_\epsilon.
\end{align*}
Hence $\nabla q$ and all its derivatives are bounded.

At almost every point $w$, one of the affine functions $\ell_j$ realizes the maximum defining $m$ and
\begin{align*}
w\cdot\nabla m(w)-m(w) &=x_j^*\cdot\nabla\Psi(x_j^*)-\Psi(x_j^*).
\end{align*}
The right side takes only finitely many values. Hence $w\cdot\nabla m(w)-m(w)$ is bounded almost everywhere.
By the definition of $q$ in \eqref{eq:q-def}, we have
\begin{align*}
y\cdot\nabla q(y)-q(y) &=\int_{\R^d} ((y-z)\cdot\nabla m(y-z)-m(y-z)) \rho_\epsilon(z)\,dz +\int_{\R^d} z\cdot\nabla m(y-z)\rho_\epsilon(z)\,dz.
\end{align*}
The first term is the convolution of a bounded function with $\rho_\epsilon$. The second is a finite sum of convolutions of the bounded functions $\partial_\nu m$ with the compactly supported smooth functions $z_\nu\rho_\epsilon(z)$.
Differentiation in $y$ only falls on the smooth one, and we know that $y\cdot\nabla q(y)-q(y)$ and all its derivatives are bounded.

We now construct the cutoffs. Set
\begin{align*}
U_j&=B(x_j^*,r_j/20).
\end{align*}
Choose smooth cutoffs $\chi_1,\tilde{\theta},\chi_0:\R^d\to\R$ satisfying $0\leq\chi_1,\tilde{\theta},\chi_0\leq1$ and the following prescribed values:
\begin{enumerate}
\item $\chi_1(y)=0$ if $|y-x_j^*|\leq r_j/10$ for at least one $j$, and $\chi_1(y)=1$ if $|y-x_j^*|\geq r_j/5$ for every $j$.
\item $\tilde{\theta}(y)=0$ if $|y-x_j^*|\leq r_j/4$ for at least one $j$, and $\tilde{\theta}(y)=1$ if $|y-x_j^*|\geq r_j/2$ for every $j$.
\item $\chi_0(y)=0$ if $|y-x_j^*|\leq3r_j/5$ for at least one $j$, and $\chi_0(y)=1$ if $|y-x_j^*|\geq4r_j/5$ for every $j$.
\end{enumerate}
Then one can verify that all desired properties are satisfied.
\end{proof}

In the proof of the convergence of $(\nu_t,J_t)$ in \eqref{eq:nu_t,J_t-def}, we will again use the approximation strategy.
That is, we use $j(u)-j(v)$ to approximate $j(r)$ and use $|u|^2-|v|^2$ to approximate $|r|^2$.
So we consider
\begin{align}
\label{eq:Aq-def}
 A_q(t) & = \int \nabla q(x/t)\cdot\big(j(u)-j(v)\big)\,dx -\frac{1}{2} \int \Big(\frac{x}{t}\cdot\nabla q(x/t)-q(x/t)\Big) (|u|^2-|v|^2)\,dx,
\end{align}
which is $D_A$ in Proposition~\ref{prop:DA,DA'} with $b= \nabla q(y)$ , $c=-\frac{1}{2} (y\cdot \nabla q(y)-q(y))$. Here $q$ is the smooth convex function as in Proposition \ref{prop:q-construction} above.  
Then we show that $A_q$ is `non-decreasing' in the following sense.

\begin{prop}
\label{prop:Aq-no-drop}
Let $u_+\in H^1(\R^d)$, and let $v(t)$ and $r(t)$ be as in \eqref{eq:v,r-def}. Let $q,A_q$ be as above. 
Let $\chi_0,\tilde{\theta},\chi_1$ be as in Proposition~\ref{prop:q-construction}, with $D^2q=0$ outside $\{\chi_0=1\}$. If $A_n\to\infty$, $A_n<B_n$, and
\begin{align*}
 \sup_{A_n\leq t\leq B_n} \|\chi_1(x/t)r(t)\|_{H^1}\to0,
\end{align*}
then
\begin{align*}
 \liminf_{n\to\infty}\big(A_q(B_n)-A_q(A_n)\big)\geq 0.
\end{align*}
\end{prop}
\begin{proof}
In the notation of Proposition~\ref{prop:DA,DA'}, we have
\begin{align*}
a = \nabla q(y)\cdot\xi-\frac{y\cdot\nabla q(y)-q(y)}{2}.
\end{align*}
Our goal is to prove that the right-hand side of \eqref{eq:DA-derivative} is asymptotically nonnegative, up to a decaying negative term, as $t\to\infty$.
We still use $V=(2\xi-y)\cdot\nabla_y$ and we have
\begin{align*}
 Va = V\big( \nabla q(y)\cdot\xi-\frac{y\cdot\nabla q(y)-q(y)}{2} \big) 
 =\frac{1}{2}(2\xi-y)^TD^2q(y)(2\xi-y) \geq 0,
\end{align*}
where we used the convexity $D^2q \geq 0$. 

Now we put things into the context of Proposition~\ref{prop:DA,DA'}, where coefficients are required to have compact support \footnote{Another way is to relax the condition of Proposition~\ref{prop:DA,DA'} to only requiring bounds on $b,c$ for the proof to go through.}, by truncating objects above.
Choose a bump function $\chi\in C_c^\infty(\R^d)$, nonincreasing in $|y|$, such that $\chi=1$ on $\{|y|\leq1\}$ and $\chi=0$ on $\{|y|\geq2\}$, and set
\begin{align*}
\chi_R(y)=\chi(y/R).
\end{align*}
Applying Proposition~\ref{prop:DA,DA'} with $b=\chi_R(y)\nabla q(y)$ and $c=-\frac{1}{2}\chi_R(y)(y\cdot\nabla q(y)-q(y))$ gives the corresponding truncated $A_q$:
\begin{align}
\label{eq:truncated-convex-action}
A_{q,R}(t) &=\int\chi_R(x/t)\nabla q(x/t)\cdot\big(j(u)-j(v)\big)\,dx -\frac{1}{2}\int\chi_R(x/t) (\frac{x}{t}\cdot\nabla q(x/t)-q(x/t)) (|u|^2-|v|^2)\,dx.
\end{align}
Then the corresponding truncated `$Va$' in \eqref{eq:Va-computation} is
\begin{align}
\label{eq:Va-3}
\begin{aligned}
&V(\chi_R(y) (\nabla q(y)\cdot\xi-\frac{y\cdot\nabla q(y)-q(y)}{2}))\\
& = \frac{1}{2} \chi_R(y)(2\xi-y)^TD^2q(y)(2\xi-y) + 
((2\xi-y)\cdot\nabla\chi_R(y)) (\nabla q(y)\cdot\xi-\frac{y\cdot\nabla q(y)-q(y)}{2}).
\end{aligned}
\end{align}
Now we consider the error term introduced by $\chi_R$. The second term on the right side of \eqref{eq:Va-3} is a polynomial of degree at most two in $\xi$ whose coefficients are $O(R^{-1})$ and supported in $\{R \leq |y| \leq 2R\}$.
As in the derivation of \eqref{eq:Bt-form-bound}, its quantization is a pseudodifferential operator with $O(R^{-1})$ symbols.
The corresponding quadratic form is bounded by the $H^1$ norms of $u$ and $v$ over $\{R\leq|x/t|\leq2R\}$, with times $O(R^{-1})$ level constant.
Similarly, using \eqref{eq:Aq-def} and \eqref{eq:truncated-convex-action}, we know $A_{q,R}(t)\to A_q(t)$ as $R \to \infty$ for every fixed $t$.

The first term on the right side of \eqref{eq:Va-3} is nonnegative on the symbol level, and now we consider its operator level positivity after quantization.
The symbol $2\xi-y$ quantizes to be $L_t=2D-x/t$, where $D=-i\nabla$. 
We have
\begin{align} \label{eq:weyl-quant-2order}
\begin{aligned}
&\la \Op^w\big( \frac{1}{2}\chi_R(x/t)(2\xi-x/t)^TD^2q(x/t)(2\xi-x/t) \big)f,f \ra
\\ &\quad= \frac{1}{2}\int\chi_R(x/t)(L_tf)^*D^2q(x/t)(L_tf)\,dx
-\frac{1}{2t^2}\int \sum_{j,k}\big(\partial_{y_j}\partial_{y_k}(\chi_R(D^2q)_{jk})\big)(x/t)|f|^2\,dx.
\end{aligned}
\end{align}
The first term is positive, and the second term is $O(t^{-2})$,  hence integrable.

The first integral on the right-hand side of \eqref{eq:weyl-quant-2order} is non-decreasing in $R$, and we define (allowed to be infinity):
\begin{align} 
\label{eq:K-def}
K(f,t) =\lim_{R\to\infty}\int \chi_R(x/t)(L_tf)^*D^2q(x/t)(L_tf)\,dx.
\end{align}
The strategy for the remaining part is as follows: we first take $f=r$, which gives the term we currently have.
If we add another multiple of $K(v,t)$, then these two terms control the second term in \eqref{eq:DA-derivative}.
So the negative part we are left to control is $K(v,t)$ and we show that it is integrable in time below. In particular, this shows that it has tail tending to zero.

More concretely, let $A_t$ be as in Proposition~\ref{prop:DA,DA'} with $\chi_R$ inserted, and consider
\begin{align}
\label{eq:Av-pairing}
\la A_t v(t), v(t) \ra = \int\chi_R(x/t)\nabla q(x/t)\cdot j(v(t))\,dx -\frac{1}{2}\int \chi_R(x/t)(\frac{x}{t}\cdot\nabla q(x/t)-q(x/t))|v(t)|^2 \, dx.
\end{align}

Let $B_t=\Op^w((Va)(x/t,\xi))$ with $a$ as above.
Since $(i\partial_t+\Delta)v=0$, using \eqref{eq:heisenberg-identity}, we have
\begin{align}
\label{eq:Atvv'}
\frac{d}{dt}\la A_tv(t),v(t)\ra =t^{-1}\la B_tv(t),v(t)\ra,
\end{align}
which is the second term in the bracket on the right-hand side of \eqref{eq:DA'}.
The right-hand side of \eqref{eq:Atvv'} is evaluated in \eqref{eq:weyl-quant-2order}.

Now we integrate this from $1$ to $T \gg 1$ and let $R\to\infty$.
Terms containing derivatives of $\chi_R$ tend to zero as the discussion after \eqref{eq:Va-3} when $R \to \infty$.
Then we have
\begin{align*}
& \big( \int\nabla q(x/t)\cdot j(v(t))\,dx -\frac{1}{2}\int (\frac{x}{t}\cdot\nabla q(x/t)-q(x/t))|v(t)|^2\,dx \big)\big|_{t=1}^{t=T}
\\ &\quad=\frac{1}{2}\int_1^T K(v,t)\frac{dt}{t} -\frac{1}{2}\int_1^T\int (\Delta^2q)(x/t)|v(t,x)|^2\,dx\frac{dt}{t^3}.
\end{align*}
The left-hand side is bounded uniformly in $T$, because $\nabla q$ and $y\cdot\nabla q-q$ are bounded by the construction in Proposition~\ref{prop:q-construction} and $\|v(t)\|_{H^1}$ is preserved.
The last integral is bounded by $\|\Delta^2q\|_{L^\infty}\|u_+\|_{L^2}^2\int_1^\infty t^{-3}\,dt$.
Therefore
\begin{align}
\label{eq:free-convex-kinetic-integrability}
\int_1^\infty K(v,t)\frac{dt}{t}<\infty.
\end{align}

We now return to apply \eqref{eq:DA-derivative} to $A_{q,R}$ (playing the role of $D_A$ there) from \eqref{eq:truncated-convex-action} to obtain
\begin{align}
\label{eq:truncated-Aq'}
\begin{aligned}
A_{q,R}'(t) &= t^{-1}\la \Op^w((Va)(x/t,\xi))r(t),r(t) \ra +2t^{-1}\Re\la \Op^w((Va)(x/t,\xi))r(t),v(t) \ra\\
&\quad-\frac{2}{d+1}t^{-1}\int (\chi_R\Delta q+\nabla\chi_R\cdot\nabla q)(x/t)|u(t,x)|^{2(d+1)/(d-1)}\,dx,
\end{aligned}
\end{align}
with $Va$ given by \eqref{eq:Va-3}.
The second term of \eqref{eq:Va-3} enters the first two terms of \eqref{eq:truncated-Aq'} only through integrals over the region $\{R \leq|x/t|\leq 2R\}$ and has the derivative of $\chi_R$, which is $O(R^{-1})$. So it has no contribution (for fixed time interval) as $R \to \infty$.

The cross term (i.e., the second term on the right-hand side of \eqref{eq:truncated-Aq'}) can be controlled it by quadratic terms in $r$ and $v$.
Using \eqref{eq:weyl-quant-2order} with $f=r$ and $f=v$, we have
\begin{align*}
&\la \Op^w\big( \frac{1}{2}\chi_R(x/t)(2\xi-x/t)^TD^2q(x/t)(2\xi-x/t) \big)r,r \ra
+ 2 \Re \la \Op^w\big( \frac{1}{2}\chi_R(x/t) (2\xi-x/t)^TD^2 q(x/t) (2\xi-x/t) \big)r,v \ra 
\\ & \geq \frac{1}{4} \int \chi_R(x/t)(L_tr)^*D^2q(x/t)(L_tr)\,dx
-C\int \chi_R(x/t)(L_tv)^*D^2q(x/t)(L_tv)\,dx -Ct^{-2}.
\end{align*}
Then we substitute this into \eqref{eq:truncated-Aq'} and integrate it in $t$ from $S$ to $T$ and send $R\to\infty$.
We have
\footnote{The first term is nonnegative and can be discarded before taking the limit to avoid the issue of finiteness of that term.}
\begin{align}
\label{eq:convex-action-production-lower-bound}
A_q(T)-A_q(S) &\geq-C\int_S^T K(v,t)\frac{dt}{t} -C\int_S^T\frac{dt}{t^3} -\frac{2}{d+1}\int_S^T\int (\Delta q)(x/t)|u(t,x)|^{2(d+1)/(d-1)} \,dx\frac{dt}{t}.
\end{align}
Of the terms on the right-hand side, it is straightforward that the tail of the second tends to zero as $S,T \to \infty$.
We also proved above that the first one has tail tending to zero.

Now we consider the last term. We will use that for any sequence $A_n, B_n \to \infty$, $A_n<B_n$, we have
\begin{align}
\label{eq:interval-nonlinear-charge-vanishing}
\int_{A_n}^{B_n}\int (\Delta q)(x/t)|r(t,x)|^{2(d+1)/(d-1)} \,dx\frac{dt}{t}\to 0.
\end{align}
We postpone its proof to Proposition~\ref{prop:interval-charge}, and the conditions there are indeed satisfied by the construction in Proposition~\ref{prop:q-construction}.
On the other hand, Lemma~\ref{lem:diagonal-radiation} gives
\begin{align}
\label{eq:AnBn-v-est}
\int_{A_n}^{B_n}\int (\Delta q)(x/t)|v(t,x)|^{2(d+1)/(d-1)} \,dx\frac{dt}{t} &\leq\|\Delta q\|_{L^\infty} \int_{A_n}^\infty \|v(t)\|_{L^{2(d+1)/(d-1)}}^{2(d+1)/(d-1)} \frac{dt}{t}\to 0.
\end{align}
Since $u=r+v$, we have $|u|^{2(d+1)/(d-1)} \leq C_d(|r|^{2(d+1)/(d-1)}+|v|^{2(d+1)/(d-1)})$, so \eqref{eq:interval-nonlinear-charge-vanishing} and \eqref{eq:AnBn-v-est} give
\begin{align}
\label{eq:convex-action-nonlinear-tail}
\int_{A_n}^{B_n}\int (\Delta q)(x/t)|u(t,x)|^{2(d+1)/(d-1)} \,dx\frac{dt}{t}\to 0.
\end{align}
Finally, taking $S=A_n$ and $T=B_n$ in \eqref{eq:convex-action-production-lower-bound} and combining this with the preceding tail estimates gives
\begin{align*}
\liminf_{n\to\infty}(A_q(B_n)-A_q(A_n))\geq 0.
\end{align*}
\end{proof}

Now we prove \eqref{eq:interval-nonlinear-charge-vanishing} used above.
\begin{prop}
\label{prop:interval-charge}
Let $u_+\in H^1(\R^d)$, and let $v(t)$ and $r(t)$ be as in \eqref{eq:v,r-def}.
Choose $\theta,\chi_0,\chi_1,\tilde{\theta} \in C^\infty(\R^d)$  with bounded derivatives such that: $0 \leq \chi_0,\chi_1,\tilde{\theta} \leq 1$, $0 \leq \theta \lesssim 1$, and $\theta=0$ outside $\{\chi_0=1\}$, $\tilde{\theta}=1$ near the supports of $\chi_0,\nabla\chi_0,\Delta\chi_0$,
and $\chi_1=1$ near $\supp\tilde{\theta}$. Assume $y\cdot\nabla\chi_0(y)$ is bounded. If $A_n\to\infty$, $A_n<B_n$, and
\begin{align*}
 \epsilon_n:=\sup_{A_n\leq t\leq B_n} \|\chi_1(x/t)r(t)\|_{H^1}\to0,
\end{align*}
then
\begin{align}
\label{eq:interval-charge-conclusion}
 \int_{A_n}^{B_n}\int\theta(x/t)|r(t,x)|^{2(d+1)/(d-1)}\,dx\frac{dt}{t} \to0.
\end{align}
\end{prop}

\begin{proof}
By definition we have
\begin{align}
\label{eq:delta-n-inq}
(i\partial_t+\Delta+|u|^{4/(d-1)})r &= -|u|^{4/(d-1)}v.
\end{align}

First, we prove the localized smallness of $|u|^{4/(d-1)}$. Interpolation between $L^2$ and $L^{2d/(d-2)}$,
followed by the Sobolev embedding $H^1\hookrightarrow L^{2d/(d-2)}$, gives
\begin{align}
\label{eq:localized-interpolation-bound}
\|f\|_{L^{2d/(d-1)}} \leq \|f\|_{L^2}^{1/2} \|f\|_{L^{2d/(d-2)}}^{1/2} \leq C\|f\|_{H^1}.
\end{align}
Because $\chi_1=1$ on a neighborhood of $\supp\tilde{\theta}$, for every $A_n\leq t\leq B_n$, using \eqref{eq:localized-interpolation-bound}, we have
\begin{align*}
\|\tilde{\theta}(x/t)|r(t)|^{4/(d-1)}\|_{L^{d/2}} &\leq \|\chi_1(x/t)r(t)\|_{L^{2d/(d-1)}}^{4/(d-1)} \leq C\|\chi_1(x/t)r(t)\|_{H^1}^{4/(d-1)} \leq C\epsilon_n^{4/(d-1)}.
\end{align*}
Using Lemma \ref{lem:free-lq-decay}, we have
\begin{align*}
\|v(t)\|_{L^{2(d+1)/(d-1)}} \to 0 \qquad\text{as }t\to\infty.
\end{align*}
The free flow is unitary on $L^2$, so $\|v(t)\|_{L^2} = \|u_+\|_{L^2}$. Interpolating these two bounds gives
\begin{align*}
\|v(t)\|_{L^{2d/(d-1)}} &\leq \|u_+\|_{L^2}^{(d-1)/(2d)} \|v(t)\|_{L^{2(d+1)/(d-1)}}^{(d+1)/(2d)} \to 0.
\end{align*}
Since $|u|^{4/(d-1)} =|r+v|^{4/(d-1)} \lesssim |r|^{4/(d-1)}+|v|^{4/(d-1)}$ \footnote{In fact, in our case, $\frac{4}{d-1}\leq 1$, and we have $|r+v|^{4/(d-1)} \leq |r|^{4/(d-1)}+|v|^{4/(d-1)}$.}, we have
\begin{align}
\label{eq:delta_n-est}
\begin{aligned}
\delta_n &:= \sup_{A_n\leq t\leq B_n} \|\tilde{\theta}(x/t)|u(t)|^{4/(d-1)}\|_{L^{d/2}} \leq C\epsilon_n^{4/(d-1)} + \sup_{t\geq A_n} \|v(t)\|_{L^{2d/(d-1)}}^{4/(d-1)} \to 0.
\end{aligned}
\end{align}

We next multiply \eqref{eq:delta-n-inq} by $\chi_0(x/t)$. Since $\tilde{\theta}=1$ on a neighborhood of $\supp\chi_0$,
\begin{align*}
\tilde{\theta}(x/t)\chi_0(x/t) = \chi_0(x/t).
\end{align*}
Then we have
\begin{align}
\label{eq:chi0-r-PDE}
\begin{aligned}
(i\partial_t+\Delta+\tilde{\theta}(x/t)|u(t,x)|^{4/(d-1)}) (\chi_0(x/t)r(t,x)) = & -it^{-1}((x/t)\cdot\nabla\chi_0(x/t))r(t,x) +2t^{-1}\nabla\chi_0(x/t)\cdot\nabla r(t,x) \\
&\qquad +t^{-2}(\Delta\chi_0)(x/t)r(t,x) -\chi_0(x/t)|u(t,x)|^{4/(d-1)}v(t,x).
\end{aligned}
\end{align}

For every nonnegative integer $k$ such that $2^kA_n<B_n$, define
\begin{align*}
I_{n,k} = \{t\in\R:2^kA_n\leq t\leq\min(2^{k+1}A_n,B_n)\}.
\end{align*}
These intervals cover every $t$ satisfying $A_n\leq t\leq B_n$ and have disjoint interiors.

For such an interval, define
\begin{align*}
\|\chi_0r\|_{S(I_{n,k})} &:= \|\chi_0(x/t)r(t)\|_{L_t^\infty L_x^2(I_{n,k})} + \|\chi_0(x/t)r(t)\|_{L_t^2L_x^{2d/(d-2)}(I_{n,k})}.
\end{align*}
Here $S(I_{n,k})$ denotes the Strichartz norm. 
By the choice of cut-offs, we know
\begin{align*}
\chi_1 = 1 \quad \text{ on a neighborhood of } \supp\chi_0 \cup \supp\nabla\chi_0 \cup \supp\Delta\chi_0.
\end{align*}
In particular,
\begin{align*}
\|\chi_0(x/(2^kA_n))r(2^kA_n)\|_{L^2} &\leq \|\chi_1(x/(2^kA_n))r(2^kA_n)\|_{L^2} \leq \epsilon_n.
\end{align*}

We now estimate the three commutator terms in \eqref{eq:chi0-r-PDE}. Since $y\cdot\nabla\chi_0(y)$ is bounded by hypothesis and is supported where $\nabla\chi_0\neq 0$,
\begin{align*}
\|((x/t)\cdot\nabla\chi_0(x/t))r(t)\|_{L^2} &\leq C\|\chi_1(x/t)r(t)\|_{L^2} \leq C\epsilon_n
\end{align*}
for $A_n\leq t\leq B_n$.

On a neighborhood of $\supp\nabla\chi_0$, one has both $\chi_1=1$ and $\nabla\chi_1=0$. Hence
\begin{align*}
\nabla\chi_0(x/t)\cdot\nabla r(t) = \nabla\chi_0(x/t)\cdot \nabla(\chi_1(x/t)r(t)),
\end{align*}
and therefore
\begin{align*}
\|\nabla\chi_0(x/t)\cdot\nabla r(t)\|_{L^2} \leq C\epsilon_n.
\end{align*}

The residual $r$ of \eqref{eq:v,r-def} is uniformly bounded in $H^1$, because $u$ is uniformly bounded in $H^1$ and the free flow preserves the $H^1$-norm of $u_+$.
Thus
\begin{align*}
\|(\Delta\chi_0)(x/t)r(t)\|_{L^2} \leq C.
\end{align*}

Since $\int_{I_{n,k}}\frac{dt}{t} \leq \log 2$,  $\int_{I_{n,k}}\frac{dt}{t^2} \leq \frac{1}{2^kA_n}$, 
the first three terms on the right-hand side of \eqref{eq:chi0-r-PDE} satisfy
\begin{align*}
\| -it^{-1}((x/t)\cdot\nabla\chi_0(x/t))r +2t^{-1}\nabla\chi_0(x/t)\cdot\nabla r +t^{-2}(\Delta\chi_0)(x/t)r \|_{L_t^1L_x^2(I_{n,k})} &\leq C \epsilon_n + \frac{C}{2^kA_n}.
\end{align*}

Using H\"older's inequality and $\delta_n$ defined in \eqref{eq:delta_n-est}, we have
\begin{align*}
\|\tilde{\theta}(x/t)|u|^{4/(d-1)}\chi_0(x/t)r\|_ {L_t^2L_x^{2d/(d+2)}(I_{n,k})} &\leq \|\tilde{\theta}(x/t)|u(t)|^{4/(d-1)}\|_{L^\infty L^{d/2}} \|\chi_0(x/t)r\|_ {L_t^2L_x^{2d/(d-2)}(I_{n,k})} \leq \delta_n\|\chi_0r\|_{S(I_{n,k})}.
\end{align*}
Since  $\chi_0 \leq \tilde{\theta}$ by our choice, we have
\begin{align*}
\|\chi_0(x/t)|u|^{4/(d-1)}v\|_ {L_t^2L_x^{2d/(d+2)}(I_{n,k})} &\leq \delta_n \|v\|_ {L_t^2L_x^{2d/(d-2)}(I_{n,k})}.
\end{align*}

Applying \eqref{eq:strichartz-homogeneous} and \eqref{eq:strichartz-inhomogeneous} to \eqref{eq:chi0-r-PDE} gives
\begin{align*}
\|\chi_0r\|_{S(I_{n,k})} &\leq C( \epsilon_n +\frac{1}{2^k A_n} +\delta_n \|\chi_0r\|_{S(I_{n,k})} + \delta_n \|v\|_{L_t^2L_x^{2d/(d-2)}(I_{n,k})}).
\end{align*}
Since $\delta_n \to 0$ as $n \to \infty$, the term containing $\|\chi_0r\|_{S(I_{n,k})}$ on the right can be absorbed for all sufficiently large $n$.
Thus
\begin{align}
\label{eq:dyadic-stri}
\|\chi_0r\|_{S(I_{n,k})} &\leq C( \epsilon_n +\frac{1}{2^k A_n} + \delta_n \|v\|_{L_t^2L_x^{2d/(d-2)}(I_{n,k})}).
\end{align}
Applying \eqref{eq:strichartz-homogeneous} with $(q,r)=(2,2d/(d-2))$ to $v$, we have
\begin{align*}
\sum_k \|v\|_{L_t^2L_x^{2d/(d-2)}(I_{n,k})}^2 \leq C_d\|u_+\|_{L^2}^2.
\end{align*}
Now we prove \eqref{eq:interval-charge-conclusion}.
For every $f\in H^1$, H\"older's inequality and Sobolev embedding give
\begin{align*}
\|f\|_{L^{2(d+1)/(d-1)}}^{2(d+1)/(d-1)} \leq \|f\|_{L^{2d/(d-1)}}^{4/(d-1)} \|f\|_{L^{2d/(d-2)}}^2 \leq C\|f\|_{H^1}^{4/(d-1)} \|f\|_{L^{2d/(d-2)}}^2.
\end{align*}
Since $\chi_0(x/t)r(t)$ is uniformly bounded in $H^1$, for all sufficiently large $n$ and $A_n\leq t\leq B_n$, we have
\begin{align*}
\|\chi_0(x/t)r(t)\|_{L^{2(d+1)/(d-1)}}^{2(d+1)/(d-1)} \leq C\|\chi_0(x/t)r(t)\|_{L^{2d/(d-2)}}^2.
\end{align*}
Since $\theta=0$ outside $\{\chi_0=1\}$ and $0\leq \theta\lesssim 1$,
\begin{align*}
\theta(x/t)|r(t,x)|^{2(d+1)/(d-1)} \lesssim |\chi_0(x/t)r(t,x)|^{2(d+1)/(d-1)}.
\end{align*}
Therefore
\begin{align*}
\int_{I_{n,k}}\int_{\R^d} \theta(x/t)|r(t,x)|^{2(d+1)/(d-1)} \,dx\,\frac{dt}{t} &\leq \frac{C}{2^kA_n} \|\chi_0(x/t)r(t)\|_ {L_t^2L_x^{2d/(d-2)}(I_{n,k})}^2 \leq \frac{C}{2^kA_n} \|\chi_0r\|_{S(I_{n,k})}^2.
\end{align*}

Using \eqref{eq:dyadic-stri},
\begin{align*}
\int_{I_{n,k}}\int_{\R^d} \theta(x/t)|r(t,x)|^{2(d+1)/(d-1)} \,dx\,\frac{dt}{t} &\leq \frac{C\epsilon_n^2}{2^kA_n} + \frac{C}{(2^kA_n)^3} + \frac{C\delta_n^2}{2^kA_n} \|v\|_{L_t^2L_x^{2d/(d-2)}(I_{n,k})}^2.
\end{align*}

Summing over $k$, and using
\begin{align*}
\sum_k\frac{1}{2^kA_n} &\leq \frac{2}{A_n},\qquad \sum_k\frac{1}{(2^kA_n)^3} \leq \frac{C}{A_n^3},
\end{align*}
we obtain
\begin{align*}
\int_{A_n}^{B_n}\int_{\R^d} \theta(x/t)|r(t,x)|^{2(d+1)/(d-1)} \,dx\,\frac{dt}{t} &\leq CA_n^{-1}\epsilon_n^2 + CA_n^{-3} + C\delta_n^2 \sum_k\frac{1}{2^kA_n} \|v\|_{L_t^2L_x^{2d/(d-2)}(I_{n,k})}^2.
\end{align*}
For the last term on the right-hand side, we have
\begin{align*}
\sum_k\frac{1}{2^kA_n} \|v\|_{L_t^2L_x^{2d/(d-2)}(I_{n,k})}^2 &\leq A_n^{-1} \sum_k \|v\|_{L_t^2L_x^{2d/(d-2)}(I_{n,k})}^2 \leq C_dA_n^{-1}\|u_+\|_{L^2}^2.
\end{align*}
Consequently,
\begin{align*}
\int_{A_n}^{B_n}\int_{\R^d} \theta(x/t)|r(t,x)|^{2(d+1)/(d-1)} \,dx\,\frac{dt}{t} &\leq C( A_n^{-1}\epsilon_n^2 + A_n^{-3} + A_n^{-1}\delta_n^2\|u_+\|_{L^2}^2),
\end{align*}
where the right-hand side tends to 0 as $n\to\infty$ since $A_n\to \infty$, $\epsilon_n\to0$, and $\delta_n\to 0$.
This concludes the proof.
\end{proof}

Finally, we prove Theorem~\ref{thm:measure-convergence}.
\begin{proof}[Proof of Theorem~\ref{thm:measure-convergence}]
Recall $F(s)=(\nu_{e^s},J_{e^s})$ defined in \eqref{eq:Fs-def}. Assume that $F(s)$ does not converge as $s\to\infty$.
Let $\Omega_\sharp$ be the set of subsequential limits from Proposition~\ref{prop:connected-omega}, which is also the set of all limits of $F(s_n)$ along sequences $s_n\to\infty$. By Proposition~\ref{prop:connected-omega}, $\Omega_\sharp$ is a nonempty compact connected set,
and because $F$ does not converge, it contains more than one point.

Let $\mu_*, \Psi$ be as in Proposition~\ref{prop:convex-exposure}.
By Proposition~\ref{prop:exposed-rigidity}, the only element of $\Omega_\sharp$ whose first component is $\mu_*$ is $Z_*=(\mu_*,\frac{y}{2}\mu_*)$.
In addition, if $s_n\to\infty$, $F(s_n)\to Z_*$, and after passing to a subsequence, $F(s_n+\tau) \to (\mu(\tau),J(\tau))$ for $\tau$ in every bounded interval, then
\begin{align*}
(\mu(\tau),J(\tau))=Z_* \quad \text{ for every }\tau \in \mathbb{R}.
\end{align*}

Applying Proposition~\ref{prop:q-construction} and writing
\begin{align*}
\mu_*=\sum_j a_j^*\delta_{x_j^*},
\end{align*}
we obtain disjoint neighborhoods $U_j$ of the points $x_j^*$, a smooth convex function $q$ such that $\nabla q$ and $y\cdot\nabla q(y)-q(y)$ are bounded,
and the cutoffs appearing in Propositions~\ref{prop:interval-evacuation} and~\ref{prop:Aq-no-drop}.
On each $U_j$, we have $q(y)$ given by \eqref{eq:q(y)-local-expression}.
Define
\begin{align*}
Q_*=\int q\,d\mu_*=\int\Psi\,d\mu_*.
\end{align*}

By Proposition~\ref{prop:atomic-and-bound}, all coefficients of delta functions here has a uniform lower bound.
A argument by contradiction shows: for measures $(\mu,J)$ that are sufficiently close to $Z_*$, these points has to be very close to at least one of $x_j^*$ as well.
This means that we can choose a sufficiently small fixed closed $d_J$-ball $\mathfrak{B}$ centered at $Z_*$ such that every element $(\mu,J)\in\Omega_\sharp$ lying in $\mathfrak{B}$ satisfies
\begin{align*}
\supp \mu \subset\bigcup_j U_j.
\end{align*}
We can also choose $\mathfrak{B}$ small enough so that it does not contain the entire $\Omega_\sharp$.
Consider the elements of $\Omega_\sharp$ lying on $\partial\mathfrak{B}$, which is a compact set. Since Proposition~\ref{prop:supp-mu-bound} places the support of all relevant measures in a fixed compact ball, so
\begin{align*}
\int \Psi d\mu
\end{align*}
is continuous in $(\mu,J)$ on $\Omega_\sharp$.

Now we show that no element $(\mu,J)\in\Omega_\sharp\cap\partial\mathfrak{B}$ can satisfy
\begin{align*}
\int\Psi\,d\mu=\int\Psi\,d\mu_*.
\end{align*}
Indeed, the uniqueness statement in Proposition~\ref{prop:convex-exposure} would give $\mu=\mu_*$, and Proposition~\ref{prop:exposed-rigidity} would then give $(\mu,J)=Z_*$,
which is impossible because $Z_*$ is the center of $\mathfrak{B}$.
Therefore compactness and the continuity in $\mu$ give a fixed number $\delta_0>0$ such that every $(\mu,J)\in\Omega_\sharp\cap\partial\mathfrak{B}$ satisfies
\begin{align} \label{eq:Psi-dmu-gap}
\int\Psi d\mu\leq\int\Psi d\mu_*-\delta_0.
\end{align}

Since $Z_*\in\Omega_\sharp$, choose $\alpha_n\to\infty$ such that
\begin{align*}
F(\alpha_n) \to Z_*.
\end{align*}
For sufficiently large $n$, $F(\alpha_n)$ lies inside $\mathfrak{B}$. Because $\Omega_\sharp$ contains a point outside $\mathfrak{B}$,
there are arbitrarily late times at which $F$ lies outside $\mathfrak{B}$. The map $F$ is continuous in $d_J$ by Proposition~\ref{prop:connected-omega},
so let $\beta_n>\alpha_n$ be the first time after $\alpha_n$ at which $F$ reaches $\partial\mathfrak{B}$.

We now prove:
\begin{align}
\label{eq:b-a-diverge}
\beta_n-\alpha_n \to\infty.
\end{align}
Otherwise, after passing to a subsequence, we can find a constant $h \geq 0$ such that
\begin{align*}
\beta_n-\alpha_n \to h.
\end{align*}
By Theorem~\ref{thm:subsequential-family}, after passing to another subsequence, there is a subsequential limit family $(\mu_\alpha(\tau),J_\alpha(\tau))$,
defined for every $\tau\in\mathbb{R}$, such that
\begin{align*}
F(\alpha_n+\tau) \to (\mu_\alpha(\tau),J_\alpha(\tau))
\end{align*}
locally uniformly for $\tau$ in bounded intervals. At $\tau=0$,
\begin{align*}
(\mu_\alpha(0),J_\alpha(0))=\lim_{n\to\infty}F(\alpha_n)=Z_*.
\end{align*}
So by Proposition~\ref{prop:exposed-rigidity}, we have
\begin{align*}
(\mu_\alpha(\tau),J_\alpha(\tau))=Z_* \qquad\text{for every }\tau\in\mathbb{R}.
\end{align*}
In particular,
\begin{align*}
F(\beta_n) =F(\alpha_n+(\beta_n-\alpha_n))  \to (\mu_\alpha(h),J_\alpha(h))=Z_*.
\end{align*}
This is impossible because every $F(\beta_n)$ lies in $\partial \mathfrak{B}$. This proves \eqref{eq:b-a-diverge}.

Applying Theorem~\ref{thm:subsequential-family} to the time sequence $\beta_n$, after passing to a subsequence,
\begin{align}
\label{eq:F-limit-1}
F(\beta_n+\tau)  \to (\mu_\beta(\tau),J_\beta(\tau))
\end{align}
locally uniformly for $\tau$ in bounded intervals. For every fixed $\tau\leq0$, equation~\eqref{eq:b-a-diverge} implies that,
for all sufficiently large $n$,
\begin{align*}
\alpha_n\leq\beta_n+\tau\leq\beta_n.
\end{align*}
By the definition of $\beta_n$, $F(s)$ remains in $\mathfrak{B}$ for every $s \in [\alpha_n, \beta_n]$.
Therefore $(\mu_\beta(\tau),J_\beta(\tau))$ lies in $\mathfrak{B}$ for every $\tau\leq0$.
Moreover, since $\beta_n+\tau\to\infty$,
each $(\mu_\beta(\tau),J_\beta(\tau))$ belongs to $\Omega_\sharp$.
By the choice of $\mathfrak{B}$,
\begin{align*}
\supp\mu_\beta(\tau)\subset\bigcup_j U_j \qquad (\tau\leq0).
\end{align*}
At $\tau=0$, equation~\eqref{eq:F-limit-1} and the fact that every $F(\beta_n)$ lies in $\partial\mathfrak{B}$ show that $(\mu_\beta(0),J_\beta(0))$ also lies in $\partial\mathfrak{B}$.

Let $q$ be the smooth convex function as in Proposition~\ref{prop:q-construction}.
Define
\begin{align}
\label{eq:Q-def}
Q(\tau)=\int q\,d\mu_\beta(\tau).
\end{align}
Since $q$ is given by \eqref{eq:q(y)-local-expression} on $U_j$, and the expression in \eqref{eq:q(y)-local-expression} agrees with a tangent plane of the convex function $\Psi$, we have $q(y)\leq\Psi(y)$ on $U_j$.
Since $(\mu_\beta(\tau),J_\beta(\tau))\in\Omega_\sharp$, the maximizing property of $\mu_*$ gives
\begin{align*}
\int\Psi\,d\mu_\beta(\tau) \leq \int\Psi\,d\mu_*.
\end{align*}
Therefore, for every $\tau\leq0$,
\begin{align*}
Q(\tau) =\int q\,d\mu_\beta(\tau) \leq\int\Psi\,d\mu_\beta(\tau) \leq\int\Psi\,d\mu_* =Q_*.
\end{align*}
At $\tau=0$, $(\mu_\beta(0),J_\beta(0))$ lies in $\partial\mathfrak{B}$, so \eqref{eq:Psi-dmu-gap} gives
\begin{align*}
Q(0) \leq\int\Psi\,d\mu_\beta(0) \leq\int\Psi\,d\mu_*-\delta_0 =Q_*-\delta_0.
\end{align*}
Thus
\begin{align}
\label{eq:Q-delta0-gap}
Q(0) \leq Q_*-\delta_0, \quad Q(\tau) \leq Q_*  \; \forall \tau \leq 0.
\end{align}

Fix $\tau \leq 0$ and consider the times satisfying $e^{\alpha_n}\leq t\leq e^{\beta_n+\tau}$, which exist by \eqref{eq:b-a-diverge} for sufficiently large $n$.
Take a sequence of times $t_n$ satisfying
\begin{align*}
e^{\alpha_n} \leq t_n \leq e^{\beta_n+\tau}.
\end{align*}
Then $F(\log t_n)$ lies in $\mathfrak{B}$. Every limit of a sequence of these states therefore belongs to $\Omega_\sharp$,
lies in $\mathfrak{B}$, and has its mass support contained in $\bigcup_j U_j$. The cutoff $\chi_1$ given by Proposition~\ref{prop:q-construction} is supported a positive distance away from $\bigcup_j U_j$.
Applying Proposition~\ref{prop:interval-evacuation} with $K=\bigcup_j\overline{U_j}$ gives
\begin{align}
\label{eq:interval-H1-est}
\sup_{e^{\alpha_n} \leq t\leq e^{\beta_n+\tau}} \|\chi_1(x/t)r(t)\|_{H^1} \to 0.
\end{align}

By \eqref{eq:interval-H1-est}, Proposition~\ref{prop:Aq-no-drop} can now be applied to these intervals and, with $A_q$ defined in \eqref{eq:Aq-def}, gives:
\begin{align}
\label{eq:Aq-nodrop-2}
\liminf_{n\to\infty} \big( A_q(e^{\beta_n+\tau})-A_q(e^{\alpha_n}) \big) \geq 0.
\end{align}
For a pair $(\mu,J)$ of measures as before, define
\begin{align}
\label{eq:def-Lq}
L_q(\mu,J) & := \int\nabla q(y)\cdot dJ(y) -\frac{1}{2}\int\big(y\cdot\nabla q(y)-q(y)\big)\,d\mu(y).
\end{align}
By this definition and the definition of $(\nu_t,J_t)$ in \eqref{eq:nu_t,J_t-def},
\begin{align*}
L_q(\nu_t,J_t) &=\int_{\R^d}\nabla q(x/t)\cdot j(r(t,x))\,dx -\frac{1}{2}\int_{\R^d} \Big(\frac{x}{t}\cdot\nabla q(x/t)-q(x/t)\Big) |r(t,x)|^2\,dx.
\end{align*}
Since $u(t,x)=v(t,x)+r(t,x)$, combining the estimates in Proposition~\ref{prop:strong-orthogonality},
we have
\begin{align}
|A_q(t)-L_q(\nu_t,J_t)| &\leq\|\nabla q\|_{L^\infty} \int_{\R^d}\big( |v(t,x)||\nabla r(t,x)|+|r(t,x)||\nabla v(t,x)| \big)\,dx\notag\\
\label{eq:action-measure-pairing-approximation}
&\quad+\|y\cdot\nabla q-q\|_{L^\infty} \int_{\R^d}|v(t,x)||r(t,x)|\,dx \to 0.
\end{align}

By the choice of $\alpha_n$ and \eqref{eq:Z*-def},
\begin{align}
\label{eq:F-alpha-limit}
 F(\alpha_n)=(\nu_{e^{\alpha_n}},J_{e^{\alpha_n}})  \to \Big(\mu_*,\frac{y}{2}\mu_*\Big).
\end{align}
Choose $\chi_R\in C_c^\infty(\R^d)$ satisfying $0\leq\chi_R\leq1$, equal to one on $B_R$, and supported in $B_{2R}$. For fixed $R$,
the functions $\chi_R\nabla q$ and $\chi_R(y\cdot\nabla q-q)$ extend continuously to $X$, so \eqref{eq:F-alpha-limit} may be tested against them.
Using Lemma~\ref{lem:exterior-cone}, we know 
\begin{align*}
\int_{|x|>Rt}|r(t,x)|^2\,dx \to 0, \text{ as } R \to \infty,
\end{align*}
and this controls any $L^2$-level term as well. In addition
\begin{align*}
 \int_{|x|>Rt}|j(r(t,x))|\,dx \leq \Big(\int_{|x|>Rt}|r(t,x)|^2\,dx\Big)^{1/2} \|\nabla r(t)\|_{L^2}
\end{align*}
and the uniform $H^1$ bound on $r$ control the current tails. The limit pair is supported in a fixed ball by Proposition~\ref{prop:supp-mu-bound}.
Thus, first letting $n\to\infty$ and then $R\to\infty$, we obtain
\begin{align}
\label{eq:Lq-alpha-limit}
 L_q(\nu_{e^{\alpha_n}},J_{e^{\alpha_n}})  \to L_q\Big(\mu_*,\frac{y}{2}\mu_*\Big).
\end{align}
Combining \eqref{eq:Lq-alpha-limit} with \eqref{eq:action-measure-pairing-approximation} gives
\begin{align}
\label{eq:Aq-limit-1}
A_q(e^{\alpha_n})  \to L_q(\mu_*,\frac{y}{2}\mu_*).
\end{align}
Using \eqref{eq:def-Lq}, we have
\begin{align*}
L_q(\mu_*,\frac{y}{2}\mu_*) &=\int\nabla q(y)\cdot\frac{y}{2}\,d\mu_*(y) -\frac{1}{2}\int\big(y\cdot\nabla q(y)-q(y)\big)\,d\mu_*(y) =\frac{1}{2}\int q\,d\mu_* =\frac{Q_*}{2}.
\end{align*}
Similarly, \eqref{eq:F-limit-1} gives
\begin{align*}
A_q(e^{\beta_n+\tau})  \to L_q(\mu_\beta(\tau),J_\beta(\tau)).
\end{align*}
Combining \eqref{eq:Aq-nodrop-2}, \eqref{eq:Aq-limit-1}, and \eqref{eq:Lq-alpha-limit}, we obtain
\begin{align}
\label{eq:Lq-lower-bd}
L_q(\mu_\beta(\tau),J_\beta(\tau)) \geq\frac{Q_*}{2} \qquad (\tau\leq0).
\end{align}

The measures in \eqref{eq:F-limit-1} satisfy \eqref{eq:mass-difference} from Theorem~\ref{thm:subsequential-family}.
Although \eqref{eq:mass-difference} is stated for compactly supported smooth functions, Proposition~\ref{prop:supp-mu-bound} places the support of every $\mu_\beta(\tau)$ and $J_\beta(\tau)$ in one fixed compact ball.
We may therefore apply \eqref{eq:mass-difference} to a compactly supported smooth function that agrees with $q$ on that ball.
Equation~\eqref{eq:mass-difference} shows that $Q$ is absolutely continuous and gives, for almost every $\tau$,
\begin{align}
\label{eq:q-int-derivative}
Q'(\tau) =\int\nabla q(y)\cdot \big(2\,dJ_\beta(\tau)(y)-y\,d\mu_\beta(\tau)(y)\big).
\end{align}
Using \eqref{eq:q-int-derivative} and the definition of $L_q$,
\begin{align}
\label{eq:Lq-Q-1}
2L_q(\mu_\beta(\tau),J_\beta(\tau)) &=2\int\nabla q(y)\cdot dJ_\beta(\tau)(y) -\int\big(y\cdot\nabla q(y)-q(y)\big)\,d\mu_\beta(\tau)(y) = Q'(\tau)+ Q(\tau).
\end{align}
Combining \eqref{eq:Lq-lower-bd} and \eqref{eq:Lq-Q-1} gives
\begin{align}
\label{eq:q-int-ODE}
Q'(\tau) + Q(\tau) \geq Q_* \qquad (\tau\leq0).
\end{align}
Applying a Gr\"onwall type argument, we have
\begin{align*}
Q(\tau) \leq e^{-\tau}(Q(0)-Q_*)+Q_*.
\end{align*}
So $|Q(\tau)|$ can be arbitrarily large by taking $\tau$ to be very negative.
But every $\mu_\beta(\tau)$ has the fixed total mass $m_\sharp$, and Proposition~\ref{prop:supp-mu-bound} places its support in one fixed compact ball.
Since $q$ is bounded on that ball,
\begin{align*}
Q(\tau)=\int q\,d\mu_\beta(\tau)
\end{align*}
is uniformly bounded for all $\tau$, which is a contradition.

So the assumption that $F(s)=(\nu_{e^s},J_{e^s})$ does not converge is false. Since $t=e^s$, this is exactly the joint weak convergence of $(\nu_t,J_t)$ as $t\to\infty$.

Denote the joint weak limit by $(\mu_\infty,J_\infty)$ and $\mu_\infty$ takes the form in Proposition~\ref{prop:mass-limit-form}.
Apply Theorem~\ref{thm:subsequential-family} to a sequence $s_n \to \infty$.
Since $e^{s_n+\tau}\to\infty$ for every fixed $\tau$, the convergence just proved and \eqref{eq:family-limit} show that the resulting subsequential limit family is constant and equal to $(\mu_\infty,J_\infty)$.
The support assertion in Theorem~\ref{thm:subsequential-family} therefore gives $J_\infty=\sum_{\ell=1}^Lb_\ell\delta_{y_\ell}$ for some $b_\ell\in\R^d$.
Substituting these representations into \eqref{eq:mass-difference} gives
\begin{align}
\label{eq:stationary-limit-atom-balance}
 \sum_{\ell=1}^L(2b_\ell-a_\ell y_\ell)\cdot\nabla\phi(y_\ell)=0
\end{align}
for every real-valued $\phi\in C_c^\infty(\R^d)$. So we have $b_\ell=a_\ell y_\ell/2$, which means $J_\infty=y\mu_\infty/2$ as stated in \eqref{eq:measure-convergence-limit-form}.
\end{proof}

\section{Sequential to uniform upgrading}
\label{sec:sequential_to_uniform_upgrading}

We now come to our byproduct: the upgrading estimates in Theorem~\ref{thm:upgrade-I} and Theorem~\ref{thm:upgrade-II} in this section.

\subsection{Sequential to uniform upgrading I: smallness}
\label{sec:upgrade-I}

We prove Theorem~\ref{thm:upgrade-I} first.
Our Theorem~\ref{thm:measure-convergence} established in the previous section already gives a very strong characterization of $r(t)$, and the remaining work is not so difficult.

\begin{proof}[Proof of Theorem~\ref{thm:upgrade-I}]
Suppose \eqref{eq:uniform-chi0-r-tend-0} fails.
Then Proposition~\ref{prop:residual-mass-positive}, since \eqref{eq:cone-upgrade-sequential} now holds, shows that $m_\sharp=M(u)-M(u_+)>0$.
By Theorem~\ref{thm:measure-convergence}, we have 
\begin{align*}
 \nu_t \rightharpoonup\mu_\infty, \quad 
 \mu_\infty=\sum_{\ell=1}^La_\ell\delta_{y_\ell},\qquad a_\ell>0,\quad y_\ell\in\R^d.
\end{align*}
By the definition of $\nu_t$ in \eqref{eq:nu_t,J_t-def}, the weak convergence gives
\begin{align}
\label{eq:chi1-r-convergence}
 \|\chi_1(x/T_n)r(T_n)\|_{L^2}^2\to\sum_{\ell=1}^La_\ell\chi_1(y_\ell)^2.
\end{align}
The left side tends to zero by \eqref{eq:cone-upgrade-sequential}, and we conclude
\begin{align}
\label{eq:chi-one-zero-on-limit-atoms}
 \chi_1(y_\ell)=0\qquad(1\leq\ell\leq L).
\end{align}

Applying Proposition~\ref{prop:active-cells}, after potentially passing to a subsequence, we have $s_n>T_n$ and a constant $m_*>0$ such that
\begin{align*}
 \|\chi_0(x/s_n)r(s_n)\|_{L^2}\geq m_*.
\end{align*}
In the same way as \eqref{eq:chi1-r-convergence} (i.e., using $\nu_t \rightharpoonup\mu_\infty$), we have
\begin{align*}
 \|\chi_0(x/s_n)r(s_n)\|_{L^2}^2\to\sum_{\ell=1}^La_\ell\chi_0(y_\ell)^2\geq m_*^2.
\end{align*}
Thus $\chi_0(y_\ell)\neq0$ for some $\ell$.
Since $\supp\chi_0\Subset\{y:\chi_1(y)=1\}$, this implies $\chi_1(y_\ell)=1$, contradicting \eqref{eq:chi-one-zero-on-limit-atoms}.
\end{proof}

\subsection{Sequential to uniform upgrade II: below threshold}
\label{sec:upgrade-II}
We prove Theorem~\ref{thm:upgrade-II} in this section.

\begin{proof}[Proof of Theorem~\ref{thm:upgrade-II}]
Let $r(t)$ and $v(t)$ be as in \eqref{eq:v,r-def}, and let $\Theta$ be as in \eqref{eq:Theta-def}.
For $i=0,2$, define
\begin{align*}
 m_i=(2\pi)^d\int\chi_i(2\xi)^2|\widehat{u}_+(\xi)|^2\,d\xi, \quad k_i=(2\pi)^d\int\chi_i(2\xi)^2|\xi|^2|\widehat{u}_+(\xi)|^2\,d\xi.
\end{align*}

Take an arbitrary subsequence of $t_n$, still denoted by $t_n$.
By Theorem~\ref{thm:compact-attractor}, after passing to a further subsequence,
\begin{align*}
 u(t_n) =e^{it_n\Delta}u_+ +\sum_{j=1}^Jw_j(\cdot-x_{j,n})+e_n, \quad e_n \to 0 \quad\text{in }H^1, \quad  |x_{j,n}-x_{k,n}| \to \infty\qquad(j\ne k).
\end{align*}
After passing to a further subsequence, for every $j$,
\begin{align*}
 \frac{x_{j,n}}{t_n}\to y_j\in X=\R^d\cup\{\infty\}.
\end{align*}
Set
\begin{align*}
 a_j=
 \begin{cases}
 \chi_2(y_j),&y_j\in\R^d,\\
 0,&y_j=\infty.
 \end{cases}
\end{align*}

By \eqref{eq:localized-scalar-limits}, after passing to a further subsequence, the limits
\begin{align*}
 M_{\mathrm{seq}}:= \lim_{n\to\infty} M(\chi_2(x/t_n)u(t_n)),\quad 
 E_{\mathrm{seq}}:= \lim_{n\to \infty} E(\chi_2(x/t_n)u(t_n))
\end{align*}
exist.
Using the results referred after \eqref{eq:mass-decoupling}-\eqref{eq:energy-decoupling}, we have
\begin{align*}
 \|\chi_2(x/t_n)w_j(\cdot-x_{j,n})-a_jw_j(\cdot-x_{j,n})\|_{H^1}\to0.
\end{align*}
Using the same referred results, we also have the analogue of \eqref{eq:mass-decoupling}
\begin{align}
\label{eq:sequential-mass-decoupling}
 M_{\mathrm{seq}}=\lim_{n\to\infty}M(\chi_2(x/t_n)e^{it_n\Delta}u_+)+\sum_{j=1}^Ja_j^2M(w_j).
\end{align}
Here we used Lemma~\ref{lem:free-lq-decay} to see that the contribution of the radiation term to the nonlinear term in the energy tends to zero.
Also, the (referred) argument proving \eqref{eq:energy-decoupling} gives
\begin{align}
\label{eq:sequential-energy-decoupling}
 E_{\mathrm{seq}}=\lim_{n\to\infty}\frac{1}{2}\|\nabla(\chi_2(x/t_n)e^{it_n\Delta}u_+)\|_{L^2}^2+\sum_{j=1}^JE(a_jw_j).
\end{align}
For $0 \leq a \leq1$, we have
\begin{align}
\label{eq:scaled-profile-energy}
 E(aw) = a^2E(w) + \frac{d-1}{2(d+1)}\big(a^2-a^{2(d+1)/(d-1)}\big)\|w\|_{L^{2(d+1)/(d-1)}}^{2(d+1)/(d-1)}.
\end{align}
Since $2(d+1)/(d-1)>2$, the second term in \eqref{eq:scaled-profile-energy} is nonnegative.
For $w_j \ne 0 $, Proposition~\ref{prop:profile-threshold} gives
\begin{align*}
 E(w_j)>0,\qquad M(w_j)E(w_j)\geq \Theta.
\end{align*}
It follows from \eqref{eq:scaled-profile-energy} that
\begin{align*}
 E(a_jw_j)\geq0
\end{align*}
for every $j$.
Thus every term on the right-hand sides of \eqref{eq:sequential-mass-decoupling} and \eqref{eq:sequential-energy-decoupling} is nonnegative.

Then we apply Proposition~\ref{prop:free-mass-current}, applied to both $\chi_2(x/t_n)e^{it_n\Delta}u_+$ and its gradient to obtain
\begin{align*}
 \lim_{n\to\infty}M(\chi_2(x/t_n)e^{it_n\Delta}u_+) = m_2, \quad  \lim_{n\to\infty}\frac12\|\nabla(\chi_2(x/t_n)e^{it_n\Delta}u_+)\|_{L^2}^2 = \frac{k_2}{2}.
\end{align*}
Terms introduced by differentiating $\chi(x/t_n)$ will tend to zero since they carry an extra $t_n^{-1}$ factor.
Passing to the limit in \eqref{eq:sequential-threshold} gives
\begin{align*}
 M_{\mathrm{seq}}E_{\mathrm{seq}}\leq(1-\delta)\Theta.
\end{align*}
The nonnegativity in \eqref{eq:sequential-mass-decoupling} and \eqref{eq:sequential-energy-decoupling} therefore gives
\begin{align}
\label{eq:outer-free-product-threshold}
 \frac{m_2k_2}{2}\leq(1-\delta)\Theta.
\end{align}

Suppose that $w_j\ne0$ and $y_j\in\supp\chi_1$ for some $j$.
The second inclusion in \eqref{eq:condition-chi-012} gives $a_j=\chi_2(y_j)=1$.
Equations~\eqref{eq:sequential-mass-decoupling} and~\eqref{eq:sequential-energy-decoupling} imply
\begin{align*}
 M(w_j)\leq M_{\mathrm{seq}},\qquad E(w_j)\leq E_{\mathrm{seq}}.
\end{align*}
Hence
\begin{align*}
 M(w_j)E(w_j)\leq M_{\mathrm{seq}}E_{\mathrm{seq}}\leq(1-\delta)\Theta<\Theta,
\end{align*}
contradicting Proposition~\ref{prop:profile-threshold}.
Therefore
\begin{align*}
 w_j\ne0\quad\implies\quad y_j\notin\supp\chi_1.
\end{align*}

For every nonzero profile, either $y_j\in\R^d\setminus\supp\chi_1$ or $y_j=\infty$.
The argument proving \eqref{eq:mass-limit-3}, now with the cutoff $\chi_1$, gives
\begin{align*}
 \|\chi_1(x/t_n)w_j(\cdot-x_{j,n})\|_{H^1} \to 0
\end{align*}
for every $j$.
Subtracting the radiation term from \eqref{eq:profile-decomposition} gives
\begin{align*}
 r(t_n)=\sum_{j=1}^Jw_j(\cdot-x_{j,n})+e_n.
\end{align*}
Multiplication by $\chi_1(x/t_n)$ is uniformly bounded on $H^1$.
Therefore
\begin{align*}
 \|\chi_1(x/t_n)r(t_n)\|_{H^1}&\leq\sum_{j=1}^J\|\chi_1(x/t_n)w_j(\cdot-x_{j,n})\|_{H^1}\\
 &\quad+C\|e_n\|_{H^1}\to0.
\end{align*}
Since the initial subsequence was arbitrary, the subsequence criterion gives
\begin{align}
\label{eq:sequential-residual-evacuation}
 \|\chi_1(x/t_n)r(t_n)\|_{H^1}\to0.
\end{align}

The first inclusion in \eqref{eq:condition-chi-012} and \eqref{eq:sequential-residual-evacuation} verify the hypotheses of Theorem~\ref{thm:upgrade-I} for the cutoffs $\chi_0,\chi_1$ and the sequence $t_n$.
Therefore
\begin{align}
\label{eq:inner-cutoff-free-approximation}
 \|\chi_0(x/t)u(t)-\chi_0(x/t)v(t)\|_{H^1}\to0.
\end{align}

On the other hand, similar as above, we apply Proposition~\ref{prop:free-mass-current} to $\chi_0(x/t)v(t)$ and its derivatves \footnote{Rigorously, Proposition~\ref{prop:free-mass-current} required $H^1$-regularity. But the first claim concerning the measure $\mu_t^{u_0}$ only requires $L^2$ membership of $u_0$. } to obtain
\begin{align*}
 M(\chi_0(x/t)v(t)) \to m_0, \quad E(\chi_0(x/t)v(t))\to \frac{k_0}{2}.
\end{align*}
The nonlinear term in the energy does not contribute to $E(\chi_0(x/t)v(t))$ in the limit by Lemma~\ref{lem:free-lq-decay}.
Since the difference of the mass and energy can be bounded by the $H^1$-norm, combining the equation above and \eqref{eq:inner-cutoff-free-approximation} shows
\begin{align*}
 M(\chi_0(x/t)u(t)) \to m_0, \quad  E(\chi_0(x/t)u(t)) \to \frac{k_0}{2}.
\end{align*}

By \eqref{eq:condition-chi-012},  we know $0\leq\chi_0\leq\chi_2$.
The terms involving the derivative of $\chi_i$ does not contribute to the limit since they carry an extra $t^{-1}$ factor. So \eqref{eq:outer-free-product-threshold} enforces 
\begin{align}
\label{eq:inner-free-product-threshold}
 \frac{m_0k_0}{2}\leq(1-\delta)\Theta,
\end{align}
which proves \eqref{eq:ME-chi0-below-threshold} for all sufficiently large $t$.
Now we prove \eqref{eq:M-kinetic-below-threshold}.
Recalling the last part of \eqref{eq:ground-state-identities}, using \eqref{eq:inner-free-product-threshold} above, we have
\begin{align*}
 \lim_{t\to\infty}\|\chi_0(x/t)u(t)\|_{L^2}\|\nabla(\chi_0(x/t)u(t))\|_{L^2} = \sqrt{m_0k_0} \leq \sqrt{\frac{1-\delta}{d}} \|Q\|_{L^2}\|\nabla Q\|_{L^2}  < \|Q\|_{L^2}\|\nabla Q\|_{L^2}.
\end{align*}
\end{proof}

\bibliographystyle{plain}
\bibliography{NLS_localized_scattering}

@article{Glassey-NLS-blowup,
    AUTHOR = {Glassey, R. T.},
     TITLE = {On the blowing up of solutions to the {C}auchy problem for
              nonlinear {S}chr\"odinger equations},
   JOURNAL = {J. Math. Phys.},
  FJOURNAL = {Journal of Mathematical Physics},
    VOLUME = {18},
      YEAR = {1977},
    NUMBER = {9},
     PAGES = {1794--1797},
      ISSN = {0022-2488,1089-7658},
   MRCLASS = {35B35 (35G25)},
  MRNUMBER = {460850},
MRREVIEWER = {A.\ A.\ Arsen\cprime ev},
       DOI = {10.1063/1.523491},
       URL = {https://doi.org/10.1063/1.523491},
}

@article{Duychaetz-Roudenko-threshold-3Dcubic,
    AUTHOR = {Duyckaerts, Thomas and Roudenko, Svetlana},
     TITLE = {Threshold solutions for the focusing 3{D} cubic
              {S}chr\"odinger equation},
   JOURNAL = {Rev. Mat. Iberoam.},
  FJOURNAL = {Revista Matem\'atica Iberoamericana},
    VOLUME = {26},
      YEAR = {2010},
    NUMBER = {1},
     PAGES = {1--56},
      ISSN = {0213-2230,2235-0616},
   MRCLASS = {35Q55 (35B40 35P25)},
  MRNUMBER = {2662148},
MRREVIEWER = {Yuichiro\ Kawahara},
       DOI = {10.4171/RMI/592},
       URL = {https://doi.org/10.4171/RMI/592},
}

@article{Nakanishi-Schalg-above-thre-3Dcubic,
    AUTHOR = {Nakanishi, K. and Schlag, W.},
     TITLE = {Global dynamics above the ground state energy for the cubic
              {NLS} equation in 3{D}},
   JOURNAL = {Calc. Var. Partial Differential Equations},
  FJOURNAL = {Calculus of Variations and Partial Differential Equations},
    VOLUME = {44},
      YEAR = {2012},
    NUMBER = {1-2},
     PAGES = {1--45},
      ISSN = {0944-2669,1432-0835},
   MRCLASS = {35Q55 (37K45)},
  MRNUMBER = {2898769},
MRREVIEWER = {Tohru\ Ozawa},
       DOI = {10.1007/s00526-011-0424-9},
       URL = {https://doi.org/10.1007/s00526-011-0424-9},
}

@article{Duychaets-Merle-thre-H1NLS,
    AUTHOR = {Duyckaerts, Thomas and Merle, Frank},
     TITLE = {Dynamic of threshold solutions for energy-critical {NLS}},
   JOURNAL = {Geom. Funct. Anal.},
  FJOURNAL = {Geometric and Functional Analysis},
    VOLUME = {18},
      YEAR = {2009},
    NUMBER = {6},
     PAGES = {1787--1840},
      ISSN = {1016-443X,1420-8970},
   MRCLASS = {35Q55 (35B40 35B44)},
  MRNUMBER = {2491692},
MRREVIEWER = {Olivier\ J.\ Goubet},
       DOI = {10.1007/s00039-009-0707-x},
       URL = {https://doi.org/10.1007/s00039-009-0707-x},
}

@article {Hassell-Jia-NLS,
    AUTHOR = {Hassell, Andrew and Qiuye Jia},
     TITLE = {The final state problem for the nonlinear {S}chr\"odinger
              equation in dimensions 1, 2 and 3},
   JOURNAL = {Pure Appl. Anal.},
  FJOURNAL = {Pure and Applied Analysis},
    VOLUME = {8},
      YEAR = {2026},
    NUMBER = {1},
     PAGES = {1--39},
      ISSN = {2578-5893,2578-5885},
   MRCLASS = {35Q41 (35P25 35Q55 35S05)},
  MRNUMBER = {5000795},
       DOI = {10.2140/paa.2026.8.1},
       URL = {https://doi.org/10.2140/paa.2026.8.1},
}

@article{Berestycki-Lions-83,
    AUTHOR = {Berestycki, H. and Lions, P.-L.},
     TITLE = {Nonlinear scalar field equations. {I}. {E}xistence of a ground
              state},
   JOURNAL = {Arch. Rational Mech. Anal.},
  FJOURNAL = {Archive for Rational Mechanics and Analysis},
    VOLUME = {82},
      YEAR = {1983},
    NUMBER = {4},
     PAGES = {313--345},
      ISSN = {0003-9527},
   MRCLASS = {35J60 (35Q20 58E99 81E99)},
  MRNUMBER = {695535},
MRREVIEWER = {Wei\ Ming\ Ni},
       DOI = {10.1007/BF00250555},
       URL = {https://doi.org/10.1007/BF00250555},
}

@article{GRGH-NLS,
  title={Scattering regularity for small data solutions of the nonlinear {S}chr\"odinger equation},
  author={Gell-Redman, Jesse and Gomes, Sean and Hassell, Andrew},
  journal={arXiv preprint arXiv:2305.12429},
  year={2023}
}

@article{GRGH-LS,
    AUTHOR = {Gell-Redman, Jesse and Gomes, Sean and Hassell, Andrew},
     TITLE = {Propagation estimates and {F}redholm analysis for the
              time-dependent {S}chr\"odinger equation},
   JOURNAL = {Amer. J. Math.},
  FJOURNAL = {American Journal of Mathematics},
    VOLUME = {147},
      YEAR = {2025},
    NUMBER = {6},
     PAGES = {1577--1652},
      ISSN = {0002-9327,1080-6377},
   MRCLASS = {35Q41 (35B40)},
  MRNUMBER = {4995130},
}

@article{Fang-Xie-Cazenave-scattering-11,
    AUTHOR = {Fang, DaoYuan and Xie, Jian and Cazenave, Thierry},
     TITLE = {Scattering for the focusing energy-subcritical nonlinear
              {S}chr\"odinger equation},
   JOURNAL = {Sci. China Math.},
  FJOURNAL = {Science China. Mathematics},
    VOLUME = {54},
      YEAR = {2011},
    NUMBER = {10},
     PAGES = {2037--2062},
      ISSN = {1674-7283,1869-1862},
   MRCLASS = {35Q55 (35B40 35P25)},
  MRNUMBER = {2838120},
MRREVIEWER = {Yoshihisa\ Nakamura},
       DOI = {10.1007/s11425-011-4283-9},
       URL = {https://doi.org/10.1007/s11425-011-4283-9},
}

@book{Tao-06-book,
    AUTHOR = {Tao, Terence},
     TITLE = {Nonlinear dispersive equations},
    SERIES = {CBMS Regional Conference Series in Mathematics},
    VOLUME = {106},
      NOTE = {Local and global analysis},
 PUBLISHER = {Conference Board of the Mathematical Sciences, Washington, DC;
              by the American Mathematical Society, Providence, RI},
      YEAR = {2006},
     PAGES = {xvi+373},
      ISBN = {0-8218-4143-2},
   MRCLASS = {35Q53 (35B35 35P25 35Q55 37K10)},
  MRNUMBER = {2233925},
MRREVIEWER = {Sebastian\ Herr},
       DOI = {10.1090/cbms/106},
       URL = {https://doi.org/10.1090/cbms/106},
}

@book{Cazenave-03-book,
    AUTHOR = {Cazenave, Thierry},
     TITLE = {Semilinear {S}chr\"odinger equations},
    SERIES = {Courant Lecture Notes in Mathematics},
    VOLUME = {10},
 PUBLISHER = {New York University, Courant Institute of Mathematical
              Sciences, New York; American Mathematical Society, Providence,
              RI},
      YEAR = {2003},
     PAGES = {xiv+323},
      ISBN = {0-8218-3399-5},
   MRCLASS = {35Q55 (35-01 35J10 35Q40)},
  MRNUMBER = {2002047},
MRREVIEWER = {Woodford\ W.\ Zachary},
       DOI = {10.1090/cln/010},
       URL = {https://doi.org/10.1090/cln/010},
}

@article{Killip-Visan-H1critical,
    AUTHOR = {Killip, Rowan and Visan, Monica},
     TITLE = {The focusing energy-critical nonlinear {S}chr\"odinger
              equation in dimensions five and higher},
   JOURNAL = {Amer. J. Math.},
  FJOURNAL = {American Journal of Mathematics},
    VOLUME = {132},
      YEAR = {2010},
    NUMBER = {2},
     PAGES = {361--424},
      ISSN = {0002-9327,1080-6377},
   MRCLASS = {35Q55 (35P25)},
  MRNUMBER = {2654778},
MRREVIEWER = {Benedetta\ Pellacci},
       DOI = {10.1353/ajm.0.0107},
       URL = {https://doi.org/10.1353/ajm.0.0107},
}

@article{DHR-3d-sc,
    AUTHOR = {Duyckaerts, Thomas and Holmer, Justin and Roudenko, Svetlana},
     TITLE = {Scattering for the non-radial 3{D} cubic nonlinear
              {S}chr\"odinger equation},
   JOURNAL = {Math. Res. Lett.},
  FJOURNAL = {Mathematical Research Letters},
    VOLUME = {15},
      YEAR = {2008},
    NUMBER = {6},
     PAGES = {1233--1250},
      ISSN = {1073-2780},
   MRCLASS = {35Q55 (35P25)},
  MRNUMBER = {2470397},
MRREVIEWER = {Yoshihisa\ Nakamura},
       DOI = {10.4310/MRL.2008.v15.n6.a13},
       URL = {https://doi.org/10.4310/MRL.2008.v15.n6.a13},
}

@article {Guevara12,
    AUTHOR = {Guevara, Cristi Darley},
     TITLE = {Global behavior of finite energy solutions to the
              {$d$}-dimensional focusing nonlinear {S}chr\"odinger equation},
   JOURNAL = {Appl. Math. Res. Express. AMRX},
  FJOURNAL = {Applied Mathematics Research Express. AMRX},
    VOLUME = {2014},
      YEAR = {2014},
    NUMBER = {2},
     PAGES = {177--243},
      ISSN = {1687-1200,1687-1197},
   MRCLASS = {35Q55},
  MRNUMBER = {3266698},
MRREVIEWER = {Eduardo\ Colorado},
       DOI = {10.1093/amrx/abt008},
       URL = {https://doi.org/10.1093/amrx/abt008},
}

@article {Tao2007Attractor,
    AUTHOR = {Tao, Terence},
     TITLE = {A (concentration-)compact attractor for high-dimensional
              non-linear {S}chr\"odinger equations},
   JOURNAL = {Dyn. Partial Differ. Equ.},
  FJOURNAL = {Dynamics of Partial Differential Equations},
    VOLUME = {4},
      YEAR = {2007},
    NUMBER = {1},
     PAGES = {1--53},
      ISSN = {1548-159X,2163-7873},
   MRCLASS = {35Q55 (35B40 35B41 35Q51 37L30)},
  MRNUMBER = {2304091},
MRREVIEWER = {W.-H.\ Steeb},
       DOI = {10.4310/DPDE.2007.v4.n1.a1},
       URL = {https://doi.org/10.4310/DPDE.2007.v4.n1.a1},
}

@article{DodsonMurphy2018,
  author = {Dodson, Benjamin and Murphy, Jason},
  title = {A new proof of scattering below the ground state for the non-radial focusing {NLS}},
  journal = {Mathematical Research Letters},
  volume = {25},
  number = {6},
  pages = {1805--1825},
  year = {2018},
  eprint = {1712.09962},
  archivePrefix = {arXiv},
  primaryClass = {math.AP}
}

@article{Weinstein1983Sharp,
  author = {Weinstein, Michael I.},
  title = {Nonlinear {S}chrödinger equations and sharp interpolation estimates},
  journal = {Communications in Mathematical Physics},
  volume = {87},
  number = {4},
  pages = {567--576},
  year = {1983},
  doi = {10.1007/BF01208265}
}

@article{Kwong1989Uniqueness,
  author = {Kwong, Man Kam},
  title = {Uniqueness of positive solutions of {$\Delta u-u+u^p=0$} in {$\mathbb R^n$}},
  journal = {Archive for Rational Mechanics and Analysis},
  volume = {105},
  number = {3},
  pages = {243--266},
  year = {1989},
  doi = {10.1007/BF00251502}
}

@article{KenigMerle2006,
  author = {Kenig, Carlos E. and Merle, Frank},
  title = {Global well-posedness, scattering and blow-up for the
           energy-critical, focusing, non-linear {S}chr\"odinger equation in
           the radial case},
  journal = {Inventiones Mathematicae},
  volume = {166},
  number = {3},
  pages = {645--675},
  year = {2006}
}

@article{HolmerRoudenko2008,
  author = {Holmer, Justin and Roudenko, Svetlana},
  title = {A sharp condition for scattering of the radial 3{D} cubic nonlinear
           {S}chr\"odinger equation},
  journal = {Communications in Mathematical Physics},
  volume = {282},
  number = {2},
  pages = {435--467},
  year = {2008},
  doi = {10.1007/s00220-008-0529-y}
}

@article{AkahoriNawa2013,
  author = {Akahori, Takafumi and Nawa, Hayato},
  title = {Blowup and scattering problems for the nonlinear {S}chr\"odinger
           equations},
  journal = {Kyoto Journal of Mathematics},
  volume = {53},
  number = {3},
  pages = {629--672},
  year = {2013}
}

@article{DodsonMurphy2017,
  author = {Dodson, Benjamin and Murphy, Jason},
  title = {A new proof of scattering below the ground state for the 3d radial
           focusing cubic {NLS}},
  journal = {Proceedings of the American Mathematical Society},
  volume = {145},
  number = {11},
  pages = {4859--4867},
  year = {2017},
  eprint = {1611.04195},
  archivePrefix = {arXiv},
  primaryClass = {math.AP}
}

@article{AroraDodsonMurphy2020,
  author = {Arora, Anudeep Kumar and Dodson, Benjamin and Murphy, Jason},
  title = {Scattering below the ground state for the 2d radial nonlinear
           {S}chr\"odinger equation},
  journal = {Proceedings of the American Mathematical Society},
  volume = {148},
  number = {4},
  pages = {1653--1663},
  year = {2020},
  eprint = {1906.00515},
  archivePrefix = {arXiv},
  primaryClass = {math.AP}
}

@article{Tao2004Radial,
  author = {Tao, Terence},
  title = {On the asymptotic behavior of large radial data for a focusing
           non-linear {S}chr\"odinger equation},
  journal = {Dynamics of Partial Differential Equations},
  volume = {1},
  number = {1},
  pages = {1--48},
  year = {2004},
  eprint = {math/0309428},
  archivePrefix = {arXiv},
  primaryClass = {math.AP}
}

@article{Tao2008Potential,
  author = {Tao, Terence},
  title = {A global compact attractor for high-dimensional defocusing
           non-linear {S}chr\"odinger equations with potential},
  journal = {Dynamics of Partial Differential Equations},
  volume = {5},
  number = {2},
  pages = {101--116},
  year = {2008},
  eprint = {0805.1544},
  archivePrefix = {arXiv},
  primaryClass = {math.AP}
}

@article{TaoSolitonsStable,
  author = {Tao, Terence},
  title = {Why are solitons stable?},
  journal = {Bulletin of the American Mathematical Society (N.S.)},
  volume = {46},
  number = {1},
  pages = {1--33},
  year = {2009},
  doi = {10.1090/S0273-0979-08-01228-7}
}

\end{document}